\documentclass[english]{article} 
\usepackage[utf8]{inputenc}
\usepackage[T1]{fontenc}
\usepackage{lmodern}
\usepackage[a4paper]{geometry}
\usepackage{enumitem}
\usepackage{verbatim}
\usepackage{tikz,pgfplots}
\usepackage{graphicx}
\usepackage{mathtools}
\usepackage{amsmath}
\usepackage{amsthm}
\usepackage{amsfonts}
\usepackage{amssymb}
\usepackage{amsmath}
\usepackage{dsfont}
\usepackage{mathrsfs}
\usepackage{stmaryrd}
\usepackage{bbm}
\usepackage{framed}
\usepackage{xcolor}
\usepackage[english]{babel}
\usepackage[citecolor=blue,colorlinks=true]{hyperref}
\hypersetup{pdfnewwindow}
\usepackage{amsthm}
\newtheorem{thm}{Theorem}[section]
\newtheorem{definition}[thm]{Definition}
\newtheorem{proposition}[thm]{Proposition}
\newtheorem{lem}[thm]{Lemma}
\newtheorem{pro}[thm]{Proposition}
\newtheorem{corollary}[thm]{Corollary}

\theoremstyle{remark}
\newtheorem{remark}{Remark}[section]

\newcommand{\ud}{\mathrm{d}}

\newcommand{\half}{{\textstyle{1\over2}}}
\newcommand{\third}{{\textstyle{1\over3}}}

\newcommand{\fourth}{{\textstyle{1\over4}}}

\newcommand{\R}{\mathds{R}}

\newcommand{\T}{\mathds{T}}
\newcommand{\Pro}{\mathds{P}}
\newcommand{\E}{\mathds{E}}
\newcommand{\Z}{\mathds{Z}}
\newcommand{\N}{\mathds{N}}

\usepackage{enumitem}
\newlist{steps}{enumerate}{1}
\setlist[steps, 1]{label = Step \arabic*:}
\newcommand{\eqdef}{\stackrel{\text{\tiny{def}}}{=}}

\numberwithin{equation}{section}
\pgfplotsset{compat=1.14}

\newcommand{\eps}{\varepsilon}
\newcommand{\dual}[2]{\langle #1, #2\rangle}

\newcommand{\footremember}[2]{%
    \footnote{#2}
    \newcounter{#1}
    \setcounter{#1}{\value{footnote}}%
}

\begin{document}


\title{\bf Global dissipative martingale solutions to the variational wave equation with stochastic forcing
}

\author{Billel Guelmame\footremember{alley}{UMPA, CNRS, ENS de Lyon, Université de Lyon, billel.guelmame@ens-lyon.fr}
  and Julien Vovelle\footremember{trailer}{UMPA, CNRS, ENS de Lyon, julien.vovelle@ens-lyon.fr}
  }


\newcommand{\nfont}{\fontshape{n}\selectfont}

%
%
%

%
%
%
%
%
%


\maketitle

\begin{abstract} We consider the variational wave equation in one-dimensional space with stochastic forcing by an additive noise. Blow-up of local smooth solutions is established, and global existence is proved in the class of weak martingale solutions.
\end{abstract} 

\medskip

 {\bf AMS Classification :} 35R60, 60H15, 35L70, 35A01
\medskip

{\bf Key words :} Stochastic PDEs, Hyperbolic equations, Wave equation, Blow-up, Young measure, Skorokhod--Jakubowski's representation theorem.

\tableofcontents


\section{Introduction}

We consider in this paper the variational wave equation with a stochastic forcing 
\begin{gather}\label{SVW0}
\ud u_t\ -\ c(u) \left( c(u)\, u_x \right)_x \ud t\ =\ \Phi\, \ud W, \quad (t,x) \in (0,T) \times \T,
\end{gather}
where $u_t$ denotes $\partial u/\partial t$ (\textit{cf.} Remark~\ref{rk:notations-intro} below), $\T = \R/\Z$ denotes the one-dimensional torus, and, over the  filtered probability space $(\Omega, \mathcal{F}, \Pro, (\mathcal{F}_t)_{t\geqslant 0})$, $W$ is a cylindrical Wiener process (see \eqref{defW}). 
The equation \eqref{SVW0} is the Euler--Lagrange equation associated to the Lagrangian 
\begin{equation*}
\mathscr{L}\ \eqdef\ \iint \left[ \half \left[ u_t^2\, -\, c(u)^2\, u_x^2 \right] \ud t\, +\, u\, \Phi\, \ud W(t) \right] \ud x.
\end{equation*}
Our aim is to study the well-posedness of \eqref{SVW0} up to an explosion time. We prove then that regular solutions may blow-up in finite time. Finally we establish the existence of global-in-time weak martingale solutions to \eqref{SVW0}. 

The deterministic variational wave equation is recovered taking $\Phi \equiv 0$. It appears in several physical contexts. For example, nematic liquid crystals \cite{Saxton89,HunterSaxton1991,GlasseyHunterZheng1997}, long waves on a dipole chain \cite{GlasseyHunterZheng1997,GI92,ZI92} and also in classical field theories and general relativity \cite{GlasseyHunterZheng1997}.
An asymptotic (and simpler) equation has been derived by Hunter and Saxton \cite{HunterSaxton1991}
\begin{equation}\label{HS}
\left[u_t\, +\, u\, u_x \right]_x\, =\ \half\, u_x^2.
\end{equation}
Both the Hunter--Saxton equation \eqref{HS} and the deterministic variational wave equation have been widely studied in the literature \cite{HunterZheng95a,HunterZheng95b,ZhangZheng1998,
ZhangZheng2000,ZhangZheng2001,ZhangZheng2003,
ZhangZheng2005,BressanConstantin05,BressanZheng2006,BressanZhangZheng07,
Dafermos,BressanChenZhang2015}.

Considering the deterministic variational wave equation on the real line ($x \in \R$), the local (in time) well-posedness can be obtained using Kato's theorem for the quasi-linear equations \cite{Kato}. Singularities may appear in finite time in the non-linear case $c'(\cdot) \not\equiv 0$ \cite{GlasseyHunterZheng1996}. 
However, if the Riemann invariants $u_t \pm c(u) u_x$ are non-positive initially, then rarefactive solutions exist globally in time \cite{ZhangZheng1998}.
Global weak solutions to the variational wave equation are not unique, indeed, at least two types of solutions exist, dissipative and conservative ones.
The conservative solutions were obtained in \cite{BressanZheng2006} using an equivalent system in the Lagrangian coordinates.
The total energy of the conservative solutions is equal to the initial energy for almost all $t>0$.
The uniqueness of the conservative solutions is established in \cite{BressanChenZhang2015}.
Dissipative solutions are obtained in \cite{ZhangZheng2003,ZhangZheng2005} by studying an approximated system and then passing to the limit.
The dissipative solutions satisfy an energy inequality and a one-sided entropy inequality.
To the author's knowledge, the uniqueness of the dissipative solutions remains an open problem. 
However, for the simpler Hunter--Saxton equation \eqref{HS}, the dissipative solutions are unique \cite{Dafermos}.

In the stochastic case, the Hunter--Saxton equation with noise has been studied in \cite{HoldenKarlsenPang2021}. In a recent work \cite{Pang24}, Pang studied the viscous variational wave equation with transport noise and established its global well-posedness. 

We consider the problem \eqref{SVW0} under study in the class of systems of stochastic non-linear hyperbolic equations. There are numerous works related to stochastic scalar conservation laws, see for instance \cite{ChenPang2021} and references therein.  In hyperbolic systems of conservation laws with stochastic forcing, which is our framework, the problematic and techniques are relatively different from the scalar case. 

Let us first make some comments on our results and the way they are established. The problem \eqref{SVW0} admits an equivalent formulation as a $2\times 2$ system (see \eqref{SVWE2} below). The existence of local regular solutions to this $2\times 2$ system is essentially a pathwise variation on the deterministic result and certainly classical in spirit, but we prefer to give the proof in full details. Global martingale solutions are obtained by the probabilistic compactness method. We adapt the method of P. Zhang and Y. Zheng, \cite{ZhangZheng2001,ZhangZheng2003,ZhangZheng2005} (see also \cite{Guelmame2023} in the context of the Green--Naghdi equations with surface tension) for the deterministic treatment of the problem.
If the stochastic aspects are not present in \cite{ZhangZheng2003,ZhangZheng2005}, our framework is also different, insofar as we consider periodic solutions and not solutions with certain localization properties. As a consequence, our approximated system involves some additional ``correction terms'' (see \eqref{SVWEep} below) to ensure the periodic character of the the solutions.
One of the central technical tool to recover the desired equation at the limit are Young measures, as introduced by DiPerna in the context of systems of conservation laws, \cite{Diperna83a}. Similar techniques have been used recently for the Hunter--Saxton equation \eqref{HS}, the Camassa--Holm equation and the Degasperis--Procesi equation with noise  \cite{CC24,HoldenKarlsenPang2021,HoldenKarlsenPang2023,GHKP22}. In addition, we exploit \cite{BerthelinVovelle19} (see also \cite{FengNualart08}), where Young measures for the study of the stochastic isentropic Euler equations play a central role too.

%
%

The paper is organized as follows. In Section \ref{sec:mainresults} we introduce the kind of noise (white in time, coloured in space) that we consider. We present the equivalent $2\times 2$ system \eqref{SVWE2} and we state the main results of the paper. 
In Section \ref{sec:localsol}, devoted to local-in-time solutions, we prove the existence of regular solutions up to an explosion time, and also give some criterion for blow-up (which amounts to the explosion of the Lipschitz norm of $u$) in finite time. 
Section \ref{sec:gloabsol} is devoted to the approximated system obtained by truncation and correction of \eqref{SVWE2} (see \eqref{SVWEep}). We prove that the approximate solutions exist globally in time, and that they satisfy some uniform estimates. In Section \ref{sec:Tightness} we prove that the sequence of approximated solutions is tight (in a suitable space), and we use the Skorokhod--Jakubowski representation theorem to pass to the limit and to derive the limit equation with some defect measures. 
Finally, it is proved in Section \ref{sec:Young} that the defect measures are trivial and that the limit is a global weak martingale solution to \eqref{SVW0}.

\section{The equations and main results}\label{sec:mainresults}

\subsection{The stochastic variational wave equation}

Let $\mathfrak{U}$ be a Hilbert space with an orthonormal basis $(g_k)_{k \geqslant 1}$ and let $\mathfrak{U}_{-1}$ be another Hilbert space such that the injection $\mathfrak{U}\hookrightarrow\mathfrak{U}_{-1}$ is Hilbert--Schmidt. Let $W$ be the cylindrical Wiener process defined by
\begin{equation}\label{defW}
W(t)\ \eqdef\ \sum_{k\geqslant 1} g_k\, \beta_k(t),\, t\geqslant 0,
\end{equation}
where $(\beta_1(t),\beta_2(t),\dotsc)$ are independent one-dimensional Wiener processes,  see Section~4.1.2 in \cite{DaPratoZabczyk14}.
Let $\Phi : \mathfrak{U} \to L^2(\T)$ such that for any $k\geqslant 1$ we have $\sigma_k \eqdef \Phi g_k \in C(\T)$ and 
\begin{equation}\label{defq}
q_0\ \eqdef\ \sum_k \|\sigma_k\|_{W^{1,\infty}(\T)}^2\ <\ \infty, \qquad q(x)\ \eqdef\ \sum_k \sigma_k(x)^2.
\end{equation}
By \eqref{defq} and the injection $L^\infty(\T)\hookrightarrow L^2(\T)$, the map $\Phi$ is Hilbert--Schmidt. Let us assume that $c\in C^\infty(\R)$ satisfies
\begin{gather}\label{coeff-c}
0\ <\ c_1\ \leqslant\ c(u)\ \leqslant\ c_2,\\ \label{coeff-c'}  0\ \leqslant\ c'(u)\ \leqslant\ c_3 .
\end{gather}
for some constants $c_1,c_2,c_3\in(0,\infty)$. We consider the stochastic variational wave equation with additive noise in $(0,T)\times\T$  
\begin{subequations}\label{SVWE1}
\begin{gather}
\ud u_t\ -\ c(u) \left( c(u)\, u_x \right)_x \ud t =\ \Phi\, \ud W,  \\
u(0,\cdot)\ =\ u_0, \qquad u_t(t=0,\cdot)\ =\ v_0.
\end{gather}
\end{subequations}
We can also consider the equivalent form
\begin{subequations}\label{SVWE2}
\begin{gather}\label{Req}
\ud R\ +\ c(u)\, R_x\, \ud t\ =\ \tilde{c}'(u) \left[R^2\, -\, S^2 \right] \ud t\ +\ \Phi\,  \ud W, \\ \label{Seq}
\ud S\ -\ c(u)\, S_x\, \ud t\ =\ \tilde{c}'(u) \left[S^2\, -\, R^2 \right] \ud t\ +\ \Phi\,  \ud W, 
\end{gather}
\end{subequations}
completed with the equation
\begin{equation}\label{udef}
u(t,x)\ =\ \mathcal{C}^{-1}\left\{ \mathcal{C} \left\{ \int_0^t {\textstyle \left(\frac{R\, +\, S}{2}  \right) (s,0)}\, \ud s\ +\ u_0(0) \right\} +\, \int_0^x {\textstyle \frac{S\, -\, R}{2}}(t,y)\, \ud y \right\},
\end{equation}
where
\begin{equation}\label{Cronde}
\mathcal{C}(r)\ =\ \int_0^r c(\sigma)\, \ud \sigma,
\end{equation}
which expresses $u$ as a non-local function of $(R,S)$. The system \eqref{SVWE2} is deduced from \eqref{SVWE1} by setting
\begin{equation}\label{defRS}
R\ \eqdef\ u_t\ -\ c(u)\, u_x, \qquad S\ \eqdef\ u_t\ +\ c(u)\, u_x, \qquad \tilde{c}(u)\ \eqdef\ {\textstyle \frac{1}{4}}\, \ln c(u).
\end{equation}
The corresponding initial conditions for \eqref{SVWE2} are therefore
\begin{equation}\label{IC}
R(0,\cdot)\ =\ R_0\ \eqdef\ v_0\, -\, c(u_0)\, u_0', \quad S(0,\cdot)\ =\ S_0\ \eqdef\ v_0\, +\, c(u_0)\, u_0'.
\end{equation}
\begin{remark}[Notations]\label{rk:notations-intro} The subscript $t$, as in $u_t$, always denote the partial derivative with respect to $t$, and never the value at the given time $t$ of a stochastic process $X$ (the latter being simply denoted by $X(t)$).
\end{remark}

We will  establish in Theorem~\ref{thm:loc-existRS} the existence of local regular solutions to \eqref{SVWE2}, which are strong in the probabilistic sense. It may happen that some of these solutions have a finite time of existence, \textit{cf.} Theorem~\ref{th:blowup0}. Nevertheless, global-in-time solutions, which are weak in the probabilistic sense (martingale solutions), and defined in an $L^2$-framework, are proved to exist, see Theorem~\ref{thm:global-existR2SS}. These global-in-time solutions are obtained as limits of solutions to the following system

\begin{subequations}\label{SVWEep}
\begin{gather}\label{Reqep}
\ud R^\varepsilon\ +\ c(u^\varepsilon)\, R^\varepsilon_x\, \ud t\ =\ \tilde{c}'(u^\varepsilon) \left[(R^\varepsilon)^2\, -\, (S^\varepsilon)^2\, -\, \chi_\varepsilon(R^\varepsilon)\, +\, 2\, R^\varepsilon\, \Theta^\varepsilon \right] \ud t\ +\ \Phi^\varepsilon\,  \ud W, \\ \label{Seqep}
\ud S^\varepsilon\ -\ c(u^\varepsilon)\, S^\varepsilon_x\, \ud t\ =\ \tilde{c}'(u^\varepsilon) \left[ (S^\varepsilon)^2\, -\, (R^\varepsilon)^2\, -\, \chi_\varepsilon(S^\varepsilon)\, -\, 2\, S^\varepsilon\, \Theta^\varepsilon \right] \ud t\ +\ \Phi^\varepsilon\,  \ud W,\\
R^\varepsilon(0,\cdot)\ =\ R^\varepsilon_0\ \eqdef\  J_\varepsilon R_0, \quad \qquad S^\varepsilon(0,\cdot)\ =\ S^\varepsilon_0\ \eqdef\   J_\varepsilon S_0,
\end{gather}
\end{subequations}
coupled with the equation ($x\in [0,1]$)
\begin{equation}\label{udefep}
u^\varepsilon(t,x)\ =\ \mathcal{C}^{-1}\left\{ \mathcal{C} \left\{ \int_0^t {\textstyle \left(\frac{R^\varepsilon\, +\, S^\varepsilon}{2}  \right) (s,0)}\, \ud s\ +\ u^\varepsilon_0(0) \right\} +\, \int_0^x \left[{\textstyle \frac{S^\varepsilon\, -\, R^\varepsilon}{2}}(t,y)\ -\ \Theta^\eps(t)\right] \ud y\right\},
\end{equation}
 where $J_\varepsilon$ is a Friedrichs mollifier, defined as the convolution operator $R\mapsto R\ast\rho_\varepsilon$, where $(\rho_\eps)$ is an approximation of the unit.
In \eqref{SVWEep}, the cut-off function $\chi_\varepsilon$ is defined by 
\begin{equation}\label{chidef}
\chi_\varepsilon (\xi)\ \eqdef\, \left(\xi\ -\ \frac{1}{\varepsilon} \right)^2 \mathds{1}_{[\frac{1}{\varepsilon}, \infty)} (\xi)\ =\ 
\begin{cases}
\left(\xi\ -\ \frac{1}{\varepsilon} \right)^2, & \xi \geqslant 1/\varepsilon, \\
0, & \xi < 1/\varepsilon.
\end{cases}
\end{equation}
We have also introduced a ``correction term'' 
\begin{equation}\label{psieps}
\Theta^\eps(t)\ \eqdef\  \int_0^1\frac{S^\varepsilon-R^\varepsilon}{2}(t,y)\, \ud y.
\end{equation}
Note that this correction $\Theta^\eps(t)$ is not necessary in the case of a problem set on the whole line $\R$, \cite{ZhangZheng2003,ZhangZheng2005}.
Finally, we define $\Phi^\varepsilon : \mathfrak{U} \to \cap_{s \geqslant 0} H^s(\T)$ such that for any $k\geqslant 1$ and $\varepsilon>0$ we have $ \Phi^\varepsilon g_k \eqdef \sigma_k^\varepsilon \eqdef J_\varepsilon\sigma_k \in C^\infty(\T)$.
Clearly, we have the domination $|\sigma_k^\varepsilon| \leqslant |\sigma_k|$, and thus (see \eqref{defq})
\begin{equation}\label{qeps-q}
\sum_k \|\sigma_k^\varepsilon\|_{C(\T)}^2\ \leqslant\ \sum_k \|\sigma_k\|_{C(\T)}^2\ \leqslant\ q_0, \qquad q^\varepsilon(x)\ \eqdef\ \sum_k \sigma_k(x)^2.
\end{equation}

%
%

\subsection{Local and global existence}

In what follows, if $E$ is a Banach space, the space $C([0,T];E)$ of continuous functions $[0,T]\to E$ is endowed with the norm 
\[
\|u\|_{C([0,T];E)}=\sup_{t\in[0,T]}\|u(t)\|_E.
\]
The local existence theorem~\ref{thm:loc-existRS} below will be proved actually in a slightly more general version, Theorem~\ref{th:loc-exist}. To solve \eqref{SVWE1} in the class $H^{s+1} \times H^s$, we need the noise to be a.s. $H^{s+1}$, which is ensured by the following hypothesis:
\begin{equation}\label{sigmas1}
\sum_{k \geqslant 1}  \|\sigma_k\|_{H^{s+1}(\T)}^2\ <\ \infty.
\end{equation}
Indeed, a martingale inequality and \eqref{defW} give
\begin{equation}\label{WHs}
\E\, \| \Phi W\|_{C([0,T], H^{s+1}(\T))}^2\ \leqslant\ C\, T \sum_{k \geqslant 1}  \|\sigma_k\|_{H^{s+1}(\T)}^2\ <\ \infty.
\end{equation}
Then \eqref{WHs} (considered for $T=1,2,\ldots$) implies that there is a set $\Omega_s$ of probability one such that
\begin{equation}\label{WHsas}
\omega\in\Omega_s\ \Rightarrow\ \forall T>0,\, \Phi W \in C([0,T], H^{s+1}(\T)).
\end{equation}

\begin{definition}[Regular solution to \eqref{SVWE1}, up to an explosion time] Let $(u_0,v_0)\in H^{s+1}(\T) \times H^s (\T)$. Let $\tau>0$ be a stopping time such that $\tau>0$ a.s. Let $\Omega_s$ be defined in \eqref{WHsas}. A process $(u(t))_{0\leqslant t<\tau}$ such that: a.s., 
\begin{equation}
u \in C([0,\tau), H^{s+1}(\T))\cap C^1([0,\tau), H^{s}(\T)),
\end{equation}
is said to be a regular (or $H^{s+1}$) solution to \eqref{SVWE1} up to the explosion time $\tau$ if, 
\begin{enumerate}
\item $u(0,\cdot) = u_0,\ \Pro$-almost surely,
\item for any stopping time $\tau'$ such that $\tau'<\tau$ a.s., for all $x\in\T$, the process $(u(t\wedge\tau' ,x))_{t\geqslant 0}$ is predictable and, for all $\omega\in\Omega_s$ such that $\tau'(\omega)<\tau(\omega)$ (and thus $\Pro$-a.s.), we have the following identity in $H^{s-1}(\T)$
\begin{equation}
u_t(\tau')\ =\ v_0\ + \int_0^{\tau'} c(u) \left( c(u)\, u_x \right)_x \ud s\ +\ \Phi\,  W(\tau'),
\end{equation}
\item the stopping time $\tau$ is an explosion time, in the sense that
\begin{equation}\label{tauntau}
\tau\ =\ \sup_{n\geqslant 1}\tau_n,\qquad\tau_n\ \eqdef\ \inf\left\{t\in[0,\tau)\, ;\, \|(u_t(t),u_x(t) )\|_{L^\infty(\T)}\ >\ n\right\}.
\end{equation}
\end{enumerate}
If $\tau=\infty$ a.s., the solution is said to be global.
\end{definition}
In \eqref{tauntau}, we use the following convention.
\begin{remark}[Convention for the infimum]\label{rk:convention-inf} If $T\geqslant 0$ and $(P_t)_{t\geqslant 0}$ are some properties depending on time, then the value of the infimum in the expressions
\begin{equation}\label{inf}
\inf\left\{t\in[0,T); P_t\mbox{ is true}\right\},\mbox{ or }\inf\left\{t\in[0,T]; P_t\mbox{ is true}\right\}
\end{equation}
is set to the terminal value $T$ if, respectively,
\begin{equation}\label{inffalse}
\forall t\in[0,T), P_t\mbox{ is false, or }\forall t\in[0,T], P_t\mbox{ is false.}
\end{equation}
\end{remark}

\begin{thm}[Local existence of regular solutions]\label{thm:loc-existRS} Let $s>3/2$. Let $c \in C^\infty(\R)$ satisfies \eqref{coeff-c} and $|c'(u)| \leqslant c_3$. Let $\Phi$ satisfy \eqref{sigmas1} and let $(u_0,v_0)\in H^{s+1} \times H^s $. Then \eqref{SVWE1} admits a solution up to an explosion time $\tau^*$ and, for $i=1,2$, two solutions $(u_i,\tau_i^*)$ to \eqref{SVWE1} defined up to an explosion time $\tau_i^*$ coincide, in the sense that $\tau_1^*=\tau_2^*$ and $u_1=u_2$ on $[0,\tau_1^*)$.
\end{thm}

\begin{remark}
The local existence of regular solutions will be established for the equivalent system \eqref{SVWE2}.
\end{remark}

We then observe that singularity's formation can indeed occur in finite time.

\begin{thm}[Blow-up in finite time]\label{th:blowup0} Let $s>3/2$. Assume that the conditions of Theorem \ref{thm:loc-existRS} are satisfied, and that $c'(u^\star) \neq 0$ for some $u^\star \in \R$. For any $\gamma>1/3$ and  $\eps>0$, there is an initial datum $u_0^\eps\in C^\infty(\T)$ such that the regular solution $(u,\tau^*)$ to \eqref{SVWE1} defined up to the explosion time $\tau^*$, with initial data $(u_0^\eps,0)$, satisfies  
\begin{equation}\label{PBlowup0}
\Pro(\tau^*\leqslant \eps^{\gamma})\geqslant 1-\eps.
\end{equation}
\end{thm}

Eventually, we can state the existence of global-in-time weak solutions.

\begin{definition}[Weak martingale solution]\label{def:WeakSol} Assume \eqref{defq}, \eqref{coeff-c} and \eqref{coeff-c'} . Let $u_0 \in H^1(\T)$ and $v_0 \in L^2(\T)$. We say that the problem \eqref{SVWE1} admits a weak martingale solution if there exists first a stochastic basis 
\begin{equation}\label{StochasticBasis}
\left(\tilde{\Omega},\tilde{\mathcal{F}},\tilde{\Pro},\left(\tilde{\mathcal{F}}_t\right),\left(\tilde{W}(t)\right)\right),
\end{equation}
where $\left(\tilde{W}(t)\right)$ is a cylindrical Wiener process on $\mathfrak{U}$, and, second, an adapted stochastic process $(u(t))$ with values in $H^1(\T)$ such that, for all $T>0$,
\begin{enumerate}
\item $u_t,u_x\in L^2(\tilde{\Omega};L^\infty([0,T]; L^2(\T)))$ and $u_t,u_x\in C([0,T]; L^2(\T)-\mathrm{weak}))$ $\tilde{\Pro}$-a.s., 
\item $\tilde{\Pro}$-a.s., $u(0,\cdot)=u_0$ and for all $\varphi\in C^1(\T)$ and $t \in [0,T]$,
\begin{multline}\label{weaku}
\int_\T u_t(t)\, \varphi\, \ud x\ -\ \int_\T v_0\, \varphi\, \ud x
\ +\ \int_0^t \int_\T \left( c(u(s))\, \varphi \right)_x c(u(s))\, u_x\, \ud x\, \ud s\ =\ \int_\T \varphi\, \Phi\, \tilde{W}(t)\, \ud x 
\end{multline}
\item $\tilde{\Pro}$-a.s, the solution $u$ satisfies the energy ``dissipation'' inequality
\begin{align}\nonumber
\int_\T \left[u_t^2\, +\, c(u)^2\, u_x^2 \right] (t_2)\, \ud x\ &\leqslant\ \int_\T \left[u_t^2\, +\, c(u)^2\, u_x^2 \right] (t_1)\, \ud x\ +\ \|q\|_{L^1(\T)}\, (t_2-t_1)\\ \label{energydissipation}
&\quad +\ 2 \int_{t_1}^{t2}\int_\T u_t\, \Phi\, \ud x\, \ud \tilde{W}(s),
\end{align}
for almost all $t_1\in[0,\infty)$ and any $t_2 \geqslant t_1$, where $q$ is the variance of the noise defined in \eqref{defq}.
\item $\tilde{\Pro}$-almost surely, for almost all $t_0 \in [0,\infty)$ we have 
\begin{equation}\label{rightcontinuity}
\lim_{t \downarrow t_0} \left\| \left( u_t(t)-u_t(t_0), u_x(t)-u_x(t_0) \right) \right\|_{L^2}\ =\ 0.
\end{equation} 
\end{enumerate}
\end{definition}

\begin{remark}
The right-continuity condition \eqref{rightcontinuity} can be interpreted as a dissipation condition. Indeed, Dafermos \cite{Dafermos} proved that, in the case of the Hunter--Saxton equation \eqref{HS}, the right-continuity condition is equivalent to the dissipation of the energy. This plays a crucial role in the uniqueness of solutions \cite{Dafermos}.
\end{remark}

\begin{thm}[Global existence of weak martingale solutions]\label{thm:global-existR2SS} Let $u_0 \in H^1(\T)$ and $v_0 \in L^2(\T)$. Assume \eqref{defq}, \eqref{coeff-c}, \eqref{coeff-c'} and suppose that
\begin{equation}\label{Positivecprime00}
c_0\ \eqdef\ \inf_{x \in \T} c'(u_0(x))\ >\ 0.
\end{equation}
Then the problem \eqref{SVWE1} admits a weak martingale solution. Moreover, the solution satisfies 
\begin{itemize}
\item for all $p \in [1,3)$ we have 
\begin{equation}\label{L3-estimates}
\tilde{\E} \int_{[0,t] \times \T} c'(u) \left[|u_t|^p\, +\, |u_x|^p \right] \ud x\, \ud t\ \leqslant\ C(T,\alpha),
\end{equation} 
\item for all $p \in [1,2]$, there exists $C(p)>0$ such that for all $t \in (0,T]$, we have the entropy inequality 
\begin{equation}\label{Entropy}
\tilde{\E} \left\| \left[ u_t\, \pm\, c(u)\, u_x \right]^- \right\|_{L^\infty}^p\! (t)\ \leqslant\ C(p,T) \left( 1\, +\, t^{-p} \right).
\end{equation}
\end{itemize}
\end{thm}

%

\section{Local-in-time regular solutions}\label{sec:localsol}

In this section we will consider the question of existence of local-in-time solutions in $H^s$ to the system \eqref{SVWEeptheta} below, which covers both systems \eqref{SVWE2} and \eqref{SVWEep}.

\subsection{Generalized system}

For some given function $\theta\in C^\infty(\R)$, we consider the system 
\begin{subequations}\label{SVWEeptheta}
\begin{gather}\label{Reqeptheta}
\ud R\ +\ c(u)\, R_x\, \ud t\ =\ \tilde{c}'(u) \left[R^2\, -\, S^2\, -\, \theta(R)\, +2\, R\, \Theta \right] \ud t\ +\ \Phi\, \ud W, \\ \label{Seqeptheta}
\ud S\ -\ c(u)\, S_x\, \ud t\ =\ \tilde{c}'(u) \left[ S^2\, -\, R^2\, -\, \theta(S)\, -2\, S\, \Theta  \right] \ud t\ +\ \Phi\, \ud W,\\
R(0,\cdot)\ =\ R_0, \quad \qquad S(0,\cdot)\ =\ S_0,
\end{gather}
\end{subequations}
where $u$ is defined by
\begin{equation}\label{udefeptheta}
u(t,x)\ =\ \mathcal{C}^{-1}\left\{ \mathcal{C} \left\{ \int_0^t {\textstyle \left(\frac{R\, +\, S}{2}  \right) (s,0)}\, \ud s\ +\ u_0(0) \right\} +\, \int_0^x \left[{\textstyle \frac{S\, -\, R}{2}}(t,y)\, -\, \Theta (t)\right] \ud y \right\},
\end{equation}
with
\begin{equation}\label{psieptheta}
\Theta (t)\ \eqdef\ \int_0^1 \frac{S-R}{2}(t,y)\ \ud y.
\end{equation}
The original system \eqref{SVWE2} simply corresponds to the case $\theta\equiv 0$ (one can prove that $\Theta (t)\equiv 0$ in that case\footnote{indeed, $\Theta $ satisfies $\Theta '=\Theta  \int_\T  \tilde{c}'(u) (R+S)\, \ud x$ and $\Theta (0)=0$.}), while the approximate system \eqref{SVWEep} corresponds to the choice $\theta=\chi_\eps$ (note that $\chi_\eps$ is not of class $C^\infty$ though, but in the class $W^{2,\infty}_\mathrm{loc}$, this is why Theorem~\ref{thm:globalsolep} will be stated in the space $H^2(\T)$).
In essential, the system \eqref{SVWEeptheta} will be solved pathwise, with the help of the following set of unknown functions:
\begin{equation}\label{introPQ}
P\ \eqdef\ R\ -\ \Phi\, W, \qquad Q\ \eqdef\ S\ -\ \Phi\, W.
\end{equation}
In terms of the unknown $(P,Q)$, the system \eqref{SVWEeptheta} can be rewritten as the following system of PDEs with random coefficients:
\begin{subequations}\label{SVWEepthetaPQ}
\begin{gather}\label{Peq}
P_t\ +\ c(u)\, P_x\ =\ \tilde{c}'(u) \left[P^2\, -\, Q^2 \, -\, \theta(P\, +\, \Phi\, W)\, +\, 2\, (P\, +\, \Phi\, W)\, \Theta \right]\, -\, \left( c(u)\, \Phi\, W \right)_x, \\
Q_t\ -\ c(u)\, Q_x\ =\ \tilde{c}'(u) \left[Q^2\, -\, P^2 \, -\, \theta(Q\, +\, \Phi\, W)\, -\, 2\, (Q\, +\, \Phi\, W)\, \Theta \right]\, +\,  \left( c(u)\, \Phi\, W \right)_x, \\ 
P(0,\cdot)\ =\ R_0, \qquad Q(0,\cdot)\ =\ S_0,
\end{gather}
\end{subequations}
where
\begin{equation}\label{udefepthetaPQ}
u(t,x)\, =\, \mathcal{C}^{-1}\left\{ \mathcal{C} \left\{ \int_0^t {\textstyle \left(\frac{P\, +\, Q}{2}\, +\, \Phi\, W  \right) (s,0)}\, \ud s\ +\ u_0(0) \right\} +\, \int_0^x \left[{\textstyle \frac{Q\, -\, P}{2}}(t,y)\, -\, \Theta (t)\right] \ud y \right\},
\end{equation}
with
\begin{equation}\label{psiepthetaPQ}
\Theta (t)\ \eqdef\ \int_0^1 \frac{Q-P}{2}(t,y)\, \ud y.
\end{equation}
The system \eqref{SVWEepthetaPQ}-\eqref{udefepthetaPQ} has the form of the quasilinear system 
\begin{subequations}\label{SVWE_Qlinear-eps}
\begin{align}
V_t\ +\ A(V)\, V_x\ &=\ F_{\theta}(V), \mbox{ in }(0,T) \times \T,\\
V(0,\cdot)\ &=\ V_0, \mbox{ in }\T.
\end{align}
\end{subequations}
where 
\begin{equation}
V\ \eqdef\ 
\begin{pmatrix}
P \\
Q
\end{pmatrix},
\end{equation}
and, $(u,\Theta )$ being given as a functions of $V$ by \eqref{udefepthetaPQ}-\eqref{psiepthetaPQ}, we set
\begin{equation}\label{defAeps}
A(V)\ \eqdef 
\begin{pmatrix}
c(u) & 0\\
0 & -c(u)
\end{pmatrix}, 
\end{equation}
and
\begin{equation}\label{defFeps}
F_{\theta}(V)\ \eqdef 
\begin{pmatrix}
\tilde{c}'(u) \left[P^2\, -\, Q^2 \, -\, \theta(P\, +\, \Phi\, W)\,  +\, 2\, (P\, +\, \Phi\, W)\, \Theta \right]\, -\, \left( c(u)\, \Phi\, W \right)_x\\
\tilde{c}'(u) \left[Q^2\, -\, P^2 \, -\, \theta(Q\, +\, \Phi\, W)\, -\, 2\, (Q\, +\, \Phi\, W)\, \Theta \right]\, +\,  \left( c(u)\, \Phi\, W \right)_x
\end{pmatrix}.
\end{equation}

\subsection{Local regular solutions}\label{subsec:localregsol}

Let $s>3/2$. We introduce the following definition.

\begin{definition}[Regular solution to \eqref{SVWE_Qlinear-eps}, up to an explosion time]\label{def:regsolPQeps} Let $V_0\in H^s(\T;\R^2)$. Let $\tau>0$ be a stopping time such that $\tau>0$ a.s. Let $\Omega_s$ be defined in \eqref{WHsas}. A process $(V(t))_{0\leqslant t<\tau}\ =\ (P(t),Q(t))_{0\leqslant t<\tau}$ such that, we have: a.s., 
\begin{equation}\label{def-sol-V}
\left(t\mapsto V(t)\right)\in C([0,\tau), H^{s}(\T;\R^2))\cap C^1([0,\tau), H^{s-1}(\T;\R^2)),
\end{equation}
is said to be a regular (or $H^s$) solution to \eqref{SVWE_Qlinear-eps} up to the explosion time $\tau$ if, 
\begin{enumerate}
\item for any stopping time $\tau'$ such that $\tau'<\tau$ a.s., for all $x\in\T$, the process $(V(t\wedge\tau' ,x))_{t\geqslant 0}$ is predictable and, for all $\omega\in\Omega_s$ such that $\tau'(\omega)<\tau(\omega)$ (and thus $\Pro$-a.s.), we have the following identity in $H^{s-1}(\T)$
\begin{equation}\label{eq-def-sol-V-eps}
V(\tau')\ =\ V_0\ -\ \int_0^{\tau'} (A(V)\,V_x\ -\ F_{\theta}(V))(s)\, \ud s,
\end{equation}
\item the stopping time $\tau$ is an explosion time, in the sense that
\begin{equation}\label{blowuptauVeps}
\tau\ =\ \sup_{n\geqslant 1}\tau_n,\qquad\tau_n\ \eqdef\ \inf\left\{t\in[0,\tau)\, ;\, \|V(t)\|_{L^\infty(\T;\R^2)}\ >\ n\right\}.
\end{equation}
\end{enumerate}
If $\tau=\infty$ a.s., the solution is said to be global.
\end{definition}

The integral identity \eqref{eq-def-sol-V-eps} can be also expressed by integration along the characteristic curves. We will see several instances of this formulation: to solve the linear system (with frozen non-linear coefficients) associated to \eqref{SVWE_Qlinear-eps}, and also to get some estimates on the solutions to \eqref{SVWE_Qlinear-eps}, with either $\theta=0$ or $\theta=\chi_\eps$, so we give a rather general form to the statement.

\begin{proposition}[Integration along characteristic curves]\label{prop:galchi} Let $\tau>0$ be a stopping time such that $\tau>0$ a.s. Let $(v(t),\bar{c}(t),f(t))_{0\leqslant t<\tau}$ be a stochastic process such that, we have: a.s., 
\begin{equation}\label{def-sol-vvv}
\left[t\mapsto v(t)\right]\in C([0,\tau), H^{s}(\T))\cap C^1([0,\tau),H^{s-1}(\T)),
\end{equation}
and
\begin{equation}\label{def-sol-cf}
\left[t\mapsto (\bar{c},f)(t)\right]\in C^1([0,\tau),H^{s-1}(\T))\times C([0,\tau),H^{s-1}(\T)).
\end{equation}
Assume also: a.s., for all $t\geqslant 0$, $x\in\T$, $\bar{c}(t,x)\in [c_1,c_2]$ (where $c_1,c_2$ are the constants in \eqref{coeff-c}). Let $(X(t,x))$ denote the flow associated to $\bar{c}$, and let $y\mapsto Y(t,y)$ denote the inverse of the map $x\mapsto X(t,x)$. Then, there is equivalence between:
\begin{enumerate}
\item for any stopping time $\tau'$ such that $\tau'<\tau$ a.s.,
\begin{equation}\label{eq-def-sol-vvv}
v(\tau')\ =\ v(0)\ -\ \int_0^{\tau'} (\bar{c}\,v_x\ -\ f)(s)\, \ud s,
\end{equation}
and 
\item a.s., for all $t\in[0,\tau)$ and all $x\in\T$,
\begin{equation}\label{vvv-char}
v(t,x)=v(0,Y(t,x))+\int_0^t f(s,X(s,Y(t,x)))\, \ud s.
\end{equation}
\end{enumerate}
\end{proposition}

\begin{proof}[Proof of Proposition~\ref{prop:galchi}] Note first that \eqref{eq-def-sol-vvv} is equivalent to: a.s., for all $t\in[0,\tau)$,
\begin{equation}\label{eq-def-sol-vvv-t}
v(t)\ =\ v(0)\ -\ \int_0^t (\bar{c}\,v_x\ -\ f)(s)\, \ud s.
\end{equation}
Indeed, \eqref{eq-def-sol-vvv} can be applied to the stopping time $t\wedge\tau'$, from which we deduce \eqref{eq-def-sol-vvv-t} for all $t\in[0,\tau'(\omega))$, and thus for all $t\in[0,\tau(\omega))$ since $\tau'<\tau$ is arbitrary. Conversely, \eqref{eq-def-sol-vvv-t} yields \eqref{eq-def-sol-vvv} when $\tau'$ is simple, and the general case follows by approximation. This being said, Proposition~\ref{prop:galchi} is essentially a classical deterministic statement. It is not difficult to check that all the terms are sufficiently regular to justify the equivalence through differentiation. In particular, the equation
\begin{equation}\label{vvv-PDE}
\partial_t v(t,x)+\bar{c}(t,x)\partial_x v(t,x)=f(t,x)
\end{equation}
is satisfied a.s., for all $t\in[0,\tau)$ and all $x\in\T$.
\end{proof}

Energy estimates in $H^s(\T)$ are one of the main ingredient in the resolution of \eqref{SVWE_Qlinear-eps} locally in time. A basic issue in the justification of such estimates for solutions a priori in $H^s(\T)$ is that the term $A(V)\partial_x$ involves an additional space derivative and necessitates the better $H^{s+1}$-regularity to be safely manipulated. There are several ways to circumvent this problem, for instance by working first on an approximate problem (by Friedrichs regularization of the term $A(V)\partial_x$, or by direct resolution of an iterative scheme in a class of sufficiently smooth solutions). We choose a different method, by a direct regularization of the equation and commutator estimates (based on Kato--Ponce's estimates \cite{KatoPonce1988} and DiPerna--Lions estimates \cite{DiPernaLions89}). We give the result on the general equation \eqref{vvv-PDE}. 
\begin{proposition}[$H^s$-estimate, linear equation]\label{prop:Hsestimate} Let $(v(t),\bar{c}(t),f(t))_{0\leqslant t<\tau}$ be as in Proposition~\ref{prop:galchi}. There is a constant $B_s$ depending on $s$ only such that, a.s., for all $T\in[0,\tau)$ satisfying 
\begin{equation}\label{smallT}
2TB_s\|\partial_x\bar{c}\|_{C([0,T];L^\infty(\T))}\leqslant 1,
\end{equation}
one has
\begin{multline}\label{vvvHs}
\|v\|_{C([0,T];H^s(\T))}\leqslant 2\|v(0)\|_{H^s(\T)}+2T\|f\|_{C([0,T];H^s(\T))}\\
+2TB_s\|v\|_{C([0,T];W^{1,\infty}(\T)}\, \|\bar{c}\|_{C([0,T];H^s(\T))}.
\end{multline}
In particular, there is a constant $C_s$ depending only on $s$ such that, if $T$ additionally satisfies
\begin{equation}\label{smallT2}
4T C_s\|\bar{c}\|_{C([0,T];H^s(\T))}\leqslant 1,
\end{equation}
then
\begin{equation}\label{vvvHs2}
\|v\|_{C([0,T];H^s(\T))}\leqslant 4\|v(0)\|_{H^s(\T)}+4T\|f\|_{C([0,T];H^s(\T))}.
\end{equation}
\end{proposition}

We adopt the following notations: $\Lambda^s$ denotes the operator $(I-\partial^2_x)^{s/2}$, $[A,B]=AB-BA$ the commutator of two operators $A$ and $B$. In order to prove Proposition~\ref{prop:Hsestimate}, we recall the following Kato--Ponce estimates \cite{KatoPonce1988}
\begin{align}
\|f\, g\|_{H^r}\ &\leqslant\ C\, \left( \|f\|_{L^\infty}\, \|g\|_{H^r}\ +\ \|f\|_{H^r}\, \|g\|_{L^\infty}\right), \label{Algebra} \\
\left\| \left[ \Lambda^r,\, f \right]\, g  \right\|_{L^2}\ &\leqslant\ C\, \left( \|f_x\|_{L^\infty}\, \|g\|_{H^{r-1}}\ +\ \|f\|_{H^r}\, \|g\|_{L^\infty} \right), \label{Commutator}
\end{align}
where $r\, \geqslant\, 0$, and $C>0$ is a constant depending only on $r$. 
We recall also the following estimate \cite{ConstantinMolinet2002}
\begin{equation}\label{Composition}
\left\|F(f)-F(0) \right\|_{H^s}\, \leqslant\ C \left( \|f\|_{L^{\infty}} \right) \|f\|_{H^s},
\end{equation}
for any $f \in H^s$ with $s >1/2$ and $F \in C^\infty(\R)$.

\begin{proof}[Proof of Proposition~\ref{prop:Hsestimate}] The second estimate \eqref{vvvHs2} follows from \eqref{vvvHs} and the continuous injection of $H^s(\T)$ into $W^{1,\infty}(\T)$ (the exponent $s$ satisfying $s>3/2$).
To prove \eqref{vvvHs}, we introduce $(\rho_\delta)$, a standard approximation of the unit, and let $J_\delta$ denote the Friedrichs operator  $f\mapsto f\ast \rho_\delta$, obtained by convolution by  $\rho_\delta$ with respect to the variable $x$. Denote also by $\bar{c}$ the operator of multiplication by $\bar{c}$. We apply $J_\delta$ to \eqref{vvv-PDE} to obtain
\begin{equation}\label{vvv-PDE-delta}
\partial_t v^\delta(t,x)+\bar{c}(t,x)\partial_x v^\delta(t,x)=f^\delta(t,x)+r^\delta(t,x),
\end{equation}
where
\[
v^\delta=J_\delta v,\quad f^\delta=J_\delta f,\quad r^\delta=\bar{c}\partial_x v^\delta-J_\delta(\bar{c}\partial_x v)=[\bar{c},J_\delta]\partial_x v.
\]
We have 
\[
v^\delta\in C^1([0,\tau);H^s(\T)),\quad \partial_x v^\delta, f^\delta,r^\delta\in C^0([0,\tau);H^s(\T)),
\]
so we can apply $\Lambda^s$ to \eqref{vvv-PDE-delta} to get 
\begin{equation}\label{vvv-PDE-delta-Lambda}
\partial_t \Lambda^s v^\delta(t,x)+\bar{c}(t,x)\partial_x \Lambda^s v^\delta(t,x)=\Lambda^s f^\delta(t,x)+R^\delta(t,x),
\end{equation}
where
\begin{equation}\label{defRdelta}
R^\delta=\Lambda^s r^\delta+[\bar{c},\Lambda^s]\partial_x v^\delta.
\end{equation}
We can then multiply \eqref{vvv-PDE-delta-Lambda} by $\Lambda^s v^\delta$, integrate on $\T$ and integrate by parts to get
\begin{multline}\label{vvv-Hs-estimate-1}
\|\Lambda^s v^\delta(t)\|_{L^2(\T)}\leqslant
\|\Lambda^s v^\delta(0)\|_{L^2(\T)}\\
+\int_0^t \left\{
\half\|\partial_x\bar{c}(\sigma)\|_{L^\infty(\T)}\|\Lambda^s v^\delta(\sigma)\|_{L^2(\T)}+\|\Lambda^s f^\delta(\sigma)\|_{L^2(\T)}+\|R^\delta(\sigma)\|_{L^2(\T)}\right\} \ud \sigma,
\end{multline}
which yields, for $T<\tau(\omega)$,
\begin{multline}\label{vvv-Hs-estimate-2}
\|v^\delta\|_{C([0,T];H^s(\T))}\leqslant
\|v^\delta(0)\|_{H^s(\T)}\\
+T \left\{
\half \|\partial_x\bar{c}\|_{C([0,T];L^\infty(\T))}\|v^\delta\|_{C([0,T];H^s(\T))}+\|f^\delta\|_{C([0,T];H^s(\T))}+\|R^\delta\|_{C([0,T]L^2(\T))}\right\}.
\end{multline}
Assume 
\begin{multline}\label{RdeltaFinal}
\|R^\delta\|_{C([0,T];L^2(\T))}\\
\leqslant B_s\left[
\|\partial_x\bar{c}\|_{C([0,T];L^\infty(\T))}\, \|v\|_{C([0,T];H^{s}(\T))}\ +\ \|v\|_{C([0,T];W^{1,\infty}(\T)}\, \|\bar{c}\|_{C([0,T];H^s(\T))}
\right],
\end{multline}
for a certain constant $B_s$.
Then we can insert \eqref{RdeltaFinal} in \eqref{vvv-Hs-estimate-2}, pass to the limit $\delta\to 0$ and, after replacing $B_s$ by $B_s+1/2$, conclude to \eqref{vvvHs}. To establish \eqref{RdeltaFinal}, we write
\begin{equation}\label{Rdelta1}
R^\delta= \Lambda^s [\bar{c},J_\delta]\partial_x v+[\bar{c},\Lambda^s]\partial_x v^\delta=[\Lambda^s, [\bar{c},J_\delta]]\partial_x v+[\bar{c},J_\delta]\Lambda^s\partial_x v+[\bar{c},\Lambda^s]\partial_x v^\delta,
\end{equation}
and use the Jacobi identity
\[
[\Lambda^s, [\bar{c},J_\delta]]+[[\bar{c},[J_ \delta,\Lambda^s]]+[J_\delta, [\Lambda^s,\bar{c}]]=0,
\]
and the fact that $[J_ \delta,\Lambda^s]=0$ to get
\begin{equation}
R^\delta=[J_\delta, [\bar{c},\Lambda^s]]\partial_x v+[\bar{c},J_\delta]\Lambda^s\partial_x v+[\bar{c},\Lambda^s]\partial_x v^\delta
=J_\delta [\bar{c},\Lambda^s]\partial_x v+[\bar{c},J_\delta]\partial_x(\Lambda^s v).\label{Rdelta3}
\end{equation}
The intermediate identities between \eqref{defRdelta} and \eqref{Rdelta3} can be understood in the sense of distribution, but we must check that each term in \eqref{Rdelta3} is in $L^2(\T)$ and satisfies the expected estimate. We have $[\bar{c},\Lambda^s]\partial_x v\in L^2(\T)$ by Kato-Ponce's commutator estimate \eqref{Commutator}, with
\begin{multline*}
\left\|[\bar{c}(t),\Lambda^s]\partial_x v(t)\right\|_{L^2(\T)}\\
\leqslant C_{\mathrm{KP}}\left[
\|\partial_x\bar{c}(t)\|_{L^\infty(\T)}\, \|\partial_x v(t)\|_{H^{s-1}(\T)}\ +\ \|\bar{c}(t)\|_{H^s(\T)}\, \|\partial_x v(t)\|_{L^\infty(\T)}
\right],
\end{multline*}
and thus 
\begin{multline*}
\left\|[\bar{c}(t),\Lambda^s]\partial_x v(t)\right\|_{L^2(\T)}\\
\leqslant C_{\mathrm{KP}}
\left[
\|\partial_x\bar{c}\|_{C([0,T];L^\infty(\T))}\, \|v\|_{C([0,T];H^{s}(\T))}\ +\ \|v\|_{C([0,T];W^{1,\infty}(\T)}\, \|\bar{c}\|_{C([0,T];H^s(\T))}
\right],
\end{multline*}
for $t\in[0,T]$. By the commutator estimate Lemma~II-1. in \cite{DiPernaLions89} (and more exactly, by the proof of Lemma~II-1), we have $[\bar{c},J_\delta]\partial_x(\Lambda^s v)\in L^2(\T)$ with 
\[
\left\|[\bar{c}(t),J_\delta]\partial_x(\Lambda^s v(t))\right\|_{L^2(\T)}
\leqslant C_{\mathrm{DPL}}\|\partial_x \bar{c}(t)\|_{L^2(\T)}\|\Lambda^s v(t)\|_{L^2(\T)},
\]
which yields $R^\delta(t)\in L^2(\T)$ for $t\in[0,T]$, and the estimate \eqref{RdeltaFinal}.
\end{proof}

To go on our analysis, we consider now the deterministic quasilinear (but non-local) equation 
\begin{equation}\label{Qlinv}
\partial_t v+\bar{c}[v]\partial_x v=F[v],
\end{equation}
on $(0,T)\times\T$, where $c,F\colon C([0,T];H^s(\T))\to C([0,T];H^s(\T))$ satisfy 
\begin{equation}\label{propQbarc1}
\|\partial_x \bar{c}[v]\|_{C([0,T];L^\infty(\T))}\leqslant \mathfrak{C}\left(\|v\|_{C([0,T];L^\infty(\T))}\right),
\end{equation}
and
\begin{equation}\label{propQbarcF10}
\|\partial_x F[v]\|_{C([0,T];L^\infty(\T))}\leqslant \mathfrak{C}\left(\|v\|_{C([0,T];L^\infty(\T))}\right) \left(1+\|\partial_x v\|_{C([0,T];L^\infty(\T))} \right),
\end{equation}
as well as the bounds
\begin{equation}\label{propQbarcF1}
\|\bar{c}[v]\|_{C([0,T];H^s(\T))}+\|F[v]\|_{C([0,T];H^s(\T))}\leqslant \mathfrak{C}\left(\|v\|_{C([0,T];L^\infty(\T))}\right)\left(1+\|v\|_{C([0,T];H^s(\T))}\right),
\end{equation}
for all $v\in C([0,T];H^s(\T))$, for a given function $\mathfrak{C}\colon\R_+\to\R_+$. We will also consider the Lipschitz condition
\begin{multline}\label{LipbarcF}
\|\bar{c}[v_1]-\bar{c}[v_2]\|_{C([0,T];L^2(\T))}
+\|F[v_1]-F[v_2]\|_{C([0,T];L^2(\T))}\\
\leqslant 
\mathfrak{C}\left(\|v_1\|_{C([0,T];W^{1,\infty}(\T))}+\|v_2\|_{C([0,T];W^{1,\infty}(\T))}\right)\|v_1-v_2\|_{C([0,T];L^2(\T))}.
\end{multline}

\begin{proposition}[$H^s$-estimate, linear equation]\label{prop:HsestimateNLIN} Assume $\bar{c}, F$ in \eqref{Qlinv} satisfy \eqref{propQbarc1} and \eqref{propQbarcF1}. Then, given $R>0$, there is a constant $\nu_R>0$ such that, if $T\nu_R\leqslant 1$, then all solutions $v\in C([0,T];H^s(\T))$ to the equation
\begin{equation}\label{QlinvFrozen}
\partial_t v+\bar{c}[w]\partial_x v=F[w],
\end{equation}
with initial datum $v_0\in H^s(\T)$ and entry $w\in C([0,T];H^s(\T))$ of size
\[
\|v_0\|_{H^s(T)}\leqslant R/3,\quad \|w\|_{C([0,T];H^s(\T))}\leqslant R,
\]
satisfy the bound 
\begin{equation}\label{vvvHsR}
\|v\|_{C([0,T];H^s(\T))}\leqslant R.
\end{equation}
\end{proposition}

\begin{proof}[Proof of Proposition~\ref{prop:HsestimateNLIN}] This is a straightforward consequence of the estimate \eqref{vvvHs}, using \eqref{propQbarc1}-\eqref{propQbarcF1} and the injection $H^s(\T)\hookrightarrow L^\infty(\T)$, which yields
\begin{equation}\label{vvvHsR0}
\|v\|_{C([0,T];H^s(\T))}\leqslant 2\|v_0\|_{H^s(T)}+C_R T\leqslant {\textstyle \frac23 } R+C_R T,
\end{equation}
for $T\nu_R\leqslant 1$, where $\nu_R>0$, $C_R\geqslant 0$. We obtain then \eqref{vvvHsR} by adjusting the vale of $\nu_R$ if necessary.
\end{proof}

Next, we state the contraction property in the low norm $L^2$.

\begin{proposition}[$L^2$-contraction, linear equation]\label{prop:L2contractNLIN} Assume $\bar{c}, F$ in \eqref{Qlinv} satisfy \eqref{propQbarc1}, \eqref{propQbarcF1} and also \eqref{LipbarcF}. Given $R>0$, there is a constant $\nu_R>0$ such that, if $T\nu_R\leqslant 1$, and $v_1,v_2\in C([0,T];H^s(\T))$ are solutions of the respective equations
\begin{equation}\label{QlinvFrozeni}
\partial_t v_i+\bar{c}[w_i]\partial_x v_i=F[w_i],\quad i\in\{1,2\},
\end{equation}
with respective initial datum $v_{1,0},v_{2,0}\in H^s(\T)$ and entries $w_1,w_2\in C([0,T];H^s(\T))$ of size
\[
\|v_0\|_{H^s(\T)}\leqslant R/3,\quad \|w_i\|_{C([0,T];L^\infty(\T))}\leqslant R,\quad i\in\{1,2\},
\]
then
\begin{equation}\label{vvvL2contract}
\|v_1-v_2\|_{C([0,T];L^2(\T))}\leqslant 4\|v_{1,0}-v_{2,0}\|_{L^2(\T)}+\half\|w_1-w_2\|_{C([0,T];L^2(\T))}.
\end{equation}
\end{proposition}

\begin{proof}[Proof of Proposition~\ref{prop:L2contractNLIN}] The difference $\bar{v}:=v_1-v_2$ satisfies the equation
\begin{equation}\label{Qlindetlasol}
\partial_t\bar{v}+\bar{c}[w_1]\partial_x\bar{v}=\bar{f}:=F[w_1]-F[w_2]+(\bar{c}[w_2]-\bar{c}[w_1])\partial_x v_2,
\end{equation}
starting from $\bar{v}_0:=v_{1,0}-v_{2,0}$, and thus the $L^2$-estimate
\begin{multline}\label{vvvL2contract1}
\|\bar{v}\|_{C([0,T];L^2(\T))}\leqslant \|v_{1,0}-v_{2,0}\|_{L^2(\T)}\\
+T\|\partial_x\bar{c}[w_1]\|_{C([0,T];L^\infty(\T))}\|\bar{v}\|_{C([0,T];L^2(\T))}+
2T\|\bar{f}\|_{C([0,T];L^2(\T))}.
\end{multline}
We apply Proposition~\ref{prop:HsestimateNLIN}: if $T$ is sufficiently small, then the bound \eqref{vvvHsR} is satisfied by $v_1$ and $v_2$. As a consequence of \eqref{propQbarc1} and \eqref{LipbarcF} (and of the injection $H^s(\T)\hookrightarrow W^{1,\infty}(\T)$ again), we obtain
\begin{equation}\label{vvvL2contract2}
\|\bar{v}\|_{C([0,T];L^2(\T))}\leqslant 2\|v_{1,0}-v_{2,0}\|_{L^2(\T)}+C'_R T\|\bar{v}\|_{C([0,T];L^2(\T))}+C''_R T\|\bar{w}\|_{C([0,T];L^2(\T))},
\end{equation}
with $\bar{w}:=w_1-w_2$, from which \eqref{vvvL2contract} follows for $T$ sufficiently small.
\end{proof}

We can then give a result of existence of solutions to \eqref{Qlinv} up to a time of explosion of the Lipschitz norm of the solution.

\begin{proposition}[Existence up to explosion of the sup-norm]\label{prop:MaxSolNlin} Assume $\bar{c}, F$ in \eqref{Qlinv} satisfy \eqref{propQbarc1}-\eqref{propQbarcF10}-\eqref{propQbarcF1} and also \eqref{LipbarcF}. Let $v_0\in H^s(\T)$ and let $T>0$. There exists a solution $v\in C([0,T^*);H^s(\T))\cap C^1([0,T^*);H^{s-1}(\T))$ to \eqref{Qlinv} with initial datum $v_0$ defined up to a time $T^*$ which satisfies
\begin{equation}\label{explosionTstar}
T^*\ =\ \sup_{n\geqslant 1}T^n,\qquad T^n\ \eqdef\ \inf\left\{t\in[0,T^*);\|v(t)\|_{L^\infty(\T)}> n\right\}.
\end{equation}
\end{proposition}

\begin{proof}[Proof of Proposition~\ref{prop:MaxSolNlin}] We consider the iterative scheme $v^k\mapsto v^{k+1}$, where $v^{k+1}$ is the solution (given by the method of characteristics, \textit{cf.} Proposition~\ref{prop:galchi}) to the equation
\begin{equation}\label{Qlinvk}
\partial_t v^{k+1}+\bar{c}[v^k]\partial_x v^{k+1}=F[v^k],
\end{equation}
with initial datum $v_0$. By Proposition~\ref{prop:HsestimateNLIN} and Proposition~\ref{prop:L2contractNLIN}, there is a time $T_\ell$ depending on $\|v_0\|_{H^s(\T)}$ only such that the sequence $(v^k)$ converges in $C([0,T_\ell];H^s(\T))$ to a solution $v$ to \eqref{Qlinv} with initial datum $v_0$ defined on $[0,T_\ell]$. This gives the local existence. Local uniqueness follows from \eqref{vvvL2contract} (possibly applied repeatedly). So  the problem \eqref{Qlinv} with initial datum $v_0$ admits a maximal solution defined up to the time $S^*$ defined by
\begin{equation}\label{explosionSstar}
S^*\ \eqdef\ \sup_{n\geqslant 1}S^n,\qquad S^n\ \eqdef\ \inf\left\{t\in[0,T^*);\|v(t)\|_{H^s(\T)}> n\right\}.
\end{equation}
Let us first introduce the explosion time of the Lipschitz norm:
\begin{equation}\label{explosionT1star}
\hat{T}^*\ \eqdef\ \sup_{n\geqslant 1}\hat{T}^n,\qquad\hat{T}^n\ \eqdef\ \inf\left\{t\in[0,T^*);\|v(t)\|_{W^{1,\infty}(\T)}> n\right\}.
\end{equation}
The norm in $H^s(\T)$ dominates the norm in $W^{1,\infty}(\T)$ so $S^*\leqslant \hat{T}^*$. On the other hand, if $t<\hat{T}_n\wedge S^*$, so that $v$ solves \eqref{Qlinv} on $[0,t]$ and 
\[
\|v\|_{C([0,t];W^{1,\infty}(\T))}\leqslant n,
\]
we observe that \eqref{vvvHs} gives (for arbitrary times $0\leqslant r<r+\sigma\leqslant t$)
\begin{multline}\label{vvvHsn}
\|v\|_{C([r,r+\sigma];H^s(\T))}\leqslant 2\|v(r)\|_{H^s(\T)}+2\sigma \left( \|F[v]\|_{C([r,r+\sigma];H^s(\T))}
+ B_s n\, \|\bar{c}[v]\|_{C([r,r+\sigma];H^s(\T))} \right).
\end{multline}
We deduce then from \eqref{propQbarcF1} that
\begin{equation}\label{vvvHsn2}
\|v\|_{C([r,r+\sigma];H^s(\T))}\leqslant 2\|v(r)\|_{H^s(\T)}+2\sigma\mathfrak{C}(n) \left(1+ B_s n \right)\! \left( \|v\|_{C([r,r+\sigma];H^s(\T))} +1 \right),
\end{equation}
and thus 
\begin{equation}\label{vvvHsn3}
\|v\|_{C([r,r+\sigma];H^s(\T))}\leqslant K(n)\left(\|v(r)\|_{H^s(\T)} + 1 \right),
\end{equation}
if $\sigma<\sigma(n)$ for a $\sigma(n)$ sufficiently small depending on $n$ (and $s$) only. By iterating \eqref{vvvHsn3}, we obtain the bound $\|v\|_{C([0,t];H^s(\T))}\leqslant H(n)$ for a given function $H\colon\N\to\N$, and thus $t<S_{H(n)}$. Letting $t\uparrow \hat{T}_n\wedge S^*$ gives $\hat{T}_n\wedge S^*\leqslant S_{H(n)}\leqslant S^*$ and thus $\hat{T}_n\leqslant S^*$. We conclude that $\hat{T} ^*\leqslant S^*$ and finally that $\hat{T}^*=S^*$. 

Let us now prove that $T^*=\hat{T}^*$. Clearly $\hat{T}^*\leqslant T^*$. We consider a time $t< T_n\wedge\hat{T}^*$. By differentiation in $x$ in \eqref{Qlinv} on $[0,t]$, we obtain the equation
\begin{equation}\label{Qlinvx}
\partial_t w+c[v]\partial_x w=\partial_x F[v]-w\partial_x c[v],
\end{equation}
where $w := \partial_x v$. By the comparison principle, and for $s\in[0,t]$,
\begin{equation}\label{xQlinPmax}
\|w(s)\|_{L^\infty(\T)}\leqslant \|\partial_x v(0)\|_{L^\infty(\T)}+\int_0^s\|\partial_x F[v]-w\partial_x c[v]\|_{L^\infty(\T)}(\sigma) \, \ud \sigma.
\end{equation}
We use \eqref{propQbarc1}, \eqref{propQbarcF10} and the Gr\"onwall lemma, to deduce from \eqref{xQlinPmax} that
\begin{equation}\label{xQlinPmax2}
\sup_{0\leqslant s \leqslant t}\|w(s)\|_{L^\infty(\T)}\ \leqslant\, \left( \|\partial_x v(0)\|_{L^\infty(\T)} + C(n) t \right) \exp(C(n)t),
\end{equation}
for a certain constant $C(n)$. This implies $t<\hat{T}_{\tilde{H}(n)}$ for a certain function $\tilde{H}\colon\N\to\N$, which, as above, gives $T^*\leqslant\hat{T}^*$, and so finally $T^*=\hat{T}^*$.
\end{proof}

\begin{thm}[Local existence of regular solutions]\label{th:loc-exist} Let $s>3/2$. Assume \eqref{coeff-c},  \eqref{sigmas1} and $|c'(u)| \leqslant c_3$. Let $V_0\in H^s(\T;\R^2)$ such that $\int_\T (Q_0-P_0)\, \ud x = 0$. Then \eqref{SVWE_Qlinear-eps} admits a solution up to an explosion time $\tau$ and two solutions $(V_i,\tau_i)$, $i=1,2$ to \eqref{SVWE_Qlinear-eps} coincide, in the sense that $\tau_1=\tau_2$ and $V_1=V_2$ on $[0,\tau_1)$.
\end{thm}

\begin{proof}[Proof of Theorem~\ref{th:loc-exist}] We use Proposition~\ref{prop:MaxSolNlin} (the results can be extended to the case of two equations in a straightforward way) to solve \eqref{SVWE_Qlinear-eps} pathwise. Therefore it will be sufficient to prove that, a.s., $u=u[V]$ being defined by \eqref{udefepthetaPQ}, and $F_\theta$ given by \eqref{defFeps}, the functions
\[
V\mapsto c[V]:=c(u),\quad V\mapsto F_{\theta}(V)
\]
satisfy \eqref{propQbarc1}-\eqref{propQbarcF10}-\eqref{propQbarcF1}-\eqref{LipbarcF}. If we take into account the regularity \eqref{WHsas} and the definition of $(P,Q)$ as the shift of $(R,S)$ by $ \Phi W$, then the property \eqref{propQbarc1} follows from the identity
\begin{equation}\label{cuxthetaeps}
c(u) u_x
=\frac{S-R}{2}-\Theta ,
\end{equation}
and the bounds in \eqref{coeff-c}. Similarly, we deduce \eqref{propQbarcF1} and \eqref{LipbarcF} from \eqref{coeff-c}, \eqref{WHsas}, \eqref{udefepthetaPQ}, \eqref{cuxthetaeps} and \eqref{Composition}. The property \eqref{propQbarcF10} is a straightforward consequence of \eqref{WHsas} and the expression \eqref{defFeps} of $F_\theta$. The local uniqueness is a consequence of \eqref{vvvL2contract}, which gives $V_1=V_2$ on $[0,\tau_1^*\wedge\tau_2^*)$, but then also yields $\tau_1^*=\tau_2^*$ since both are explosion times. Finally, the predictable character of the stochastic processes can be deduced from the predictable character of each element of the iteration \eqref{Qlinvk}.
\end{proof}

\subsection{Conservative form, integration along the characteristics, It\^o's formula}\label{subsec:ItoANDchar}

The identity \eqref{eq-def-sol-V-eps} should be understood in $H^{s-1}(\T)$. In particular, \eqref{eq-def-sol-V-eps} is satisfied point-wise for every $x\in \T$ and so is the system \eqref{SVWEeptheta} (once interpreted in the standard integral form of SDEs). The It\^o formula can therefore be applied for every $x\in\T$ to the real-valued processes $(R(t,x))_{t<\tau}$ and $(S(t,x))_{t<\tau}$ and gives, for all smooth function $h\in C^\infty(\R)$,
\begin{align}\nonumber
\ud h(R)\ +\ c(u)\, h(R)_x\, \ud t\ &=\ \tilde{c}'(u) \, h'(R)\, (R^2\, -\, S^2\, -\, \theta(R)\, +2\, R\, \Theta )\,\ud t\ +\ \half\, q\, h''(R)\, \ud t\\ \label{ReqepthetaIto}
&\quad   +\ h'(R)\, \Phi\, \ud W, 
\end{align}
and
\begin{align}\nonumber
\ud h(S)\ -\ c(u)\, h(S)_x\, \ud t\ &=\ \tilde{c}'(u) \, h'(S)\, (S^2\, -\, R^2\, -\, \theta(S)\, -2\, S\, \Theta ) \, \ud t\ +\ \half\, q\, h''(S)\, \ud t\\ \label{SeqepthetaIto}
& \quad +\ h'(S)\, \Phi\, \ud W.
\end{align}
We can then use the identity \eqref{cuxthetaeps} to obtain the conservative form
\begin{gather}\nonumber
\ud h(R)\ +\ (c(u)\, h(R))_x\, \ud t\ =\ h'(R)\, \Phi\, \ud W\\ \label{ReqepthetaItoConservative}
+\ \tilde{c}'(u) \left[(S\, -\, R)\, B_h(R,S)\, -\, h'(R)\, \theta(R)\, +\,
2\, \Theta \, (R\, h'(R)\, -\, 2\, h(R))\right] \ud t\ +\ \half\, q\, h''(R)\, \ud t,
\end{gather}
and
\begin{gather}\nonumber
\ud h(S)\ -\ (c(u)\, h(S))_x\, \ud t\ =\ h'(S)\, \Phi\, \ud W\\  \label{SeqepthetaItoConservative}
+\ \tilde{c}'(u) \left[(R\, -\, S)\, B_h(S,R)\, -\, h'(S)\, \theta(S)\, -\,
2\, \Theta \, (S\, h'(S)\, -\, 2\, h(S))\right] \ud t\ +\ \half\, q\, h''(S)\, \ud t,
\end{gather}
where
\begin{equation}\label{defBh}
B_h(R,S)\ \eqdef\ 2\, h(R)\ -\ (R+S)\, h'(R).
\end{equation}
By taking $h(R)=R^2$, we obtain 
\begin{equation}
B_h(R,S)=-2RS=B_h(S,R),\quad R\, h'(R)\, -\, 2\, h(R)=0,
\end{equation}
and, by adding \eqref{ReqepthetaItoConservative} to \eqref{SeqepthetaItoConservative},
\begin{align} \nonumber
(R^2+S^2)(\tau')\ &=\ R^2(0)\ +\ S^2(0)\ +  \int_0^{\tau'} 2\, (R\ +\ S)\, \Phi\, \ud W
\\ \label{sumsquare-theta}
& \quad +\int_0^{\tau'} \left[ (c(u)\, (S^2-R^2))_x\ -\ 2\, \tilde{c}'(u) \left( R\, \theta(R)\, +\,  S\, \theta(S)\right)
\, +\ 2\, q\right] \ud t,
\end{align}
for any stopping time $\tau'<\tau$.\medskip

We will also need the evolution equation of $\frac{h(R)}{c(u)}$ (and similarly for $S$, see below). To that purpose, we compute the derivative of $u$.

\begin{proposition}[Computation of $u_t$]\label{prop:ut} Let $(V,\tau)$ be a regular solution to \eqref{SVWE_Qlinear-eps} up to the explosion time $\tau$ and let $(u,\Theta )$ being given as a functions of $V$ by \eqref{udefepthetaPQ}-\eqref{psiepthetaPQ}. We have then
\begin{equation}\label{psiprime}
\Theta '(t)=\alpha(t)+\beta(t)\Theta (t),
\end{equation}
where
\begin{equation}\label{psiprimealphabeta}
\alpha(t)=\frac12\int_0^1 \tilde{c}'(u)\left(\theta(R)-\theta(S)\right) \ud y,
\quad\beta(t)=\int_0^1 \tilde{c}'(u)\left(R+S\right) \ud y,
\end{equation}
and
\begin{equation}\label{utOK}
u_t=\frac{R+S}{2}+\frac{1}{c(u)}\int_0^x \left[\zeta(t,y)-\bar{\zeta}(t)\right]\ud y,	
\end{equation}
where
\begin{equation}\label{zeta}
\zeta=\tilde{c}'(u)\left[\frac{\theta(R)-\theta(S)}{2}+(R+S)\Theta \right],\quad \bar{\zeta}(t)=\int_0^1\zeta(t,y)\, \ud y.
\end{equation}
\end{proposition}

\begin{proof}[Proof of Proposition~\ref{prop:ut}] We consider the equations \eqref{ReqepthetaItoConservative}-\eqref{SeqepthetaItoConservative} with $h(R)=R$, and subtract the first equation from the second one to obtain
\begin{equation}\label{SRconservative}
\partial_t(S-R)-(c(u)(R+S))_x=\tilde{c}'(u)\left[(\theta(R)-\theta(S))+2\, (R+S)\, \Theta \right].	
\end{equation}
By integration on $(0,1)$ in \eqref{SRconservative}, we obtain \eqref{psiprime}. Then we observe that, taking $x=0$ in \eqref{udefeptheta} gives
\begin{equation}\label{ut0}
u(t,0)=\int_0^t  \left(\frac{R\, +\, S}{2}  \right) (s,0)\, \ud s\ +\ u_0(0),
\end{equation}
so that
\begin{equation}\label{CCu}
\mathcal{C}(u(t,x))=\mathcal{C}(u(t,0))+\int_0^x \left[\frac{S-R}{2}(t,y)-\Theta (t)\right]\ud y.	
\end{equation}
By differentiation in \eqref{ut0} and \eqref{CCu} with respect to $t$, we obtain
\begin{equation}\label{cut1}
c(u)u_t=c(u(t,0))\left(\frac{R+S}{2} \right) (t,0)+\int_0^x \left[\frac{\partial_t(S-R)}{2}(t,y)-\Theta '(t)\right]\ud y.	
\end{equation}
We use \eqref{SRconservative} and \eqref{psiprime} to get
\begin{equation}\label{cut2}
c(u)u_t=c(u)\frac{R+S}{2}+\int_0^x \left[\zeta(t,y)-\bar{\zeta}(t)\right]\ud y,	
\end{equation}
where $\zeta$ and $\bar{\zeta}$ are defined in \eqref{zeta}. Dividing by $c(u)$ yields \eqref{utOK}.
\end{proof}

Set 
\begin{equation}\label{xiut}
\Xi\ \eqdef\ u_t\, -\, \frac{R+S}{2}\ =\ \frac{1}{c(u)}\int_0^x \left[\zeta(t,y)-\bar{\zeta}(t)\right]\ud y.	
\end{equation}
It follows from \eqref{ReqepthetaIto}-\eqref{SeqepthetaIto} that 
\begin{gather}\nonumber
\ud \left[{\textstyle \frac{h(R)}{c(u)}}\right] +\ [h(R)]_x\, \ud t\ =\ h'(R)\, \Phi\, \ud W\\ \label{RCeqepthetaIto}
+\ {\textstyle\frac{\tilde{c}'(u)}{c(u)}} \left[4\, h(R)\, \Xi\, +\, (S\, +\, R)\, F_h(R,S)\, -\, h'(R)\, \theta(R)\, +\,
2\, \Theta \, R\, h'(R)\right] \ud t\  
+\ \tfrac{q\, h''(R)}{2\, c(u)}\, \ud t,
\end{gather} 
and
\begin{gather}\nonumber
\ud \left[ {\textstyle \frac{h(S)}{c(u)}}\right] -\ [h(S)]_x\, \ud t\ =\ h'(S)\, \Phi\, \ud W\\  \label{SCeqepthetaIto}
+\ {\textstyle\frac{\tilde{c}'(u)}{c(u)}} \left[4\, h(S)\, \Xi\, +\,(R\, +\, S)\, F_h(S,R)\, -\, h'(S)\, \theta(S)\, -\,
2\, \Theta \, S\, h'(S)\right] \ud t\
+\ \tfrac{q\, h''(S)}{2\, c(u)}\, \ud t,
\end{gather}
where
\begin{equation}\label{defFh}
F_h(R,S)\ :=\ 2\, h(R)\ -\ (R-S)\, h'(R).
\end{equation}


We complete this section with the the expression of \eqref{ReqepthetaIto}-\eqref{SeqepthetaIto} after integration along the characteristic curves (\textit{cf.} Proposition~\ref{prop:galchi}).

\begin{proposition}[Integration along the characteristic curves]\label{prop:characteristics} Let $(V(t))_{t<\tau}$ be a regular solution to \eqref{SVWE_Qlinear-eps} up to the explosion time $\tau$. Let $h\in C^\infty(\R)$. Let $X^\pm$ denote the flow associated to the vector field $\pm c(u(t,\cdot))$: 
\begin{equation}\label{flowXpm}
\frac{\ud\;}{\ud t} X^\pm(t,x)\ =\ \pm c(u(t,X^\pm(t,x))), \qquad X^\pm(0,x)\ =\ x,\qquad x\in\T,\quad t\in [0,\tau).
\end{equation}
Let also $Y^\pm(t,\cdot)$ denote the inverse of the map $x\mapsto X^\pm(t,x)$. We have then
\begin{enumerate}
\item $X^\pm$ and $Y^\pm$ are well-defined, $C^1$-maps on $[0,\tau)\times\T$; for all $x\in\T$, $(X^\pm(t,x))$ and $(Y^\pm(t,x))$ are predictable processes, and the stochastic integrals
\begin{equation}\label{Mh}
\mathcal{M}_h(t,x)=\sum_k\int_0^t \left\{h'(R) \sigma_k\right\}\left(s,X^+(s,x)\right)\, \ud \beta_k(s),
\end{equation}
and
\begin{equation}\label{Nh}
\mathcal{N}_h(t,x)=\sum_k\int_0^t \left\{h'(S) \sigma_k\right\}\left(s,X^-(s,x)\right)\, \ud \beta_k(s),
\end{equation}
define martingales,
\item we have
\begin{gather} \nonumber
h(R)(t,x)\ =\ h(R_0)(Y^+(t,x))\ +\ \mathcal{M}_h(t,Y^+(t,x))\\ \label{hRalongchar}
+ \int_0^t \left\{\tilde{c}'(u)\, h'(R)\, (R^2\, -\, S^2\, -\, \theta(R)\, +2\, R\, \Theta )\, +\, \half\, q\, h''(R)\right\}\left(s,X^+(s,Y^+(t,x))\right) \ud s,
\end{gather}
and
\begin{gather}\nonumber
h(S)(t,x)\ =\ h(S_0)(Y^-(t,x))\ +\ \mathcal{N}_h(t,Y^-(t,x)) \\ \label{hSalongchar}
+ \int_0^t \left\{\tilde{c}'(u)\, h'(S)\, (S^2\, -\, R^2\, -\, \theta(S)\, -2\, S\, \Theta ) \, +\, \half\, q\, h''(S)\right\}\left(s,X^-(s,Y^-(t,x))\right) \ud s ,
\end{gather}
for any stopping time $t\in[0,\tau)$ and all $x\in\T$.
\end{enumerate}
\end{proposition}


\subsection{Energy estimate}\label{subsec:Eestimates}

We give here an energy estimate for the system~\eqref{SVWE2}. The corresponding estimate for \eqref{SVWEep} will be proved later in Section~\ref{subsec:Energy-eps}.

\begin{proposition}[Energy estimate]\label{prop:GlobalEnergy} Let $(R(t),S(t))_{t<\tau}$ be a regular solution to \eqref{SVWE2} up to a stopping time $\tau$. Let
\begin{equation}\label{TotalEnergy}
\mathcal{E}(t)\eqdef \|(R,S)(t)\|_{L^2(\T)}^2\ =\ \int_\T (R^2\ +\ S^2)(t,x)\, \ud x
\end{equation}
denote the total energy of the system. Then, for any stopping time $t<\tau$ a.s.,
\begin{equation}\label{eq:TotalEnergy}
\mathcal{E}(t)\ =\ \mathcal{E}(0)\ +\ 2\, \|q\|_{L^1}\, t\ +\, 2\, \mathcal{M}(t),
\end{equation}
where $\mathcal{M}(t)$ is a martingale satisfying the bound
\begin{equation} \label{BoungMGlobalEnergy}
\E\left[\mathcal{M}(t\wedge T)^2\right]\, \leqslant\ 2\, q_0\, \mathcal{E}(0)\, T\ +\ 2\, q_0^2\, T^2,
\end{equation}
for all $T>0$. 
\end{proposition}

\begin{proof}[Proof of Proposition~\ref{prop:GlobalEnergy}] We apply \eqref{sumsquare-theta} (where $\theta\equiv 0$).
By integration with respect to $x\in\T$ in (and by the stochastic Fubini Theorem), we obtain \eqref{eq:TotalEnergy} where $\mathcal{M}(t)$ is the stochastic integral 
\begin{multline}
\mathcal{M}(t)\ \eqdef\ \int_0^t \int_\T \sum_{k \geqslant 1 }\, (R+S)(s,x)\, \sigma_k(x)\, \ud x\, \ud \beta_k(s)\\
=\ \int_0^\infty \mathds{1}_{s<t}\int_\T \sum_{k \geqslant 1 }\, (R+S)(s,x)\, \sigma_k(x)\, \ud x\, \ud \beta_k(s),
\end{multline}
with variance
\begin{equation}\label{VarM}
\E\left[\mathcal{M}(t)^2\right]\, =\ \E \int_0^t \sum_{k \geqslant 1 } \left| \int_T (R+S)(s,x)\, \sigma_k(x)\, \ud x \right|^2 \ud s.
\end{equation}
From \eqref{eq:TotalEnergy} follows the global control on the energy
\begin{equation}\label{Eene}
\E\left[\mathcal{E}(t\wedge T)\right]\, \leqslant\ \mathcal{E}(0) +\ 2\, q_0\, T,
\end{equation}
which in turns (using the Cauchy--Schwarz inequality and \eqref{defq}) gives \eqref{BoungMGlobalEnergy}.
\end{proof}

\subsection{Solutions with finite explosion time}

\subsubsection{Some preliminary estimates}\label{sec:estimates}

In this section, we consider solely the system \eqref{SVWE2}. We establish a bound on the square $R^2$ (\textit{resp.} $S^2$) integrated along the characteristic curve $X^-$ (\textit{resp.} $X^+$). If applied to the first equation \eqref{Req} for instance, the bound of $S^2$ along $X^+$ can be used to control the negative term $-S^2$ in the right-hand side of \eqref{Req}, which then cannot prevent the blow-up due to the term $R^2$. This is exploited in the following section~\ref{subsec:singularityoccur}, where solutions with a finite blow-up time are exhibited.

\begin{proposition}[Control along the opposite characteristic curve]\label{prop:controlEnergies} Let $(R(t),S(t))_{t<\tau}$ be a regular solution to \eqref{SVWE2} up to a stopping time $\tau$. Let $X^\pm$ denote the characteristic curves \eqref{flowXpm}. There is a constant $C$ depending on the constants $c_1$ and $c_2$ in \eqref{coeff-c} only, such that: for all $x_1\in\T$, 
\begin{align}\nonumber
\int_{0}^{t\wedge T} \left[R^2(s,{X^-(s,x_1+2 c_2 T)})\, +\,  S^2(s,{X^+(s,x_1)}) \right] \ud s\ \leqslant &\, C\, (T+1)\, \|(R_0, S_0)\|_{L^2(\T)}^2\, +\ C\, q_0\, T^2\\ \label{RSalongcharac}
&\ +\ M_{x_1}(t\wedge T),
\end{align} 
for any stopping time $t$ such that $t<\tau$ a.s., and all $T>0$, where $M_{x_1}(t)$ satisfies
\begin{equation}\label{boundMx12}
\E\left[ M_{x_1}(t\wedge T)^2\right]\, \leqslant\ C\, q_0\, \|(R_0,S_0)\|_{L^2(\T)}^2\, T\ +\ C\, q_0^2\, T^2, \qquad  \E\, M_{x_1}(t\wedge T)\ =\ 0,
\end{equation}
for all $T>0$ and all stopping time $t<\tau$ a.s.
\end{proposition}

\begin{proof}[Proof of Proposition~\ref{prop:controlEnergies}] The letter $C$ will denote any constant (possibly varying from line to line) depending solely on $c_1$ and $c_2$.
We consider the solution $(R,S)$ as a periodic function defined on the real line.
To establish \eqref{RSalongcharac}, we define
\begin{equation}\label{distx12}
x_2\ \eqdef\ x_1\, +\, 2\, c_2\, T. 
\end{equation}
The definition \eqref{distx12} and the bound of $c(u)$ by $c_2$ in \eqref{coeff-c} ensure that 
\begin{equation}
X^+(t\wedge T,x_1)\ \leqslant\  X^-(t\wedge T,x_2).
\end{equation}
Since $c(u)$ is bounded from below by $c_1$, the map $s\mapsto X^+(s,x_1)$ is a $C^1$-diffeomorphism from $[0,t]$ onto $[x_1,X^+(t,x_1)]$, we denote by $\tau^+_{x_1}$ its inverse and extend the domain of definition of $\tau^+_{x_1}$ by setting
\begin{equation}\label{deftauplus}
\tau^+_{x_1}(y)\ \eqdef\ \inf\left\{s\in[0,t\wedge T]\, ;\, X^+(s,x_1)\geqslant y \right\}.
\end{equation}
Similarly, we set
\begin{equation}\label{deftaumoins}
\tau^-_{x_2}(y)\ \eqdef\ \inf\left\{s\in[0,t\wedge T]\, ;\, X^-(s,x_2)\leqslant y \right\}.
\end{equation}
\begin{figure}[!ht]
\begin{tikzpicture}[thick, transform canvas={scale=1}, shift={(8,-5)}]
\draw[->] (-6,0) -- (6,0) node[right]{$x$};
\draw[->] (-6,0) -- (-6,5) node[right]{$t$};
 
\draw [black] plot [smooth, tension=1] coordinates { (4.7,0) (3.8,1) (2.5,3) (1,4)};

\draw [black] plot [smooth, tension=1] coordinates { (-5,0) (-4,1) (-3,3) (-1,4)};

\draw[dashed] (-1,4) -- (-6,4) node[left]{$t \wedge T$};
\draw[dashed] (0,4) -- (0,0) node[below]{$\bar{x}$};
\draw[dashed] (-1,4) -- (-1,0);
\draw[dashed] (1,4) -- (1,0) ;

\draw[dashed] (1.5,0) -- (1.5,0) node[below]{{\scriptsize $X^-(x_2,t\wedge T)$}};

\draw[dashed] (-1.2,0) -- (-1.2,0) node[below]{{\scriptsize $X^+(x_1,t\wedge T)$}};

\draw[] (-1,4) -- (1,4) ;

\draw[dashed] (-3,3) -- (-6,3) node[left]{$s$};
\draw[dashed] (-3,3) -- (-3,0) node[below]{{\scriptsize $X^+(x_1,s)$}};

\draw[] (4.7,0) -- (4.7,0) node[below]{$x_2$};
\draw[] (-5,0) -- (-5,0) node[below]{$x_1$};
\end{tikzpicture}
\vspace{5.7cm}
\caption{Characteristics.}
\label{fig:combinedcharac}
\end{figure}
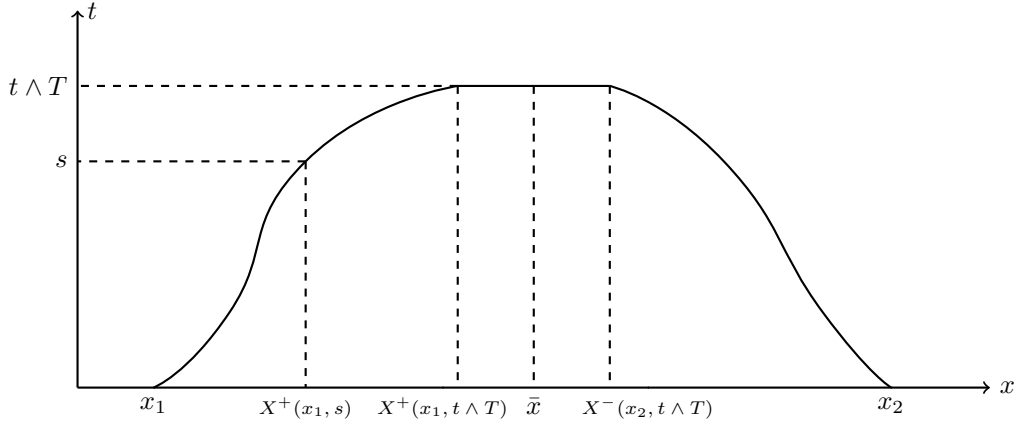
We define $\bar{x} \eqdef (x_1+x_2)/2 \in [X^+(t\wedge T,x_1),  X^-(t\wedge T,x_2)]$  (see Figure~\ref{fig:combinedcharac}), and we use the  two parametrizations
\begin{align*}
A_1\ \eqdef\ \{(s,y), x_1 \leqslant y \leqslant \bar{x},\ 0 \leqslant s \leqslant \tau^+_{x_1}(y) \}\ =\ \{(s,y), 0 \leqslant s \leqslant t\wedge T,\  X^+(x_1,s) \leqslant y \leqslant \bar{x} \},\\
A_2\ \eqdef\ \{(s,y), \bar{x} \leqslant y \leqslant x_2,\ 0 \leqslant s \leqslant \tau^-_{x_2}(y) \}\ =\ \{(s,y), 0 \leqslant s \leqslant t\wedge T,\ \bar{x}  \leqslant y \leqslant X^-(x_2,s) \}.
\end{align*}
Define
\begin{equation}
G\ \eqdef\ 
\begin{pmatrix}
R^2+S^2-\tilde{\mathcal{M}}\\
c(u)\, [R^2-S^2] 
\end{pmatrix}, \qquad \tilde{\mathcal{M}}(t,x)\ \eqdef\ \int_0^{t} 2 \left[R(s,x)\, +\, S(s,x) \right]  \Phi(x)\, \ud W(s).
\end{equation}
Then, using \eqref{sumsquare-theta} (with $\theta \equiv 0$), we obtain 
\begin{equation}\label{div}
\mathrm{div}_{t,x} G\ \leqslant\ 2\, q_0.
\end{equation}
Integrating now \eqref{div} on $A_1 \cup A_2$, using the divergence theorem and some changes of variables we obtain
\begin{multline}\label{appli-stokes}
2 \int_{0}^{t\wedge T} \left[c(u)\,R^2(s,{X^-(s,x_2)})\, +\, c(u)\, S^2(s,{X^+(s,x_1)}) \right] \ud s\\ 
\leqslant\ 2\, q_0\, T\, (x_2-x_1)\, +\, \int_{x_1}^{x_2} (R_0^2 + S_0^2)\, \ud x\, -\, \int_{X^+(t \wedge T,x_1)}^{X^-(t \wedge T,x_2)} (R^2 + S^2)\, \ud x\\
+\, \int_{x_1}^{\bar{x}} \tilde{\mathcal{M}} (\tau_{x_1}^+(y),y)\, \ud y\, +\, \int^{x_2}_{\bar{x}} \tilde{\mathcal{M}} (\tau_{x_2}^-(y),y)\, \ud y .
\end{multline}
This inequality \eqref{appli-stokes} and the lower bound $c(u) \geqslant c_1$ imply \eqref{RSalongcharac}, where
\begin{align*}
M_{x_1}(t \wedge T)\ &\eqdef\  (2\, c_1)^{-1} \left[   \int_{x_1}^{\bar{x}} \tilde{\mathcal{M}} (\tau_{x_1}^+(y),y)\, \ud y\, +\, \int^{x_2}_{\bar{x}} \tilde{\mathcal{M}} (\tau_{x_2}^-(y),y)\, \ud y  \right]. \\
%
\end{align*}
Since $\tilde{\mathcal{M}}$ is a martingale and $\tau^+_{x_1}(y)$ and $\tau^-_{x_2}(y)$ are stopping times, then $\E M_{x_1}(t \wedge T) =0$. 
Using
\begin{equation}\label{MtildeM}
|M_{x_1}(t\wedge T)|^2\ \leqslant\ C(T) \int_\T \sup_{s \in [0,t \wedge T]} |\tilde{\mathcal{M}}(s,y)|^2\, \ud y
\end{equation}
and the energy inequality \eqref{Eene} we obtain  \eqref{boundMx12}.
\end{proof}
\subsubsection{Singularity formation}\label{subsec:singularityoccur}

Let $u^\star \in \R$ be such that $c'(u^\star)>0$. We introduce (using the continuity of $c'(\cdot)$) a length $L>0$ such that 
\begin{equation}\label{c'>0}
\left|u\, -\, u^\star \right|\ \leqslant\ L \quad  \implies \quad  c'(u)\ \geqslant\ c'(u^\star)/2\ >\ 0.
\end{equation}
Let $\varphi \in C^\infty_c(\R)$ be supported in the interval  $[1/4,3/4]$ and non trivial: we assume that $\varphi'(x_0)<0$, for a given $x_0\in(1/4,3/4)$. For 
$\varepsilon \in (0,1)$ and $\alpha, \nu, \gamma$ such that 
\begin{equation}
\alpha\ >\ 1, \qquad \nu\ \in (0,\alpha-1), \qquad \gamma\ >\ 1/3,
\end{equation}
we consider the initial data
\begin{equation}\label{u0eps}
u^\eps_0(x)\ =\ u^\star\ +\ \varepsilon^\alpha\, \varphi \left( \frac{x} {\varepsilon^{\alpha + \nu+\gamma}} \right),\qquad v^\eps_0(x)=0,
\end{equation}
defined for $x\in [0,1)$ and extended by periodicity. 

\begin{thm}[Blow-up in finite time]\label{th:blowup} Let $(u(t))_{t<\tau}$ be a regular solution up to the explosion time $\tau$ of \eqref{SVWE1} with initial data given by \eqref{u0eps}. There exists $\eps_0>0$ depending on the constants $c_1$, $c_2$ in \eqref{coeff-c} and $c_3$, on $L$, on $q_0$ and on the profile $\varphi$ only, such that
\begin{equation}\label{PBlowup}
\Pro(\tau\leqslant \eps^{\gamma})\geqslant 1-\eps,
\end{equation}
for all $\eps<\eps_0$.
\end{thm}

\begin{proof}[Proof of Theorem~\ref{th:blowup}] Set $x_\eps=\eps^{\alpha + \nu + \gamma} x_0$ and 
\begin{equation}\label{RplusSplus}
R_{\chi^+}(t)\ =\ R(t,X^+(t,x_\eps)),\qquad S_{\chi^+}(t)\ =\ S(t,X^+(t,x_\eps)),
\end{equation}
where $X^+$ is the characteristic curve defined by \eqref{flowXpm}. We will estimate the possible blow-up of $R_{\chi^+}$ before a time $\eps^{\gamma}$. The letters $\eps_0$ and $C$ will denote any positive constant depending on the constants $c_1$, $c_2$ in \eqref{coeff-c},  on $c_3$, on $L$, on $q_0$ and on the profile $\varphi$ only, with a value possibly changing from line to line. We will however use the specific notation
\begin{equation}
\delta\ \eqdef\ \third\, c_1\, |\varphi'(x_0)|.
\end{equation} 
By \eqref{hRalongchar}, we have 
\begin{subequations}\label{Ralongchar-blowup}
\begin{align}
R_{\chi^+}(t)\ =\ &R_0({x_\varepsilon})\ +\ \int_0^t  \tilde{c}'(u)(s,X^+(s,x_\eps))\, R_{\chi^+}^2(s)\, \ud s\label{Ralongchar-blowup1}\\ 
&-\ \int_0^t \left(\tilde{c}'(u)\, S^2 \right)\! (s,X^+(s,x_\eps))\, \ud s\ +\ \int_0^t \sum_{k \geqslant 1 } \sigma_k(X^+(s,x_\eps))\, \ud \beta_k(s),\label{Ralongchar-blowup2}
\end{align}
\end{subequations}
for any stopping time $t<\tau$ a.s. Using \eqref{coeff-c}, the first term in the right-hand side \eqref{Ralongchar-blowup1} is bounded as follows:
\begin{equation}\label{Reps0}
R_0(x_\varepsilon)\ =\ - c(u_0(x_\varepsilon))\, \varepsilon^{-\nu - \gamma}\, \varphi'(x_0)\ \geqslant\ 3\, \delta\, \varepsilon^{-\nu - \gamma}.
\end{equation}
To bound from below the coefficient $c'(u)$ in factor of $R_{\chi^+}^2(s)$ in \eqref{Ralongchar-blowup1}, we will use \eqref{c'>0}. By \eqref{defRS} we have,
\begin{equation}\label{uplus}
u(s,X^+(s,x_\eps))\, -\, u_0({x_\varepsilon})\ =\ \int_0^s S_{\chi^+}(r)\, \ud r,
\end{equation}
and so 
\begin{equation}\label{ualongcharsquare}
\sup_{s \leqslant \eps^{\gamma} \wedge t } \left| u(s,X^+(s,x_\eps))\, -\, u_0({x_\varepsilon}) \right|^2\, 
\leqslant\ \varepsilon^{\gamma} \int_0^{\eps^{\gamma} \wedge t} S_{\chi^+}^2(r)\, \ud r.
\end{equation}
We use the estimate \eqref{RSalongcharac} to get 
\begin{equation}\label{ESplussquare}
\E \int_0^{\eps^{\gamma} \wedge t} S_{\chi^+}^2(r)\, \ud r\ \leqslant C\, \mathcal{E}(0)
\ +\ C\, \eps^{2 \gamma}.
\end{equation}
The initial energy satisfies 
\begin{equation}\label{ene0}
\mathcal{E}(0)\ =\ \int_\T \left( R_0^2\, +\, S_0^2 \right)(x)\, \ud x\ \leqslant\ 2\, c_2^2\, \varepsilon^{-2 \nu -2\gamma} \int_\T \left|\varphi' \left(\varepsilon^{-\alpha - \nu - \gamma} x  \right)\right|^2 \ud x\ \leqslant\ C\, \varepsilon^{\alpha-\nu-\gamma},
\end{equation}
so 
\begin{equation}\label{ESplussquare2}
\eps^{\gamma}\, \E \int_0^{\eps^{\gamma} \wedge t} S_{\chi^+}^2(r)\, \ud r\ \leqslant C\left( \varepsilon^{\alpha-\nu}\, +\, \varepsilon^{3 \gamma}\right),
\end{equation}
and we infer from \eqref{ualongcharsquare} and the Markov inequality that
\begin{equation}\label{ucloseustar}
\Pro \left( \sup_{s \leqslant \eps^{\gamma} \wedge t } \left| u(s,X^+(s,x_\eps))\, -\, u_0(x_\varepsilon) \right| \geqslant L/2 \right)
\ \leqslant\ C\left( \varepsilon^{\alpha-\nu}\, +\, \varepsilon^{3 \gamma}\right).
\end{equation}
For $\eps<\eps_0$ we have $\|u_0 - u^\star\|_{L^\infty} \leqslant \varepsilon^\alpha \| \varphi \|_{L^\infty} \leqslant L/2$, and then \eqref{c'>0} and \eqref{ucloseustar} imply 
\begin{equation}\label{positivecprimeOK}
\Pro \left( \inf_{s \leqslant \eps^{\gamma} \wedge t }  c'\left(u(s,X^+(s,x_\eps)) \right) \geqslant c'(u^\star)/2 \right)
\ \geqslant\ 1-C\left( \varepsilon^{\alpha-\nu}\, +\, \varepsilon^{3 \gamma}\right).
\end{equation}
Let us now estimate the terms in \eqref{Ralongchar-blowup2}. Using the Doob maximal inequality, Ito's isometry and the definition \eqref{defq} of the variance $q$, we have 
\begin{multline}\label{PM1}
\Pro \left( \sup_{s \leqslant \eps^{\gamma} \wedge t} \left| \int_0^s \sum_{k \geqslant 1 } \sigma_k(X^+(s,x_\eps))\, \ud \beta_k(s) \right| \geqslant \delta\, \varepsilon^{-\nu-\gamma} \right)\\
 \leqslant\ 
C\, \delta^{-2}\, \varepsilon^{2\nu + 2 \gamma}\, \E \left|\int_0^{\eps^{\gamma} \wedge t} \sum_{k \geqslant 1 } \sigma_k(X^+(s,x_\eps))\, \ud \beta_k(s) \right|^2
\ \leqslant\ C\ \varepsilon^{2 \nu+3\gamma}.
\end{multline}
By \eqref{coeff-c} and $|c'(u)| \leqslant c_3$, we also have
\begin{equation}\label{cprimeSplussquare}
\left|\int_0^{\eps^{\gamma} \wedge t} \left(\tilde{c}'(u)\, S^2 \right)\! (s,X^+(s,x_\eps))\, \ud s\right|
\ \leqslant\ C\, \int_0^{\eps^{\gamma} \wedge t} S_{\chi^+}^2(r)\, \ud r,
\end{equation}
so \eqref{ESplussquare2} and the Markov inequality give
\begin{equation}\label{PS1}
\Pro \left( \sup_{s \leqslant \eps^{\gamma} \wedge t}  \left| \int_0^{s} \left(\tilde{c}'(u)\, S^2 \right)\! (r,X^+(r,x_\eps))\, \ud r \right| \geqslant \delta\, \varepsilon^{-\nu - \gamma} \right)\, \leqslant\ C\left( \varepsilon^{\alpha}\, +\, \varepsilon^{\nu + 3 \gamma}\right).
\end{equation}
We gather the estimates \eqref{Reps0}, \eqref{positivecprimeOK}, \eqref{PM1}, \eqref{PS1} to obtain, based on \eqref{Ralongchar-blowup}, and for $\eps<\eps_0$, the inequality 
\begin{equation}\label{RplusblowupOK}
R_{\chi^+}(\eps^{\gamma} \wedge t)\ \geqslant \ \delta\, \varepsilon^{-\nu-\gamma}\ +\ \frac{c'(u^\star)}{2}\int_0^{\eps^{\gamma} \wedge t}   R_{\chi^+}^2(s)\, \ud s,
\end{equation}
with a probability greater than $1-C\eps^{\min \{\alpha- \nu, 3 \gamma \}}$. Since $\min \{\alpha- \nu, 3 \gamma \}>1$ and since the blow-up in \eqref{RplusblowupOK} occurs before the time $\frac{2\eps^{\nu + \gamma}}{c'(u^\star)\delta}$, we can in particular conclude to \eqref{PBlowup}.
\end{proof}

\section{Global-in-time regular solutions to the approximate system}\label{sec:gloabsol}

This section is devoted to the analysis of the approximate system \eqref{SVWEep}. As already explained when introducing the generalized system~\eqref{SVWEeptheta}, the cut-off function $\chi_\eps$ is in the class $W^{2,\infty}_\mathrm{loc}$, but not in $C^\infty$. Nevertheless, the non-linear estimate~\eqref{Composition} holds true for $s = 2$ and $F\in W^{2,\infty}_\mathrm{loc}$, and so the  results of Section~\ref{subsec:localregsol} as well, if we limit ourselves to $s = 2$. By the change of unknown
\begin{equation}\label{introPQeps}
P^\eps\ \eqdef\ R^\eps\ -\ \Phi^\varepsilon\, W, \qquad Q^\eps\ \eqdef\ S^\eps\ - \ \Phi^\varepsilon\, W,
\end{equation}
we transform the stochastic system \eqref{SVWEeptheta} into a system with random coefficients
\begin{subequations}\label{SVWEepthetaPQeps}
\begin{gather}\nonumber
P^\eps_t\ +\ c(u^\eps)\, P^\eps_x\ =\ \tilde{c}'(u^\eps) \left[|P^\eps|^2\, -\, |Q^\eps|^2 \, -\, \chi_\eps(P^\eps\ +\ \Phi^\varepsilon\, W)\, +\, 2\, (P^\eps\, +\,\Phi^\varepsilon\, W)\, \Theta^\eps\right]\\  \label{Peqeps}
-\, \left( c(u^\eps)\, \Phi^\varepsilon\, W \right)_x, \\ \nonumber
Q^\eps_t\ -\ c(u^\eps)\, Q^\eps_x\ =\ \tilde{c}'(u^\eps) \left[|Q^\eps|^2\, -\, |P^\eps|^2 \, -\, \chi_\eps(Q^\eps\ +\ \Phi^\varepsilon\, W)\, -\, 2\, (Q^\eps\, +\, \Phi^\varepsilon\, W)\, \Theta^\eps\right]\\ +\,  \left( c(u^\eps)\, \Phi^\varepsilon\, W \right)_x, \\ 
P^\eps(0,\cdot)\ =\ R_0 \ast \rho_\varepsilon, \qquad Q^\eps(0,\cdot)\ =\ S_0 \ast \rho_\varepsilon,
\end{gather}
\end{subequations}
where
\begin{multline}\label{udefepthetaPQeps}
u^\eps(t,x)\\
 =\ \mathcal{C}^{-1}\left\{ \mathcal{C} \left\{ \int_0^t {\textstyle \left(\frac{P^\eps\, +\, Q^\eps}{2}\, +\, \Phi^\varepsilon\, W  \right) (s,0)}\, \ud s\ +\ u^\eps_0(0) \right\} +\, \int_0^x \left[{\textstyle \frac{ Q^\eps\, -\, P^\eps}{2}}(t,y)\, -\, \Theta^\eps(t)\right] \ud y \right\},
\end{multline}
with (for $x \in [0,1)$)
\begin{equation}\label{psiepthetaPQeps}
\Theta^\eps(t)\ \eqdef\ \int_0^1 \frac{Q^\eps - P^\eps}{2}(t,y)\, \ud y, \qquad u^\eps_0(x)\ =\ \mathcal{C}^{-1} \left( \int_0^x \frac{S^\eps_0 - R^\eps_0}{2}\, \ud x \right).
\end{equation}
Then in the definition~\ref{def:regsolPQeps} of the solution $V$ stands for $V^\eps=(P^\eps,Q^\eps)$.

\begin{thm}[Regular solutions to the approximate system]\label{thm:glob-exist-ep} Let $\varepsilon>0$, $R_0,S_0\in H^2(\T)$ such that $\int_\T (S_0-R_0)\, \ud x = 0$. Then \eqref{SVWEep} admits a regular solution $(V^\varepsilon,\tau^\varepsilon)$ in $H^2(\T)$, associated to the initial data $R_0^\varepsilon,S_0^\varepsilon$, and defined up to the explosion time $\tau^\varepsilon$. 
\end{thm}

Our purpose is to prove that $\tau^\varepsilon=\infty$ a.s. We will use various estimates on the solution $(R^\eps,S^\eps)$ to show that blow-up cannot happen in finite time. Although we work at fixed $\varepsilon$ here, it will be useful, when possible, to prove some estimates uniform in $\varepsilon$ to prepare the limit $\varepsilon\to0$, established in the next Section~\ref{sec:Tightness} and Section~\ref{sec:Young}.

\paragraph{Notations.}\label{par:notation} We use the following convention: the letter $\mathtt{C}$ will denote any constant independent on $\eps$, that may otherwise depend on the constant $c_0$ in \eqref{Positivecprime00}, on the constant $c_1$ and $c_2$  in \eqref{coeff-c}, on $c_3$, on the value of the covariance $q$ in \eqref{defq}, and on the $L^2$-norms $\|R(0)\|_{L^2(\T)}$ and $\|S(0)\|_{L^2(\T)}$ only. So
\begin{equation}\label{constantC}
\mathtt{C}=\mathtt{C}\left(c_0,c_1,c_2,c_3,q_0,\|R(0)\|_{L^2(\T)},\|S(0)\|_{L^2(\T)}\right).
\end{equation}
The precise value of the constant may vary from lines to lines. An additional dependence on other parameters will be indicated, as $\mathtt{C}(T)$ in particular to indicate the additional dependence on a given final time $T$. 

\subsection{Energy estimates}\label{subsec:Energy-eps}

The aim of this section is to obtain energy estimates of the approximated system as in Proposition \ref{prop:GlobalEnergy} and Proposition \ref{prop:controlEnergies}. 

\begin{proposition}[Energy estimate]\label{prop:GlobalEnergyep} Let 
$(R^\varepsilon(t),S^\varepsilon(t))_{t<\tau^\varepsilon}$ be a regular solution to \eqref{SVWEep} in $H^2(\T)$, up to a stopping time $\tau^\varepsilon$. Let
\begin{equation}\label{TotalEnergyep}
\mathcal{E}^\varepsilon(t)\eqdef \|(R^\varepsilon,S^\varepsilon)(t)\|_{L^2(\T)}^2\ =\ \int_\T ((R^\varepsilon)^2\ +\ (S^\varepsilon)^2)(t,x)\, \ud x
\end{equation}
denote the total energy of the system. Then, for any stopping time $t<\tau^\varepsilon$ a.s.,
\begin{equation}\label{eq:TotalEnergyep}
\begin{aligned}
\mathcal{E}^\varepsilon(t)\ + 2 \int_0^t \int_\T \tilde{c}'(u^\varepsilon)\left[ R^\varepsilon\, \chi_\varepsilon(R^\varepsilon)\, +\, S^\varepsilon\, \chi_\varepsilon(S^\varepsilon) \right] \ud x\, \ud t\ 
&=\ \mathcal{E}^\varepsilon(0)\ +\ 2\, \|q^\varepsilon\|_{L^1}\, t\ +\, \mathcal{M}^\varepsilon(t)\\
&\leqslant\ \mathtt{C}\ +\ 2\, q_0\, t\ +\, \mathcal{M}^\varepsilon(t),
\end{aligned}
\end{equation}
where $\mathcal{M}^\varepsilon(t)$ is a martingale satisfying the bound
\begin{equation}\label{LPene}
\E\left[\left|\sup_{0\leqslant s\leqslant \tau^\eps\wedge T}\mathcal{M}^\varepsilon(s)\right|^p\right]\leqslant\mathtt{C}(p,T),
\end{equation}
for all $p\geqslant 1$. Moreover, we have
\begin{equation}\label{EXPene}
\E\left[\exp\left(\eta\sup_{0 \leqslant s\leqslant \tau^\eps\wedge T} \mathcal{E}^\varepsilon(s)\right)\right] \leqslant\ \mathtt{C}(T), 
\end{equation}
for all $\eta$ satisfying the smallness condition $4 q_0 \eta\leqslant 1$.
\end{proposition}

A direct consequence of \eqref{eq:TotalEnergyep}-\eqref{LPene} is the bound
\begin{equation}\label{LPeneE}
\E\left[\left|\sup_{0\leqslant s\leqslant \tau^\eps\wedge T}\mathcal{E}^\varepsilon(s)\right|^p\right]\leqslant\mathtt{C}(p,T),
\end{equation}
for all $p\geqslant 1$. We also state the following corollary to Proposition~\ref{prop:GlobalEnergyep}.

\begin{corollary}[Bound on the correction terms]\label{cor:boundpsi} Let 
$(R^\varepsilon(t),S^\varepsilon(t))_{t<\tau^\varepsilon}$ be a regular solution to \eqref{SVWEep} in $H^2(\T)$. The correction 
$\Theta^\varepsilon$ defined in \eqref{psieps} satisfies the bound
\begin{equation}\label{eq:boundpsi}
\E\left[\|\Theta^\varepsilon\|_{L^\infty(0,\tau^\eps\wedge T)}^p\right]\leqslant\mathtt{C}(p,T)\, \varepsilon^{1/2},
\end{equation}
for all $p\geqslant 1$, while
\begin{equation}\label{uteps0}
\E\left[\left\|u^\varepsilon_t-(R^\varepsilon+S^\varepsilon)/2\right\|^p_{L^1((0,\tau^\eps\wedge T);L^\infty(\T))}\right]\leqslant \mathtt{C}(T,p)\, \varepsilon^{1/2}.
\end{equation}
and
\begin{equation}\label{uteps01}
\E\left[\left\|u^\varepsilon_t-(R^\varepsilon+S^\varepsilon)/2\right\|^p_{L^\infty((0,\tau^\eps\wedge T) \times \T)}\right]\leqslant \mathtt{C}(T,p),
\end{equation}
for all $p\geqslant 1$.
\end{corollary}

\begin{proof}[Proof of Proposition~\ref{prop:GlobalEnergyep}] We use the fact that $\xi \chi_\varepsilon(\xi) \geqslant 0$ and integrate over $x\in\T$ in \eqref{sumsquare-theta} (where $\theta=\chi_\varepsilon$), to obtain \eqref{eq:TotalEnergyep} with
\begin{equation}\label{Meps-energy}
\mathcal{M}^\varepsilon(t)\ \eqdef\ 2\int_0^t \int_\T \sum_{k \geqslant 1 }\, (R^\varepsilon+S^\varepsilon)(s,x)\, \sigma_k^\varepsilon(x)\, \ud x\, \ud \beta_k(s).
\end{equation}
Using Itô's isometry, Jensen's inequality and the bound \eqref{qeps-q}, we then have
\begin{align} \nonumber
\E [\mathcal{M}^\varepsilon(t\wedge T)]^2\ &\leqslant\ 4\, q_0 \int_0^T \E \int_\T \left[ (R^\varepsilon+S^\varepsilon)^2 \right] \ud x\,   \ud s\ \leqslant\ 8\, q_0\, \mathcal{E}(0)\, T\ +\ 8\, q_0^2\, T^2\ \leqslant\ \mathtt{C}(T).
\label{EM22}
\end{align}
There remains to prove \eqref{EXPene} and \eqref{LPene}. Let $\eta\in (0,1]$. By \eqref{eq:TotalEnergyep}, we have
\begin{equation}\label{EXPenergy1}
\eta\sup_{0 \leqslant s\leqslant t\wedge T} \mathcal{E}^\varepsilon(s)\leqslant \mathtt{C}+\eta
\sup_{0 \leqslant s\leqslant t\wedge T}\mathcal{M}^\varepsilon(s).
\end{equation}
Let 
\begin{equation}
\dual{\mathcal{M}^\varepsilon}{\mathcal{M}^\varepsilon}(t)\ \eqdef\ 4 \int_0^t \sum_{k\geqslant 1 } \left|\int_0^1 \sigma_k(R^\varepsilon+S^\varepsilon)\, \ud x\right|^2 \ud s
\end{equation} 
denote the quadratic variation of $\mathcal{M}^\varepsilon$. We use an exponential martingale inequality:
\begin{equation}\label{expMineq}
\Pro\left(\sup_{0 \leqslant s\leqslant t\wedge T}\mathcal{M}^\varepsilon(s)
\geqslant \left(a+b\dual{\mathcal{M}^\varepsilon}{\mathcal{M}^\varepsilon}(t\wedge T)\right)\lambda\right)
\leqslant e^{-2a b \lambda^2},
\end{equation}
with $\lambda=1$ to obtain
\begin{equation}\label{OKZ}
\E\left[\exp\left(2\eta Z\right)\right]\, \leqslant\ \mathtt{C}(b),\quad Z\ \eqdef\ \sup_{0 \leqslant s\leqslant t\wedge T}\mathcal{M}^\varepsilon(s)-b \dual{\mathcal{M}^\varepsilon}{\mathcal{M}^\varepsilon}(t\wedge T),
\end{equation}
if $\eta\leqslant b/2$. By \eqref{EXPenergy1} and the Cauchy--Schwarz inequality, we have then
\begin{equation}\label{EXPenergy2}
\E\left[\exp\left(\eta\sup_{0 \leqslant s\leqslant t\wedge T} \mathcal{E}^\varepsilon(s)\right)\right]
\leqslant \mathtt{C} \left\{\E\left[\exp\left(2\eta Z\right)\right]\right\}^{1/2}
\left\{\E\left[\exp\left(2b\eta \dual{\mathcal{M}^\varepsilon}{\mathcal{M}^\varepsilon}(t\wedge T)\right)\right]\right\}^{1/2}.
\end{equation}
The bound \eqref{OKZ} therefore gives
\begin{equation}\label{EXPenergy3}
\E\left[\exp\left(\eta\sup_{0 \leqslant s\leqslant t\wedge T} \mathcal{E}^\varepsilon(s)\right)\right]
\leqslant \mathtt{C}(b) 
\left\{\E\left[\exp\left(2b\eta \dual{\mathcal{M}^\varepsilon}{\mathcal{M}^\varepsilon}(t\wedge T)\right)\right]\right\}^{1/2}.
\end{equation}
By the Cauchy--Schwarz inequality, we also have 
\begin{equation}\label{quadVarMeps}
\dual{\mathcal{M}^\varepsilon}{\mathcal{M}^\varepsilon}(t\wedge T)
\leqslant q_0\, \mathcal{E}^\varepsilon(t\wedge T)\leqslant q_0 \sup_{0 \leqslant s\leqslant t\wedge T} \mathcal{E}(s),
\end{equation}
so 
\begin{equation}\label{EXPenergy4}
\E\left[\exp\left(\eta\sup_{0 \leqslant s\leqslant t\wedge T} \mathcal{E}^\varepsilon(s)\right)\right]
\leqslant \mathtt{C}(b) 
\left\{\E\left[\exp\left(2b q_0 \eta\sup_{0 \leqslant s\leqslant t\wedge T} \mathcal{E}^\varepsilon(s)\right)\right]\right\}^{1/2}.
\end{equation}
We can let $t\uparrow\tau^\eps$ and choose $b=1/(2q_0)$ to conclude to \eqref{EXPene}, under the condition $4q_0\eta\leqslant 1$. Let us now establish \eqref{LPene}. By the Burkholder--Davis--Gundy inequality, we have
\begin{equation}
\E\left[\left|\sup_{0\leqslant s\leqslant t\wedge T}\mathcal{M}^\varepsilon(s)\right|^p\right]\leqslant
\mathtt{C}(p)\left\{\E\left[\dual{\mathcal{M}^\varepsilon}{\mathcal{M}^\varepsilon}(t\wedge T)^{p/2}\right]\right\}^{1/2}.
\end{equation}
We use \eqref{quadVarMeps}, the bound $|x|^{p/2}\leqslant \mathtt{C}(p)(1+e^{|x|/(8q_0)})$ and \eqref{EXPene} to conclude.
\end{proof}

\begin{proof}[Proof of Corollary~\ref{cor:boundpsi}] We assume $2\leqslant p$. We use the equations \eqref{psiprime}-\eqref{psiprimealphabeta}, where $\theta=\chi_\eps$ (this is why we will add an exponent $\varepsilon$ to the quantities $\alpha$ and $\beta$ below). By integration in \eqref{psiprime}, we obtain
\begin{equation}\label{psiintegrate}
\Theta^\varepsilon(t)\ =\ \int_0^t\alpha^\varepsilon(s)\, e^{-\int_s^t\beta^\varepsilon(\sigma) \ud\sigma}\, \ud s,
\end{equation}
which gives
\begin{equation}\label{psiintegratep}
\|\Theta^\varepsilon\|_{L^\infty(0,t\wedge T)}^p\ \leqslant\ \int_0^{t\wedge T}|\alpha^\varepsilon(s)|^p\, \ud s  \left(\int_0^{t\wedge T} e^{p'\int_s^t|\beta^\varepsilon(\sigma)| \ud\sigma}\, \ud s\right)^{p/p'},
\end{equation}
where $p'$ is the conjugate exponent to $p$, and then
\begin{multline}\label{psieps1}
\E\left[\|\Theta^\varepsilon\|_{L^\infty(0,t\wedge T)}^p\right]\\
\leqslant 
\left\{\E\left[\int_0^{t\wedge T}|\alpha^\varepsilon(s)|^{2p} \ud s\right]\right\}^{1/2}
\left\{\E\left[\left(\int_0^{t\wedge T} \exp\left(p'\int_s^{t\wedge T}|\beta^\varepsilon(\sigma)| \ud\sigma\right) \ud s\right)^{(2p)/p'}\right]\right\}^{1/2}.
\end{multline}
We estimate $\alpha^\eps$ in two manners. Let us introduce the notation
\begin{equation}\label{supM}
\bar{\mathcal{M}}^\eps(t)=\sup_{0\leqslant s\leqslant t}\mathcal{M}^\eps(s),
\end{equation}
where $\mathcal{M}^\varepsilon$ is the martingale \eqref{Meps-energy}. First we have
\begin{equation}\label{alphaeps1}
|\alpha^\varepsilon(s)|\leqslant c_3 \int_0^1 |\chi_\eps(R^\varepsilon)+\chi_\eps(S^\varepsilon)|(s,x)\, \ud x
\leqslant c_3\mathcal{E}(s)\leqslant \mathtt{C}(T)+c_3\bar{\mathcal{M}}^\varepsilon(s).
\end{equation}
We can also use the bound $\chi_\eps(R)\leqslant R\chi_\eps(R)\eps$ and the second term in the left-hand side of \eqref{eq:TotalEnergyep} to obtain
\begin{equation}\label{alphaeps2}
\int_0^{t\wedge T}|\alpha^\varepsilon(s)|\, \ud s\ \leqslant \eps \int_0^{t\wedge T}\int_0^1 \tilde{c}'(u^\varepsilon) \left[ R^\varepsilon\, \chi_\varepsilon(R^\varepsilon)\, +\, S^\varepsilon\, \chi_\varepsilon(S^\varepsilon) \right] \ud x\, \ud s
\leqslant \eps\left[\mathtt{C}(T)+\bar{\mathcal{M}}^\varepsilon(t\wedge T)\right].
\end{equation}
We combine \eqref{alphaeps1} with \eqref{alphaeps2} to obtain
\begin{equation}\label{alphaeps3}
\int_0^{t\wedge T}|\alpha^\varepsilon(s)|^{2p}\, \ud s
\leqslant \mathtt{C}(T)\left[1+\left|\bar{\mathcal{M}}^\varepsilon(s)\right|^{2p} \right] \eps,
\end{equation}
and the bound \eqref{LPene} gives
\begin{equation}\label{alphaeps4}
\E\left[\int_0^{t\wedge T}|\alpha^\varepsilon(s)|^{2p}\, \ud s\right]\leqslant \mathtt{C}(p,T)\, \eps.
\end{equation}
To estimate the remaining term in \eqref{psieps1}, we first note that 
\begin{equation}\label{betaeps0}
\left(\int_0^{t\wedge T} \exp\left(p'\int_s^{t\wedge T}|\beta^\varepsilon(\sigma)|\, \ud\sigma\right) \ud s\right)^{(2p)/p'}
\leqslant \mathtt{C}(p,T)\exp\left(\mathtt{C}(p,T) \sup_{0\leqslant s\leqslant t\wedge T}\mathcal{E}(\sigma)^{1/2}\right),
\end{equation}
and
\begin{equation}\label{betaeps1}
\mathtt{C}(p,T) \sup_{0\leqslant s\leqslant t\wedge T}\mathcal{E}(\sigma)^{1/2}
\leqslant \mathtt{C}(p,T)+(8q_0)^{-1}\sup_{0\leqslant s\leqslant t\wedge T}\mathcal{E}(s).
\end{equation}
Using \eqref{EXPene}, \eqref{betaeps0} and \eqref{betaeps1} yields, 
\begin{equation}\label{betaeps2}
\E\left[\left(\int_0^{t\wedge T} \exp\left(p'\int_s^{t\wedge T}|\beta^\varepsilon(\sigma)| \ud\sigma\right) \ud s\right)^{(2p)/p'}\right]\leqslant\mathtt{C}(p,T).
\end{equation}
We combine \eqref{alphaeps4} and \eqref{betaeps2} with \eqref{psieps1} to conclude to \eqref{eq:boundpsi} with $t$ instead of $\tau^\eps$. Then we let $t\uparrow\tau^\eps$ to get \eqref{eq:boundpsi}. The estimates \eqref{uteps0} and \eqref{uteps01} are a consequence of \eqref{utOK} (where all quantities are indexed by $\varepsilon$) and the bounds
\begin{equation}\label{utbyzeta}
\E\left[\left\|\zeta^\eps\right\|^p_{L^1((0,\tau^\eps\wedge T) \times \T)}\right]\leqslant\mathtt{C}(T,p)\, \eps^{1/2},
\end{equation}
and
\begin{equation}\label{utbyzeta2}
\E\left[\left\|\zeta^\eps\right\|^p_{L^\infty((0,\tau^\eps\wedge T);L^1(\T))}\right]\leqslant\mathtt{C}(T,p),
\end{equation}
where
\begin{equation}\label{zetaeps}
\zeta^\varepsilon=\tilde{c}'(u^\varepsilon)\left[\frac{\chi_\varepsilon(R^\varepsilon)-\chi_\varepsilon(S^\varepsilon)}{2}+(R^\varepsilon+S^\varepsilon)\Theta^\varepsilon\right].
\end{equation}
This follows easily from the bounds \eqref{alphaeps1}-\eqref{alphaeps2} and \eqref{eq:TotalEnergyep}-\eqref{LPene}-\eqref{eq:boundpsi}.
\end{proof}

Let $X^{\varepsilon,\pm}$ denote the flow associated to the vector field $\pm c(u^\varepsilon(t,\cdot))$ (characteristic curves): 
\begin{equation}\label{flowXpmep}
\frac{\ud\;}{\ud t} X^{\varepsilon,\pm}(t,x)\ =\ \pm c(u^\varepsilon(t,X^{\varepsilon,\pm}(t,x))), \qquad X^{\varepsilon,\pm}(0,x)\ =\ x,\qquad x\in\T,\quad t\in [0,\tau^\varepsilon).
\end{equation}
In parallel with Proposition \ref{prop:controlEnergies}, we have the following result.

\begin{proposition}[Control along the opposite characteristic curve]\label{prop:controlEnergiesep} 
Consider a regular solution $(R^\varepsilon(t),S^\varepsilon(t))_{t<\tau^\varepsilon}$ to \eqref{SVWEep} up to a stopping time $\tau^\varepsilon$.
Let $X^{\varepsilon,\pm}$ denote the characteristic curves \eqref{flowXpmep}. There is a constant $\mathtt{C}$ as in \eqref{constantC}, such that: for all $x_1\in\T$ and $x_2 = x_1 + 2c_2 T$, 
\begin{align}\label{RSalongcharacep}
\int_{0}^{t\wedge T} \left[(R^\varepsilon)^2(s,{X^-(s,x_2)})\, +\,  (S^\varepsilon)^2(s,{X^+(s,x_1)}) \right] \ud s\ \leqslant \ \mathtt{C}\ +\ \mathtt{C}\, M^\varepsilon_{x_1}(t\wedge T),
\end{align} 
for any stopping time $t$ such that $t<\tau^\varepsilon$ a.s., where $M^\varepsilon_{x_1}(t)$   satisfies the bound 
\begin{equation}\label{MLinfty}
\E\left[ \sup_{s \in [0, \tau^\eps\wedge T)} \sup_{\substack{x_1}} |M^\varepsilon_{x_1} (s)|^2\right] \ \leqslant\ \mathtt{C}(T).
\end{equation}
\end{proposition}

\begin{proof} The proof of Proposition \ref{prop:controlEnergiesep} is similar to the one of Proposition \ref{prop:controlEnergies}. 
The constant in \eqref{MtildeM} is independent on $x_1$, then we obtain \eqref{MLinfty} by following the proof of \eqref{boundMx12} and using \eqref{LPeneE}.
\end{proof}
\subsection{Global solutions} 

We can now use the estimates established in Section~\ref{subsec:Energy-eps} to show that the regular solutions of \eqref{SVWEep} given by Theorem \ref{thm:glob-exist-ep} exist globally in time almost surely. 

\begin{thm}\label{thm:globalsolep}
Let $\varepsilon>0$, $R_0,S_0\in L^2(\T)$ such that $\int_\T (S_0-R_0)\, \ud x =0$. Let $\tau^\varepsilon$ be the explosion time of the solution of \eqref{SVWEep} given by Theorem \ref{thm:glob-exist-ep}. Then $\Pro (\tau^\varepsilon = \infty) = 1$.
\end{thm}

\begin{proof}[Proof of Theorem~\ref{thm:globalsolep}] By \eqref{sigmas1}, \eqref{WHsas} (satisfied by $\Phi^\varepsilon W$) and 
the bounds \eqref{MLinfty} and \eqref{eq:boundpsi}, there exists $\bar{\Omega}^\varepsilon$ such that $\Pro(\bar{\Omega}^\varepsilon)=1$ and, for all $\omega\in \bar{\Omega}^\varepsilon$, for all $T>0$, 
\begin{equation}\label{WHsasep2}
\Phi^\varepsilon\, W \in C([0,T], H^{3}(\T)), \quad
M^\varepsilon_{\cdot} \in L^\infty ( [0,\tau^\varepsilon\wedge T) \times \R ),\quad
\Theta^\eps\in L^\infty(0,\tau^\eps\wedge T).
\end{equation} 
Let $\omega \in \bar{\Omega}^\varepsilon$ be fixed. We have to establish the bound 
\begin{equation}\label{OKglobaleps}
\sup_{t \in [0,\tau^\varepsilon\wedge T)} \|V^\varepsilon(t)\|_{L^\infty(\T)} < \infty,
\end{equation}
for all $T>0$. We remark that $\xi^2 - \chi_\varepsilon(\xi) \geqslant 0$, so the equations \eqref{SVWEepthetaPQeps} imply 
\begin{equation}\label{Peps-neg}
P^\varepsilon_t\ +\ c(u^\varepsilon)\, P^\varepsilon_x\ \geqslant\ - \tilde{c}'(u^\varepsilon)\left[  (S^\varepsilon)^2\, +\, 2\, ([P^\eps]^-\, +\, |\Phi^\varepsilon\, W|)\, |\Theta^\eps|\right]
 -\, c(u^\varepsilon)\, (\Phi^\varepsilon\, W)_x,
\end{equation}
where $s^-=\max(-s,0)$ is the negative part of a real $s\in\R$. We integrate along the characteristics in \eqref{Peps-neg} (\textit{cf.} \eqref{hRalongchar} for instance) and use \eqref{RSalongcharacep} to get
\begin{equation}
\|[P^\varepsilon]^-(t)\|_{L^\infty(\T)}\leqslant \mathtt{C}(\omega,\eps,T)\left(1+\int_0^t \|[P^\varepsilon]^-(s)\|_{L^\infty(\T)}\ud s\right), 
\end{equation}
for $t\in [0,\tau^\eps\wedge T]$, which gives
\begin{equation}\label{Peps-negOK}
\sup_{t \in [0,\tau^\varepsilon\wedge T)}\|[P^\varepsilon]^-(t)\|_{L^\infty(\T)}\leqslant \mathtt{C}(\omega,\eps,T), 
\end{equation}
by the Gr\"onwall lemma. The bounds on the positive part of $P^\varepsilon$ is established in a similar manner: it follows from \eqref{SVWEepthetaPQeps} and the bound $R^2-\chi_\eps(R)\leqslant 2\eps^{-1}R^+$ that 
\begin{equation}\label{Peps-pos}
P^\eps_t\ +\ c(u^\eps)\, P^\eps_x\ \leqslant\ \mathtt{C}(\omega,\eps,T)\left[1\ +\ [P^\varepsilon]^+\right].
\end{equation}
By the comparison principle (or integration along the characteristics) and the Gr\"onwall lemma, \eqref{Peps-pos} implies 
\begin{equation}\label{Peps-posOK}
\sup_{t \in [0,\tau^\varepsilon\wedge T)}\|[P^\varepsilon]^+(t)\|_{L^\infty(\T)}\leqslant \mathtt{C}(\omega,\eps,T).
\end{equation}
We have a of course the bounds similar to \eqref{Peps-negOK}-\eqref{Peps-posOK} on $Q^\varepsilon$, so \eqref{OKglobaleps} is indeed satisfied.
\end{proof}

\subsection{\texorpdfstring{$L^{3^-}$}{} estimates}

In this section we will establish the following almost $L^3$ estimates with weights on $R^\varepsilon$ and $S^\varepsilon$.
\begin{pro}[Estimates in $L^{3^-}_{t,x}$]\label{prop:L3estimate}
Let $\alpha \in [0,1)$. We have the bound
\begin{equation}\label{alpha+2}
\sup_{\varepsilon}\, \E\, \int_{0}^{T} \int_\T c'(u^\varepsilon)
\left[ |R^\varepsilon|^{2+\alpha} + |S^\varepsilon|^{2+\alpha} \right]
  \ud x\, \ud t\ \leqslant\mathtt{C}(T,\alpha),
\end{equation}
for all $T\geqslant 0$.
\end{pro}


\begin{proof}[Proof of Proposition~\ref{prop:L3estimate}] Let $h\colon\R\to\R$ be the non-decreasing function of class $W^{2,\infty}$ defined as
\begin{equation*}
h(R)\ \eqdef\ \int_0^R \left( r^2\, +\, 1 \right)^\frac{\alpha}{2} \ud r.
\end{equation*}
The function $h$ satisfies the growth conditions 
\begin{equation}\label{subalphah}
|h(R)|\ \leqslant\ C\, (1\ +\ |R|^2)^{\frac{\alpha+1}{2}},\quad |h'(R)|\ \leqslant\ C\, (1\ +\ |R|)^\alpha,\quad |h''(R)|\ \leqslant\ C,\quad R\in\R.
\end{equation}
We sum up the two equations \eqref{ReqepthetaItoConservative} and \eqref{SeqepthetaItoConservative} with $\theta=\chi_\eps$ and integrate on $\T$ to obtain (using $h'(R)\chi_\eps(R)\geqslant 0$ and the energy estimate \eqref{LPeneE})
\begin{equation}\label{EstimL3B0}
\E\, \int_{0}^{T} \int_\T c'(u^\varepsilon)\, \Delta_1(R^\varepsilon,S^\varepsilon)\, \ud x\, \ud t\, \leqslant\, \mathtt{C}(T)
\, +\, \mathtt{C}(T)\,\E\, \int_{0}^{T} |\Theta^\eps(t)|\int_{\T} (|R^\eps|^2\ +\ |S^\eps|^2)^{\frac{1+\alpha}{2}}\, \ud x\,  \ud t,
\end{equation}
where
\begin{equation}
\Delta_1(R,S)\ :=\ (R\ -\ S)\, (B_h(R,S)\ -\ B_h(S,R)),
\end{equation}
and $B_h(R,S)$ is defined in \eqref{defBh}.
By H\"older's inequality, \eqref{eq:boundpsi} and \eqref{LPeneE}, we can bound the last term in \eqref{EstimL3B0} and get
\begin{equation}\label{EstimL3B}
\E\, \int_{0}^{T} \int_\T c'(u^\varepsilon)\, \Delta_1(R^\varepsilon,S^\varepsilon)\, \ud x\, \ud t\ \leqslant \mathtt{C}(T,\alpha).
\end{equation}
Similarly, starting from \eqref{RCeqepthetaIto}-\eqref{SCeqepthetaIto} and using the estimates \eqref{uteps01} and \eqref{LPeneE}
we can establish the bound
\begin{equation}\label{EstimL3F}
\E\, \int_{0}^{T} \int_\T \frac{c'(u^\varepsilon)}{c(u^\varepsilon)}\, \Delta_2(R^\varepsilon,S^\varepsilon)\, \ud x\, \ud t\ \leqslant \mathtt{C}(T,\alpha),
\end{equation}
where
\begin{equation}
\Delta_2(R,S)\ :=\ (R\ +\ S)\, (F_h(R,S)\ +\ F_h(S,R)),
\end{equation}
and $F_h(R,S)$ is defined in \eqref{defFh}.

We claim that 
\begin{equation}\label{belowPsiL3}
(R\ -\ S)^2\, (|R|^\alpha\ +\ |S|^\alpha)\, \leqslant \mathtt{C}(\alpha)\, \Delta_1(R,S),\quad (R\ +\ S)^2\, (|R|^\alpha\ +\ |S|^\alpha)\, \leqslant \mathtt{C}(\alpha)\, \Delta_2(R,S).
\end{equation}
Summing up the two estimates \eqref{EstimL3B} and \eqref{EstimL3F} and using \eqref{coeff-c} and \eqref{belowPsiL3} one then gets \eqref{alpha+2}. Let us now give a proof of the first inequality in \eqref{belowPsiL3} (the second inequality is obtained by similar arguments). Using the definition of $B_h(R,S)$ and the monotonicity of $\xi \mapsto \xi / h'(\xi)$ we obtain 
\begin{align*}
\Delta_1(R,S)\, &\geqslant\, {\textstyle \frac{1-\alpha}{1+\alpha}} (R - S) \left[ R\, h'(S)\, -\, S\, h'(R) \right] + \, (R - S) \left[2\, h(R)\, -\, R\, h'(R)\, -\, 2\, h(S)\, +\, S\, h'(S) \right] \\
& =\, {\textstyle \frac{1-\alpha}{1+\alpha}} (R - S)^2 \left[ h'(R)+ h'(S) \right] + 2 (R - S) \left[ h(R) - {\textstyle \frac{1}{1+\alpha}} R h'(R) - h(S) + {\textstyle \frac{1}{1+\alpha}} S h'(S) \right]\! .
\end{align*}
The result follows from the monotonicity of $\xi \mapsto h(\xi) - \xi\, h'(\xi)/(\alpha +1)$ and from the bound from below $h'(\xi) \geqslant |\xi|^\alpha$.
\end{proof}

\subsection{One-sided entropy estimates}\label{sec:Oleinik} 

The central result in this section is Proposition~\ref{prop:one-sided-estimates}. We will need first a control on $c'(u(t))$ for small times. More precisely, recall that 
\begin{equation}
c_0\ =\ \inf_{x \in \T} c'(u_0(x))\ >\ 0
\end{equation}
by \eqref{Positivecprime00}. For $\eps$ sufficiently small, we have then 
\begin{equation}\label{Positivecprime-eps0}
\inf_{x \in \T} c'(u^\varepsilon(0,x))\ \geqslant\ c_0/2.
\end{equation}
Consider the stopping time 
\begin{equation}\label{bartdef}
\bar{t}^\varepsilon\ \eqdef\ \inf \left\{t \geqslant 0, \inf_{x \in \T} c'(u^\varepsilon(t,x))\, \leqslant\, c_0/4 \right\}.
\end{equation}
We have the following estimate on $\bar{t}^\varepsilon$.

\begin{lem}\label{lem:tepspos} We have
\begin{align}\label{bart}
\E \left[ (\bar{t}^\varepsilon)^{-p} \right]\, \leqslant\ \mathtt{C}(p),
\end{align}
for all $p\geqslant 1$.
\end{lem}

\begin{proof}[Proof of Lemma~\ref{lem:tepspos}] By continuity of the map $u \mapsto c'(u)$ and by \eqref{Positivecprime-eps0}, there exists $\eta>0$ such that 
\begin{equation}\label{Ldef2}
\|v-u^\varepsilon(0,\cdot)\|_{L^\infty(\T)}\ \leqslant\ \eta \qquad \implies \qquad
\inf_{x \in \T} c'(v(x))\ >\ c_0/4,
\end{equation}
for any $v\in C(\T)$. In particular, 
\begin{equation}\label{finitebart}
\eta\ \leqslant\ \sup_{s \leqslant \bar{t}^\varepsilon} \|u^\varepsilon(s) - u^\varepsilon_0\|_{L^\infty(\T)}.
\end{equation}
We apply the estimate
\[
\|v\|_{L^\infty(\T)}^2\ \leqslant\ \|v\|_{L^2(\T)}\, \|\partial_x v\|_{L^2(\T)},
\]
valid for any $v\in H^1(\T)$, to $v=u^\varepsilon(t) - u^\varepsilon_0$. Writing 
\begin{equation}
u^\varepsilon(t) - u^\varepsilon_0=\int_0^t u^\varepsilon_t(s)\, \ud s
\end{equation}
yields 
\begin{equation}
\|u^\varepsilon(t) - u^\varepsilon_0\|_{L^\infty(\T)}^2\ \leqslant\ 2\, t\, \|u^\varepsilon_t\|_{L^\infty((0,t),L^2(\T))}\, 
\|\partial_x u^\varepsilon\|_{L^\infty((0,t);L^2(\T))},
\end{equation}
and, by \eqref{finitebart}
\begin{equation}
(\bar{t}^\varepsilon)^{-1}\ \leqslant\ (\bar{t}^\varepsilon\wedge T)^{-1}\ \leqslant\ 2\,\eta^{-1}\, \|u^\varepsilon_t\|_{L^\infty((0,T),L^2(\T))}\, 
\|\partial_x u^\varepsilon\|_{L^\infty((0,T);L^2(\T))},
\end{equation}
where $T$ is an arbitrary positive time (we may take $T=1$). We use the identities \eqref{cuxthetaeps}-\eqref{utOK} and the bounds \eqref{eq:boundpsi}, \eqref{uteps01}, \eqref{LPeneE} to conclude to \eqref{bart}.
\end{proof}


We can now state the following one-sided estimate, where $s^-=-\min \{s,0\}$ denotes the negative part of a real $s$.

\begin{pro}[Oleinik's estimate]\label{prop:one-sided-estimates} We have
\begin{equation}\label{Oleinik}
\E\left[ \left\| \left(R^\varepsilon\right)^-(\tau \wedge T,\cdot) \right\|_{L^\infty(\T)}^p\right]\ +\ \E\left[ \left\| \left(S^\varepsilon\right)^-(\tau \wedge T,\cdot) \right\|_{L^\infty(\T)}^p\right]\ \leqslant\ \mathtt{C}(p,T) \left(
1\, +\, \E\left[\frac{1}{\left(\tau \wedge T\right)^p} \right]\right),
\end{equation}
for all $p\in [1,2]$, for any time $T>0$ and any stopping time $\tau >0$ a.s.
\end{pro}

\begin{proof}[Proof of Proposition~\ref{prop:one-sided-estimates}] For simplicity, we will drop the superscript $\eps$ on the quantities $R$, $S$, etc. in the proof. Let us consider the equation for $P=R-\Phi W$: use
\begin{equation}\label{EqROleinik1}
R^2\ =\ (P+\Phi W)^2\ \geqslant\ \half\, P^2\, -\, |\Phi W|^2,
\end{equation}
and 
\begin{equation}\label{EqROleinik2}
2\, R\, \Theta \ =\ 2\, (P+\Phi W)\, \Theta \ \geqslant\ - \fourth\, P^2\, -\, C(|\Theta |^2\, +\, |\Phi W|^2),
\end{equation}
to obtain
\begin{equation}\label{EqROleinik3}
P_t\ +\ c(u)\, P_x\ \geqslant\
{\textstyle \frac{\tilde{c}'(u)}{4}} P^2 -C S^2 -C \chi_\eps(P+\Phi W)-C K_T^2,
\end{equation}
where  
\begin{equation}\label{EqROleinik4}
K_T:=\left\{\sup_{t\in[0,T]} \left(\|\Phi W(t)\|_{W^{1,\infty}(\T)}^2+|\Theta (t)|^2\right)\right\}^{1/2}.
\end{equation}
We deduce from \eqref{EqROleinik3} that
\begin{equation}\label{cutOleinik}
H_t\ +\ c(u)\,  H_x\ \leqslant\ -{\textstyle \frac{\tilde{c}'(u)}{8}} H^2 +CS^2+C K_T^2,
\end{equation}
where $H\eqdef (P+K_T)^-$. Integrate along the characteristics: for $t\leqslant \tau\wedge T$,
\begin{equation}\label{charOleinik1}
H(t,X^+(t,x))\
\leqslant\ H(s,X^+(s,x))\ -\int_s^t {\textstyle \frac{\tilde{c}'(u)}{8}} H^2(r,X^+(r,x))\, \ud r\ +\ A_T.
\end{equation}
where
\begin{equation}\label{charOleinik2}
A_T\ \eqdef\ C \int_0^{\tau\wedge T} S^2(s,X^+(s,x)))\, \ud s\ +\ C\, T\, K_T^2.
\end{equation}
On the interval $[0,\bar{t}^\eps]$, assuming $\tau\wedge T\geqslant\bar{t}^\eps$, we have
\begin{equation}\label{charOleinik3}
H(\sigma ,X^+(\sigma,x))\
\leqslant\ H(s,X^+(s,x))\, -{ \textstyle \frac{c_0}{32}}\int_s^\sigma H^2(r,X^+(r,x))\, \ud r\ +\ A_T,\quad [s,\sigma]\subset [0,\bar{t}^\eps].
\end{equation}
We deduce from \eqref{charOleinik3} that
\begin{equation}\label{OKbartOleinik}
H(t,X^+(t,x))\ \leqslant\ {\textstyle \frac{32}{c_0 t}}\ +\ A_T,\quad\forall t\in (0,\bar{t}^\eps].
\end{equation}
Indeed, assume that \eqref{OKbartOleinik} is not realized, and let $t_1\in (0,\bar{t}^\eps]$ be such that
\begin{equation}\label{notOKbartOleinik}
H(t_1,X^+(t_1,x))\ >\ \frac{32}{c_0 t_1}\ +\ A_T.
\end{equation}
Define
\begin{equation}
t_0\ \eqdef\ \sup\left\{t\in (0,t_1]; H(t,X^+(t,x)) \leqslant {\textstyle \frac{32}{c_0 t}} \right\}.
\end{equation}
We have $t_0>0$ since $H(0,x)<\infty$ and, by continuity,
\begin{equation}
H(t_0,X^+(t_0,x))\ =\ {\textstyle \frac{32}{c_0 t_0}},\quad H(t,X^+(t,x))\ >\ {\textstyle \frac{32}{c_0 t}},\quad \forall t\in(t_0,t_1].
\end{equation}
Taking $s=t_0$ and $\sigma=t_1$ in \eqref{charOleinik3} therefore gives
\begin{equation}\label{OKOKbartOleinik}
H(t_1,X^+(t_1,x))\ \leqslant\ {\textstyle \frac{32}{c_0 t_0}}\ +\, \left[{\textstyle \frac{32}{c_0 r}}\right]_{t_0}^{t_1} +\, A_T\ =\ {\textstyle \frac{32}{c_0 t_1}}\ +\ A_T,
\end{equation}
which is in contradiction with \eqref{notOKbartOleinik}. Whatever the value of $\tau$ may be, infer from \eqref{OKbartOleinik} that, for all $x\in\T$,
\begin{equation}\label{OKbartOleiniktau}
H(\tau\wedge T\wedge \bar{t}^\eps,x)\ \leqslant\ \frac{32}{c_0 \tau\wedge T\wedge \bar{t}^\eps}\ +\ A_T.
\end{equation}
If $\tau\wedge T>\bar{t}^\eps$, we complete the argument as follows: we use \eqref{charOleinik1} between the times $\bar{t}^\eps$ and $\tau\wedge T$ and exploit \eqref{OKbartOleinik} with $t=\bar{t}^\eps$ to get
\begin{equation}\label{OKbartOleiniktau2}
H(\tau\wedge T,x)\ \leqslant\ {\textstyle \frac{32}{c_0 \bar{t}^\eps}}\ +\ 2A_T.
\end{equation}
Finally, we deduce from \eqref{OKbartOleiniktau} and \eqref{OKbartOleiniktau2} that
\begin{equation}\label{OKbartOleiniktau3}
H(\tau\wedge T\wedge \bar{t}^\eps,x)\ \leqslant\ \frac{32}{c_0 \tau\wedge T\wedge \bar{t}^\eps}\ +\ 2\, A_T,
\end{equation}
and thus, using \eqref{WHs}, \eqref{eq:boundpsi} and Proposition~\ref{prop:controlEnergiesep} to estimate $\E\left[A_T^p\right]$, 
\begin{equation}\label{OleinikH}
\E\left[ \left\| H(\tau \wedge T,\cdot) \right\|_{L^\infty(\T)}^p\right]\ \leqslant\ \mathtt{C}(p,T) \left(
1\, +\, \E\left[\frac{1}{\left(\tau \wedge T\right)^p} \right]\right).
\end{equation}
Then we conclude with the inequality $R^-\leqslant H+2K_T$.
\end{proof}

\section{First limit elements}\label{sec:Tightness}
It follows from \eqref{ReqepthetaItoConservative}-\eqref{SeqepthetaItoConservative} that $(R^\eps,S^\eps)$ satisfies the following set of equations 
\begin{gather}\nonumber
\ud h(R^\varepsilon)\ +\ (c(u^\varepsilon)\, h(R^\varepsilon))_x\, \ud t\ =\ h'(R^\varepsilon)\, \Phi^\eps\, \ud W\ +\  2\, \tilde{c}'(u^\varepsilon)\, \Theta^\varepsilon\, (R^\varepsilon\, h'(R^\varepsilon)\, -\, 2\, h(R^\varepsilon))\, \ud t  \\ \label{ReqepthetaItoConservativeEPS}
=\ \tilde{c}'(u^\varepsilon) \left[(S^\varepsilon\, -\, R^\varepsilon)\, B_h(R^\varepsilon,S^\varepsilon)\, -\, h'(R^\varepsilon)\, \chi_\varepsilon(R^\varepsilon)\right] \ud t\, +\, \half\, q^\varepsilon\, h''(R^\varepsilon) \, \ud t,
\end{gather}
and
\begin{gather}\nonumber
\ud h(S^\varepsilon)\ -\ (c(u^\varepsilon)\, h(S^\varepsilon))_x\, \ud t\ +\  h'(S^\varepsilon)\,\Phi^\eps\, \ud W -\, 2\, \tilde{c}'(u^\varepsilon)\, \Theta^\varepsilon\, (S^\varepsilon\, h'(S^\varepsilon)\, -\, 2\, h(S^\varepsilon))\, \ud t \\  \label{SeqepthetaItoConservativeEPS}
=\ \tilde{c}'(u^\varepsilon) \left[(R^\varepsilon\, -\, S^\varepsilon)\, B_h(S^\varepsilon,R^\varepsilon)\, -\, h'(S^\varepsilon)\, \chi_\varepsilon(S^\varepsilon)\right] \ud t\, +\, \half\, q^\varepsilon\, h''(S^\varepsilon)\, \ud t.
\end{gather}
In this section, we use the estimates derived in the previous part to obtain a first set of limit elements, including random Young measures (see Proposition~\ref{pro:tightness}). Then we pass to the limit in \eqref{ReqepthetaItoConservativeEPS}-\eqref{SeqepthetaItoConservativeEPS} to obtain a first set of limit equations with defect measures. In the next section~\ref{sec:Young}, we will complete the analysis of the limit elements, prove that the Young measures reduce to Dirac masses, and establish the limit equation for good. 

The plan of this section is the following one: some elements on random Young measures are given in Section~\ref{subsec:YoungMeasure}. Taking the limit of the stochastic integral in \eqref{ReqepthetaItoConservativeEPS}-\eqref{SeqepthetaItoConservativeEPS} is done as usual by strengthening the probabilistic mode of convergence, passing from convergence in law to a.s.-convergence, at the expense of a modification of the probability spaces and measures. We use the 
Skorokhod--Jakubowski theorem to that effect, explained in Section~\ref{subsec:QPolish}. Compactness results are established in Section~\ref{subsec:tight} and the limit equation with defect measure derived in Section~\ref{subsec:limdefectmeas}.

%
%
%
%
%
%
%

\subsection{Young measures}\label{subsec:YoungMeasure}

Let $T>0$ and let $\mathcal{P}_T([0,T] \times \T \times \R^2)$ denote the space of (non-negative) measures $\mu$ on $[0,T] \times \T \times \R^2$ having total measure 
\begin{equation}\label{TotalMassYoungMeasure}
\mu([0,T] \times \T \times \R^2)=T.
\end{equation} 
We start this section by defining Young measures on $[0,T] \times \T$.

\begin{definition}
Let $\mu \in \mathcal{P}_T([0,T] \times \T \times \R^2)$, we say that $\mu$ is a Young measure on $[0,T] \times \T $ (with state space $\R^2$) if for almost all $(t,x) \in [0,T] \times \T$ there exists a probability measure $\mu_{t,x}$ on $\R^2$ such that 
\begin{itemize}
\item the map $(t,x) \mapsto \dual{\mu_{t,x}}{\varphi}$ is measurable for any $\varphi \in C_b(\R^2)$
\item $\dual{\mu}{\psi} = \int_{[0,T] \times \T} \int_{\R^2} \psi(t,x,\xi)\, \ud \mu_{t,x}(\xi,\eta)\, \ud x\, \ud t $ for any $\psi \in C_b([0,T] \times \T \times \R^2)$.
\end{itemize}
We denote by $\mathscr{Y}$ the space of Young measures on $[0,T] \times \T$ with state space $\R^2$.
\end{definition}

The space $\mathscr{Y}$ is equipped with the topology of weak convergence in $\mathcal{P}_T([0,T] \times \T \times \R^2)$, i.e. 
\begin{equation}\label{CVweakmu}
\mu^\varepsilon\ \to\ \mu \quad \mathrm{iff} \quad \dual{\mu^\varepsilon}{\psi}\ \to\ \dual{\mu}{\psi}\ \forall \psi \in C_b([0,T] \times \T \times \R^2).
\end{equation} 
The weak convergence in the space $\mathcal{P}_T([0,T] \times \T \times \R^2)$ (and also $\mathscr{Y}$) coincides with the convergence with respect to the Prokhorov metric \cite[p.72]{BillingsleyBook}. Then,  $\mathscr{Y}$ is a Polish (separable and completely metrizable) space, \cite[Theorem~6.8]{BillingsleyBook}.
In order to obtain tightness in the space $\mathscr{Y}$, one needs to identify some compact subsets of $\mathscr{Y}$. 
For that purpose, we recall the following result (see Proposition 4.1 in \cite{BerthelinVovelle19})

\begin{pro}\label{CompactYoung}
Let $f \in C (\R^2)$ be a non-negative function satisfying 
\begin{equation}\label{finfinite}
\lim_{|(\xi,\eta)| \to \infty} f(\xi,\eta)\ =\ \infty.
\end{equation}
Let also $C>0$, then the set 
\begin{equation}\label{Kicdef}
\left\{\mu \in \mathscr{Y},\ \int_{[0,T] \times \T \times \R^2} f(\xi,\eta)\, \ud \mu(t,x,\xi,\eta) \leqslant C \right\}
\end{equation}
is a compact subset of $\mathscr{Y}$.
\end{pro}

\begin{remark}[Convergence in $\mathscr{Y}$]\label{rk:CVYoung} If $\mu$ and $\mu_1,\mu_2,\dotsc$ are Young measures in the set \eqref{Kicdef}, and if \eqref{finfinite} is satisfied, then to check the weak convergence $\mu_n\to\mu$ (which corresponds to convergence against continuous and bounded functions on $\T\times[0,T]\times\R^2$, \textit{cf.} \eqref{CVweakmu}), it is sufficient to check convergence against functions in $C_0(\T\times[0,T]\times\R^2)$ (functions which tend to $0$ at infinity). Actually, we can even conclude that 
\begin{equation}\label{CVmun}
\int_0^T\int_\T\int_{\R^2}\varphi(t,x)g(\xi,\eta)\, \ud \mu_n\to \int_0^T\int_\T\int_{\R^2}\varphi(t,x)g(\xi,\eta)\, \ud \mu,
\end{equation}
for all $\varphi\in C([0,T]\times\T)$ and $g\in C(\R^2)$ satisfying $|g(\xi,\eta)|\leqslant C f(\xi,\eta)^\delta$, where $\delta\in[0,1)$, since for any $\mu$ in the set \eqref{Kicdef}, and for any cut-off function $\bar{\chi}_R\in C([0,\infty))$ such that $\chi_R\equiv 1$ on the interval $[0,R]$, we have the uniform bound
\begin{equation}\label{Boundmun}
\left|\int_0^T\int_\T\int_{\R^2}\varphi(t,x)\left[ g(\xi,\eta)-g(\xi,\eta)\bar{\chi}_R(f(\xi,\eta))\right]\, \ud \mu\right|\, \leqslant\ \|\varphi\|_{C([0,T]\times\T)}\, C\, R^{\delta-1}.
\end{equation}
\end{remark}

In the same spirit as Remark~\ref{rk:CVYoung}, and following the proof of Proposition 4.4 in \cite{BerthelinVovelle19} we can establish the following result.

\begin{pro}\label{pro:young}
Let $[t_1,t_2] \subset [0,T]$ and let $(a_n)_n$ be a sequence of non-negative $C([t_1,t_2] \times \T)$-valued random variables. Let also $(\mu_n)_n$ be a sequence of random Young measures ($\mathscr{Y}$-valued random variables) satisfying
\begin{equation}\label{afmu_n_estimate}
C(f)\ \eqdef\ \sup_n\, \E \left( \int_{[t_1,t_2] \times \T \times \R^2} a_n(t,x)\, f(\xi,\eta)\, \ud \mu_n(t,x,\xi,\eta)\right)^{\gamma}\,  <\ \infty,
\end{equation}
where $f \in C (\R^2)$ is a non-negative function and $\gamma > 0$.
Assume that almost-surely $\mu_n \to \mu$ in $\mathscr{Y}$, and $a_n \to a$ in $C([t_1,t_2] \times \T)$, then 
\begin{equation}\label{afmu_estimate}
\E \left( \int_{[t_1,t_2] \times \T \times \R^2} a(t,x)\, f(\xi,\eta)\, \ud \mu(t,x,\xi,\eta)\right)^{\gamma}\, \leqslant\ C(f).
\end{equation}
Moreover, for all $g \in C (\R^2)$ satisfying $|g(\xi,\eta)| \leqslant f(\xi,\eta)^\delta$ with $\delta \in (0,1)$ and all  $\varphi \in L^{\frac{1}{1-\delta}}([t_1,t_2] \times \T)$ we have 
\begin{multline}\label{Youngconvergence}
\lim_n \E \bigg| \int_{[t_1,t_2] \times \T \times \R^2} a_n(t,x)\, \varphi(t,x)\,  g(\xi,\eta)\, \ud \mu_n(t,x,\xi,\eta)\\ 
-\ \int _{[t_1,t_2] \times \T \times \R^2} a(t,x)\, \varphi(t,x)\, g(\xi,\eta)\, \ud \mu(t,x,\xi,\eta)\bigg|^\gamma\, =\ 0.
\end{multline}
\end{pro}

\begin{proof}[Proof of Proposition~\ref{pro:young}] The proof follows from \cite[Proposition~4.4]{BerthelinVovelle19} applied to the measure $\mu_n':=a_n\mu_n$. Although $\mu_n'$ may not have the required total mass \eqref{TotalMassYoungMeasure}, the arguments hold true.  Note also that \cite[Proposition~4.4]{BerthelinVovelle19} addresses the case $\gamma=1$ only, but the arguments can be adapted easily if one uses Egorov's theorem directly, instead of Vitali's theorem.
\end{proof}

\subsection{Quasi-Polish spaces and the Skorokhod--Jakubowski theorem}\label{subsec:QPolish}

%

In order to study the limit $\varepsilon \to 0$, we need to prove the compactness of $(R^\varepsilon, S^\varepsilon)$ in the non-metrizable space $C([0,T],L^2_w(\T))$, where $L^2_w$ is the $L^2$ space equiped with its weak topology. Therefore, the classical Skorokhod theorem cannot be used. We use instead the generalized Skorokhod--Jakubowski theorem \cite{J97}.

\begin{definition}
A topological space $(\mathcal{X}, \tau)$ is said to be quasi-Polish if there is a sequence $(f_n)_n$ of continuous functions $f_n : \mathcal{X} \to [-1,1]$  separating points in $\mathcal{X}$.
\end{definition}
We refer to \cite[Section 3]{BOS16} (where the terminology ``quasi-Polish'' was introduced by the way), \cite[Appendix B]{GHKP22}, and the references therein for more details on quasi-Polish spaces.

\begin{remark}
($i$) Any Polish space is a quasi-Polish space; ($ii$) A countable product of quasi-Polish spaces is a quasi-Polish space (if, for each $k$, $(f_{k,n})$ is a separating sequence of continuous functions on $\mathcal{X}_k$, and if $\pi^k\colon\mathcal{X}\to\mathcal{X}_k$ is the projection from the product space $\mathcal{X}$ onto the factor $\mathcal{X}_k$, then $\{f_{k,n}\circ\pi^k; k,n\in\N\}$ is a separating sequence of continuous functions on $\mathcal{X}$); ($iii$) The spaces $C([0,T],H^1_w(\T))$, $C([0,T],L^2_w(\T))$ are quasi-Polish spaces.
\end{remark}

We recall now the Skorokhod--Jakubowski theorem \cite[Theorem 2]{J97} 

\begin{thm}\label{thm:Jakubowski}
Let $(\mathcal{X}, \tau)$ be a quasi-Polish space and let $(X_n)_n$ be a sequence of $\mathcal{X}$-valued random variable. Suppose that for all $\delta >0$ there exists a compact subset $K_\delta \subset \mathcal{X}$ such that 
\begin{equation*}
\inf_n \Pro (X_n \in K_\delta)\ >\ 1- \delta.
\end{equation*} 
Then, up to a subsequence of $(X_n)_n$ (noted also $(X_n)_n$), there exists a probability space $(\tilde{\Omega},\tilde{\mathcal{F}},\tilde{\Pro})$, an $\mathcal{X}$-valued random variable $\tilde{X}$ and an $\mathcal{X}$-valued sequence of random variables $(\tilde{X}_n)_n$ defined on $(\tilde{\Omega},\tilde{\mathcal{F}},\tilde{\Pro})$ such that
\begin{gather}
X_n\ \sim\ \tilde{X}_n, \quad n\ =\ 1,2, \cdots \\
\tilde{X}_n\ \to_\tau\ \tilde{X}\ \mathrm{for\ almost\ all\ } \omega \in \tilde{\Omega}.
\end{gather}
\end{thm}

\begin{remark}
The probability space $(\tilde{\Omega},\tilde{\mathcal{F}},\tilde{\Pro})$ can be chosen to be $([0,1],\mathcal{B}_{[0,1]},\mathcal{L})$.
\end{remark}

\subsection{Compactness}\label{subsec:tight}

Let $R^\varepsilon, S^\varepsilon$ and $u^\varepsilon$ be the global solutions of \eqref{SVWEep}-\eqref{udefep} given by Theorems \ref{thm:glob-exist-ep} and \ref{thm:globalsolep}. Let $\mathcal{L}^2$ denote the Lebesgue measure on $[0,T] \times \T$. We define for any $(t,x) \in [0,T] \times \T$ and any $\varepsilon> 0$ the measures 
\begin{gather}
\mu^{\varepsilon}_{t,x}\ \eqdef\ \delta_{R^\varepsilon(t,x)} \otimes \delta_{s^\varepsilon(t,x)}, \qquad  \mu^{\varepsilon}\ \eqdef\  \mu^{\varepsilon}_{t,x}  \rtimes	 \mathcal{L}^2.
\end{gather}
That is to say 
\begin{align*}
\dual{\mu^\varepsilon}{\psi}\ &=\ \int_{[0,T] \times \T} \int_{ \R^2} \psi(t,x,\xi,\eta)\, \ud \mu^{\varepsilon}_{t,x}(\xi,\eta)\, \ud x\, \ud t\ =\ \int_{[0,T] \times \T} \psi(t,x,R^\varepsilon(t,x),S^\varepsilon(t,x))\, \ud x\, \ud t,
\end{align*}
for all  $\psi \in C_b([0,T] \times \T \times \R^2)$.

Let $\mathcal{S}$ be a countable set of functions in $C^\infty_c(\R)$ which is dense in $C_0(\R)$. Set 
\begin{equation}
X_R^\varepsilon\ \eqdef\ \left( f(R^\varepsilon)\right)_{f\in\mathcal{S}},\qquad 
X_S^\varepsilon\ \eqdef\ \left( f(S^\varepsilon)\right)_{f\in\mathcal{S}}.
\end{equation}

Let also $L^2_w$ denote the $L^2$ space equipped with its weak topology and define $C([0,T],L^2_w(\T))^{\aleph_0}$ as the countable product
\begin{equation}
C([0,T],L^2_w(\T))^{\aleph_0}\ \eqdef\ \prod_{\kappa \in \N} C([0,T],L^2_w(\T)),
\end{equation}
with the product topology. Let $\mathcal{X}$ denote the product (which is a quasi-Polish space) of the following spaces 
\begin{equation}
\mathscr{Y},\, \left( C([0,T],L^2_w(\T))^{\aleph_0} \right)^2,\,  C([0,T] \times \T),\,  C([0,T], \mathfrak{U}_{-1}(\T)),\,  C(0,T),\,  L^1((0,T);C(\T)).
\end{equation} 
Our aim is to show that the sequence (of laws) of the $\mathcal{X}$-valued Random variables $(X^\varepsilon)_\varepsilon$ defined by 
\begin{equation}\label{Xepdef}
X^\varepsilon\ \eqdef\ \left( \mu^{ \varepsilon}, X_R^\varepsilon, X_S^\varepsilon, u^\varepsilon, W, \Theta^\varepsilon,\Xi^\varepsilon\right) 
\end{equation}
is tight in $\mathcal{X}$.


\begin{pro}\label{pro:tightness}
For any $\delta>0$, there exists a compact $K_\delta \in \mathcal{X}$ such that 
\begin{gather*}
\sup_\varepsilon \Pro \left( X^\varepsilon \notin K_\delta \right)\ \leqslant\ \delta.
\end{gather*}
\end{pro}

\begin{proof}[Proof of Proposition~\ref{pro:tightness}] The proof will be done in several steps. We show that every element in $X^\varepsilon$ is tight in its proper space. We use the same convention as in \eqref{constantC} for the use of the constant $\mathtt{C}$.\medskip
 
\textbf{Step 1.} (Young measure).
By the energy estimate \eqref{LPeneE} (with $p=1$), we obtain the bound
\begin{gather}\label{Moment2mueps}
\E \int_{[0,T] \times \T \times \R^2} \left[ \xi^2\, +\, \eta^2 \right] \ud \mu^{\varepsilon}(t,x,\xi,\eta)\ =\
 \E \int_{[0,T] \times \T} \int_{\R^2} \left[ \xi^2\, +\, \eta^2 \right] \ud \mu^{\varepsilon}_{t,x}(\xi,\eta)\, \ud x\, \ud t\ \leqslant\ \mathtt{C}(T),
\end{gather}
and then, by the Markov inequality,
\begin{gather*}
\Pro \left( \int_{[0,T] \times \T \times \R^2} \left[ \xi^2\, +\, \eta^2 \right] \ud \mu^{\varepsilon}(t,x,\xi,\eta) \geqslant \mathtt{C}(T)/\delta \right)\ \leqslant\ \delta,
\end{gather*}
for all $\delta >0$: Proposition~\ref{CompactYoung} shows that $(\mu^\varepsilon)_\varepsilon$ is tight.\medskip

\textbf{Step 2.}  (Factor $(X_R^\varepsilon, X_S^\varepsilon)_\varepsilon$). 
It follows from the energy estimate \eqref{LPeneE} (still with $p=1$) that, for $f\in\mathcal{S}$,
\begin{equation}\label{weakcontfRpunct}
\E\left[\|f(R^\varepsilon)\|_{C([0,T];L^2(\T)}^2\right]\ \leqslant\ \mathtt{C}(T).
\end{equation}
Let now $\varphi \in H^1(\T)$. By \eqref{ReqepthetaItoConservativeEPS} we have 
\begin{gather*}
\int_\T \left[f(R^\varepsilon)(t) - f(R^\varepsilon)(s) \right] \varphi\, \ud x\ =\  \int_s^t \int_\T c(u^\varepsilon) f(R^\varepsilon)\, \varphi_x\, \ud x\, \ud \sigma\ +\ \int_\T \left[ \int_s^t f'(R^\varepsilon)\, \Phi^\varepsilon\, \ud W(\sigma) \right] \varphi\, \ud x
\\
+\ \int_s^t \int_\T \left[\half\, q^\varepsilon\, f''(R^\varepsilon)\, +\, \tilde{c}'(u^\varepsilon) \left[ \left( (R^\varepsilon)^2\, -\, (S^\varepsilon)^2 \right) f'(R^\varepsilon)\, +\, 2\, (S^\varepsilon\, -\, R^\varepsilon)\, f(R^\varepsilon) \right] \right]  \varphi\, \ud x\, \ud \sigma\\
+\ \int_s^t \int_\T  \tilde{c}'(u^\varepsilon) \left[ \left( 2\, \Theta^\varepsilon\, R^\varepsilon\, -\, \chi_\varepsilon(R^\varepsilon)\right) f'(R^\varepsilon)\, -\, 4\, \Theta^\varepsilon\, f(R^\varepsilon) \right] \varphi\, \ud x\, \ud \sigma.
\end{gather*} 
Using the energy estimate \eqref{LPeneE} and straightforward computations gives us
\begin{equation}\label{weakcontfR}
\E \left| \int_\T \left[f(R^\varepsilon)(t) - f(R^\varepsilon)(s) \right] \varphi\, \ud x \right|^4\\ 
 \leqslant\ \mathtt{C}(T,\|\varphi\|_{H^1(\T)})\, |t-s|^2.
\end{equation}
By \cite[Theorem~8.2]{Bass2011} we deduce from \eqref{weakcontfR} that, for all $R>0$,
\begin{equation}\label{weakcontfRP}
\Pro \left( \sup_{0\leqslant s<t\leqslant T}\frac{\int_\T \left(f(R^\varepsilon)(t) - f(R^\varepsilon)(s)\right) \varphi\, \ud x}{|t-s|^{1/4}} \geqslant R\right) \leqslant \mathtt{C}(T,\|\varphi\|_{H^1(\T)}) R^{-4}.	
\end{equation}
Combining \eqref{weakcontfRpunct} with \eqref{weakcontfRP} (and the Markov inequality) gives us
\begin{equation}\label{weakcontfRtot}
\Pro \left( \|f(R^\varepsilon)\|_{C([0,T];L^2(\T))}^2\ +\sup_{0\leqslant s<t\leqslant T}\frac{\int_\T \left(f(R^\varepsilon)(t) - f(R^\varepsilon)(s)\right) \varphi\, \ud x}{|t-s|^{1/4}} \geqslant R\right) \leqslant \mathtt{C}(T,\|\varphi\|_{H^1(\T)}) R^{-1}.	
\end{equation}
Then, using \cite[Lemma C.1]{Lions} we obtain that $(f(R^\varepsilon))_\varepsilon$ is tight in $C([0,T],L^2_w(\T))$. Now if $\delta>0$ is fixed and if $\left\{f_1,f_2,\dotsc\right\}$ is an enumeration of $\mathcal{S}$, then, for any $n\geqslant 1$, there exists a compact $K_n$ such that $\Pro (f_n(R^\varepsilon) \in K_n^c) \leqslant \delta /2^n$. From the union bound  
\[
\Pro(\exists n\geqslant 1, f_n(R^\varepsilon)\in K_n^c)\ \leqslant\ \sum_{n\geqslant 1} \Pro(f_n(R^\varepsilon)\in K_n^c)\ \leqslant\ \delta,
\]
and Tychonoff's theorem, we can then deduce that $(X_R^\varepsilon)$ is tight. The proof for $(X_S^\varepsilon)$ is similar.
\medskip

\textbf{Step 3.} (Factor $(u^\varepsilon)_\varepsilon$).
Using the identity $u_t^\varepsilon = (R^\varepsilon+S^\varepsilon)/2 + \Xi^\varepsilon$ we obtain  
\begin{equation*}
\left\|u^\varepsilon(t)\, -\,u^\varepsilon(s) \right\|_{L^2(\T)}\, 
\leqslant\ \mathtt{C}(T)\, |t\, -\, s|\, \sup_{t \in [0,T]} \|(R^\varepsilon, S^\varepsilon, \Xi^\varepsilon)\|_{L^2(\T)}.
\end{equation*}
Then, using the identity $u_x^\varepsilon = (S^\varepsilon-R^\varepsilon)/(2 c(u^\varepsilon))- \Theta  ^\varepsilon/c(u^\varepsilon)$ and the bounds \eqref{LPeneE}, \eqref{eq:boundpsi}, \eqref{uteps01}, we obtain
\begin{equation*}
\E \sup_{t \in [0,T]} \|u^\varepsilon\|_{H^1(\T)}^2\ +\ \E \sup_{\substack{t,s \in [0,T]\\ s \neq t}}\frac{\left\|u^\varepsilon(t)\, -\,u^\varepsilon(s) \right\|_{L^2(\T)}^2}{|t\, -\, s|^2}\ \leqslant\ \mathtt{C}(T).
\end{equation*}
This implies that 
\begin{equation}\label{tightuepsC}
\Pro \left(\sup_{t \in [0,T]} \|u^\varepsilon\|_{H^1(\T)}^2\, +\ \sup_{\substack{t,s \in [0,T]\\ s \neq t}}\frac{\left\|u^\varepsilon(t)\, -\,u^\varepsilon(s) \right\|_{L^2(\T)}^2}{|t\, -\, s|^2}\ \geqslant\ \mathtt{C}(T)/\delta  \right)\, \leqslant\ \delta.
\end{equation}
By \cite[Theorem~5]{Simon87}, \eqref{tightuepsC} implies that $(u^\varepsilon)_\varepsilon$ is tight in the space $C([0,T] \times \T)$.\medskip

\textbf{Step 4.} (Factor $(W)_\varepsilon$).
The process $W$ belongs to the Polish space $C([0,T],\mathfrak{U}_{-1})$ almost surely \cite[Theorem~4.5]{DaPratoZabczyk14}. It is tight by \cite[Theorem~1.3]{BillingsleyBook}, but also for the more elementary reason that $\E\|W\|^2_{C([0,T],\mathfrak{U}_{-1})}<\infty$.\medskip

\textbf{Step 5.} (Factor $(\Theta^\varepsilon,\Xi^\varepsilon)_\varepsilon$). The bounds \eqref{eq:boundpsi} and \eqref{uteps01} show that $(\Theta^\varepsilon,\Xi^\varepsilon)$ is tight in $C(0,T)\times  L^1((0,T);C(\T))$.

\end{proof}

\subsection{Limit equation with defect measure}\label{subsec:limdefectmeas}

\subsubsection{Application of the Skorokhod--Jakubowski theorem}\label{subsubsec:limdefectmeas-appSJ}

From Proposition \ref{pro:tightness} and Theorem \ref{thm:Jakubowski} we deduce the following result.

\begin{thm}[Convergence up to a change of the probability space]\label{th:AppSJ} 
There exists a sequence $(\varepsilon_n)_n$, a probability space $(\tilde{\Omega},\tilde{\mathcal{F}},\tilde{\Pro})$, a $\mathcal{X}$-valued sequence of random variables $(\tilde{X}^{n})_n$ and another $\mathcal{X}$-valued random variable $\tilde{X}$ defined on $(\tilde{\Omega},\tilde{\mathcal{F}},\tilde{\Pro})$ such that for all $n \in \N$ we have
\begin{gather}
X^{\varepsilon_n}\ \sim\ \tilde{X}_n, \quad n\ =\ 1,2, \cdots \\
\tilde{X}_n\ \to\ \tilde{X}\ \mathrm{in}\ \mathcal{X}\  \tilde{\Pro}
\text{-a.s. in}\  \tilde{\Omega}.
\end{gather}
\end{thm}

\begin{remark}
For the sake of simplicity, we omit the subsequence in the remainder of this paper and we write 
\begin{gather} \label{lawsequality}
X^{\varepsilon}\ \sim\ \tilde{X}^\varepsilon, \quad \forall \varepsilon\ \mathrm{(in\ fact\ for\ a\ countable\ values\ of\ \varepsilon)} \\ \label{ConvergenceInX}
\tilde{X}^\varepsilon\ \to\ \tilde{X}\ \mathrm{in}\ \mathcal{X}\  \mathrm{for\ almost\ all\ } \omega \in \tilde{\Omega}.
\end{gather}
\end{remark}

The $\tilde{\Pro}$-a.s. convergence (which also gives convergence in probability) in Theorem~\ref{th:AppSJ} will be exploited to pass to the limit in the equations \eqref{ReqepthetaItoConservativeEPS} and \eqref{SeqepthetaItoConservativeEPS}. We examine first all terms but the stochastic integral in Section~\ref{subsubsec:limdefectmeas-det} below, and then devote Section~\ref{subsubsec:limdefectmeas-sto} to the analysis of the stochastic integral.

\subsubsection{Convergence and identification of some deterministic elements}\label{subsubsec:limdefectmeas-det}

As in \eqref{Xepdef} we have the decomposition 
\begin{gather}\label{tildeXepdef}
\tilde{X}^\varepsilon\ \eqdef\ \left( \tilde{\mu}^{ \varepsilon}, \tilde{X}_R^\varepsilon, \tilde{X}_S^\varepsilon, \tilde{u}^\varepsilon, \tilde{W}^\varepsilon, \tilde{\Theta}^\varepsilon,\tilde{\Xi}^\varepsilon\right), \\  \label{tildeXdef}
\tilde{X} \ \eqdef\ \left( \tilde{\mu}, \tilde{X}_R , \tilde{X}_S , \tilde{u} , \tilde{W} , \tilde{\Theta},\tilde{\Xi} \right).
\end{gather}
If any of the original factors $\mu^\varepsilon$, $R^\varepsilon,\ldots$ satisfy $\Pro$-a.s. a relation which can be characterized in terms of a Borel subset $A$ of $\mathcal{X}$, then this relation is still satisfied $\tilde{\Pro}$-a.s. by the factors with tildas. If moreover $A$ is closed for the topology considered on $\mathcal{X}$, then the relation is also satisfied $\tilde{\Pro}$-a.s. by the limit factors since
\begin{equation}\label{limOK}
\tilde{\Pro}\left(\tilde{X}\in A\right)\ \geqslant\ \limsup_{\varepsilon\to0}\tilde{\Pro}\left(\tilde{X}^\varepsilon\in A\right)=1.
\end{equation}
As a consequence, we have the following identities.

\begin{proposition}\label{prop:structureLimit}
The following identities hold $\tilde{\Pro}$-almost surely:
\begin{equation}\label{NullPerturb}
\tilde{\Theta}\ =\ 0,\quad \tilde{\Xi}\ =\ 0,
\end{equation}
as well as \eqref{udefep} with tildas and 
\begin{gather}\label{tildeRSuep}
\tilde{u}^\eps_x\, =\, \frac{\tilde{S}^\eps\, -\, \tilde{R}^\eps}{2\, c(\tilde{u}^\eps)}\, -\, \frac{\tilde{\Theta}^\eps}{c(\tilde{u}^\eps)}, \quad
\tilde{u}^\varepsilon_t\, =\, \frac{\tilde{S}^\varepsilon\, +\, \tilde{R}^\varepsilon}{2}\, +\, \tilde{\Xi}^\varepsilon, 
\quad \tilde{u}_x\, =\, \frac{\tilde{S}\, -\, \tilde{R}}{2\, c(\tilde{u})}, \quad \tilde{u}_t\ =\ \frac{\tilde{S}\, +\, \tilde{R}}{2}, 
\end{gather}
and
\begin{gather}
\label{tildemuep-structure}
\tilde{\mu}^{\varepsilon}\ =\ \delta_{\tilde{R}^\varepsilon(t,x)} \otimes \delta_{\tilde{S}^\varepsilon(t,x)}    \rtimes	 \mathcal{L}^2, \qquad \tilde{\mu}\ =\ \tilde{\mu}_{t,x} \rtimes \mathcal{L}^2,
\end{gather}
where, for $\mathcal{L}^2$-almost all $(t,x) \in[0,T]\times\T$, $\tilde{\mu}_{t,x}$ is probability measures on $\R^2$. In particular, we have $\tilde{u}\in H^1( (0,T) \times \T)$ with 
\begin{equation}\label{tildeuH1}
\|\tilde{u}\|_{H^1( (0,T) \times \T)}\leqslant \mathtt{C}\left(\|\tilde{S}\|_{L^2( (0,T) \times \T)}+\|\tilde{R}\|_{L^2( (0,T) \times \T)}\right).
\end{equation}
Finally, we have the following convergence results: for all function $g\colon\R\to\R_+$ such that $g(\xi)=o(\xi)$ when $\xi\to\infty$, $\tilde{\Pro}$-a.s.
\begin{equation}\label{cvchiepstilde}
\lim_{\eps\to 0}\int_0^T\int_\T \tilde{c}'(\tilde{u}^\eps)\left[ g(\tilde{R}^\eps)\, \chi_\eps(\tilde{R}^\eps)\,
+\,  g(\tilde{S}^\eps)\, \chi_\eps(\tilde{S}^\eps)\right] \ud x\, \ud t=0.
\end{equation}
\end{proposition}

Due the equality of laws \eqref{lawsequality}, all the estimates obtained in Section \ref{sec:gloabsol}
 are valid with the tildas.

\begin{proof}[Proof of Proposition~\ref{prop:structureLimit}]  The identities in \eqref{NullPerturb} follow from \eqref{eq:boundpsi}-\eqref{uteps01} (apply \eqref{limOK} with $A$ a closed ball around the origin of arbitrary small radius). We also deduce \eqref{cvchiepstilde} from \eqref{eq:TotalEnergyep} and the fact that $\chi_\eps(R)\not=0\Rightarrow \eps R\geqslant 1$.
Let us prove also \eqref{tildemuep-structure} for instance: 
the first identity between the components $\mu^\eps,R^\eps,S^\eps$ of $X^\varepsilon$ can be characterized by the four relations
\begin{multline}\label{tildemuep-structureMV}
\int_{[0,T] \times \T\times\R\times\R}\varphi(t,x)(\alpha_1\xi+\alpha_2\eta+\alpha_3\xi^2+\alpha_4\eta^2)\, \ud \mu(t,x,\xi,\eta)\\
=\int_{[0,T] \times \T}\varphi(t,x) (\alpha_1 R(t,x)+\alpha_2 S(t,x)+\alpha_3 R(t,x)^2+\alpha_4 S(t,x)^2)\, \ud x\, \ud t,
\end{multline}
where $\alpha_i\in\{0,1\}$, for $\varphi$ ranging in a dense countable subset of $C([0,T] \times \T)$, and is therefore a Borel set since each term in \eqref{tildemuep-structureMV} defines a Borel function (in fact continuous or semi-continuous function) of
\begin{equation}
(\mu,R,S)\in \mathscr{Y}\times C([0,T],L^2_w(\T))\times C([0,T],L^2_w(\T)).
\end{equation}
The last equation $\mu\ =\ \mu_{t,x} \rtimes \mathcal{L}^2$ can be characterized by the relations
\begin{equation}\label{useDecYM}
\int_{[0,T] \times \T\times\R\times\R}\varphi(t,x)\, \ud \mu(t,x,\xi,\eta)
=\int_{[0,T] \times \T}\varphi(t,x) \, \ud x\, \ud t,
\end{equation}
for $\varphi$ ranging in a dense countable subset of $C([0,T] \times \T)$, \cite[Section~4.1]{BerthelinVovelle19}, so is a Borel set as well. 
\end{proof}

\begin{proposition}\label{prop:structureLimit2}
After redefinition of $(t,x)\mapsto\tilde{\mu}_{t,x}$ on a $\mathcal{L}^2$-negligible set, we have the following convergence properties: $\tilde{\Pro}$-a.s., for all functions $f_1,f_2\in C(\R)$ satisfying the growth condition
\begin{equation}\label{growthf1f2}
|f(\xi)|\leqslant C_f(1+|\xi|^r),\quad |f'(\xi)|\leqslant C_f(1+|\xi|^{r-1}),\quad |f''(\xi)|\leqslant C_f,\quad r\in [1,2),
\end{equation}
for all $\varphi \in C(\T)$, we have 
\begin{equation}\label{CVut}
\int_\T \left( f_1(\tilde{R}^\varepsilon)+f_2(\tilde{S}^\varepsilon)\right) \varphi\, \ud x\ \to\  \left[t\mapsto\int_{\T\times\R^2}\varphi(x)\left[ f_1(\xi)+f_2(\zeta)\right] \ud \tilde{\mu}_{t,x}(\xi,\zeta)\, \ud x\right]
\end{equation} 
in $C([0,T])$. Moreover, at the initial time $t=0$, we have
\begin{equation}\label{Reduction0}
\tilde{\mu}_{0,x}=\delta_{R_0(x)}\otimes\delta_{S_0(x)}.
\end{equation}
\end{proposition}

\begin{proof}[Proof of Proposition~\ref{prop:structureLimit2}] We first observe that if a function $g\in L^1(0,T)$ coincide a.e. with a function $h\in C([0,T])$, then the set of right-Lebesgue points of $g$ is the full set $[0,T)$ and 
\begin{equation}
h(t)\ =\ \lim_{\delta\to 0} \frac{1}{\delta}\int_t^{t+\delta} g(s)\, \ud s
\end{equation}
for all $t\in[0,T)$. We consider therefore the representative $\tilde{\mu}^*_{t,x}$ of $\tilde{\mu}_{t,x}$ given by 
\begin{equation}\label{Lebmu}
\int_{\T\times\R^2}\varphi(x) F(\xi,\zeta)\, \ud \tilde{\mu}^*_{t,x}(\xi,\zeta)\, \ud x\ \eqdef\ \lim_{\delta\to 0} \frac{1}{\delta}\int_t^{t+\delta}\int_{\T\times\R^2}\varphi(x) F(\xi,\zeta)\, \ud\tilde{\mu}_{s,x}(\xi,\zeta)\, \ud x\, \ud s,
\end{equation}
where $\varphi\in C(\T)$, $F\in C_0(\R^2)$. By differentiation's theory, \eqref{Lebmu} is satisfied for all $t\in I_{\varphi,F}$, where $I_{\varphi,F}$ is the set of right-Lebesgue points of the function 
\begin{equation}
t\mapsto \int_{\T\times\R^2}\varphi(x) F(\xi,\zeta)\, \ud \tilde{\mu}_{s,x}(\xi,\zeta)\, \ud x
\end{equation}
and is of full measure in $[0,T)$. We denote by $I$ the intersection of the sets $I_{\varphi,f}$ over $\varphi$ in a dense countable subset of $C(\T)$ and $F$ in a dense countable subset of $C_0(\R^2)$. Without loss of generality, we can assume that the functions $F$ of the form
\begin{equation}
F_{1,2}(\xi,\zeta)\ =\ f_1(\xi)\, +\, f_2(\zeta),\quad f_1,f_2\in\mathcal{S},
\end{equation}
are here taken into account. If $f_1,f_2\in\mathcal{S}$, then, by \eqref{ConvergenceInX}, the left-hand side of \eqref{CVut} is converging in $C([0,T])$ to a continuous function $\Lambda$. It is also converging in $\mathcal{D}'(0,T)$ to the function
\begin{equation}
\Gamma(t)\ \eqdef\ \int_{\T\times\R^2}\varphi(x) F_{1,2}(\xi,\zeta)\, \ud \tilde{\mu}_{t,x}(\xi,\zeta)\, \ud x,\quad F_{1,2}(\xi,\zeta)\ \eqdef\ f_1(\xi)\, +\, f_2(\zeta),
\end{equation} 
since $\tilde{\mu}^\eps\to\tilde{\mu}$ in $\mathscr{Y}$. Consequently $\Gamma$ and $\Lambda$ coincide a.e. and by our initial observation, \eqref{CVut} is satisfied with the representative $\tilde{\mu}^*_{t,x}$ of $\tilde{\mu}_{t,x}$. A simple limiting argument shows that the result holds true for $f_1$ and $f_2$ satisfying \eqref{growthf1f2}. The last assertion \eqref{Reduction0} is a simple consequence of the fact that $R_0^\eps$ and $S_0^\eps$ are deterministic so $\tilde{\Pro}$-a.s.,
\begin{equation}
\left(\tilde{X}^\eps(0),\tilde{X}^\eps_S(0)\right)\, =\, \left(f(R_0^\eps),f(S_0^\eps)\right)_{f\in\mathcal{S}}.
\end{equation}
This is equivalent to $(\tilde{R}^\eps(0),\tilde{S}^\eps(0))=(R_0^\eps,S_0^\eps)$, $\tilde{\Pro}$-a.s. Since $(R_0^\eps,S_0^\eps)$ converges strongly to $(R_0,S_0)$ in $L^2(\T)$, $\tilde{\mu}_{0,x}$ reduces to the Dirac mass $\delta_{(R_0(x),S_0(x))}=\delta_{R_0(x)}\otimes\delta_{S_0(x)}$.
\end{proof}

Using the estimates \eqref{LPeneE}, \eqref{alpha+2}, \eqref{Oleinik} with Proposition \ref{pro:young} we obtain the following result.

\begin{pro}\label{pro:Youngconv}
 Let $t_0 \in (0,T)$ and $r \in (1,2)$, $p \in (1,\infty)$, $q \in (1,3)$. Then the measure $\tilde{\mu}$ satisfies
\begin{equation}\label{Rbar_L2}
\tilde{\E} \int _{[0,T] \times \T \times \R^2} \left[ |\xi|^2\, +\, |\eta|^2 \right] \ud \tilde{\mu}(t,x,\xi,\eta)\ <\ \infty,
\end{equation}
and
\begin{gather}\label{Rbar_L2plus}
\tilde{\E} \left( \int _{[t_0,T] \times \T \times \R^2} \left[ |\xi^-|^p\, +\, |\eta^-|^p \right] \ud \tilde{\mu}(t,x,\xi,\eta) \right)^{2/p}\, <\ \infty,  \\ \label{L3-}
\tilde{\E}  \int _{[0,T] \times \T \times \R^2} c'(\tilde{u}(t,x)) \left[ |\xi|^q\, +\, |\eta|^q \right] \ud \tilde{\mu}(t,x,\xi,\eta)\ <\ \infty .
\end{gather}
Moreover, for any functions $f,g \in C(\R^2)$ satisfying 
\begin{equation}\label{fgmu}
|f(\xi,\eta)|\, \leqslant\, C \left[ 1\, +\, |\xi|^r\, +\, |\eta|^r \right],   \quad
|g(\xi,\eta)|\, \leqslant\, C \left[1\, +\, |\xi^-|^p\, +\, |\eta^-|^p\, +\, |\xi^+|^q\, +\, |\eta^+|^q\right],
\end{equation}
we have
\begin{multline}
\lim_{\varepsilon \to 0} \tilde{\E} \bigg| \int _{[0,T] \times \T \times \R^2} \phi(t,x)\,  f(\xi,\eta)\, \ud \tilde{\mu}^\varepsilon(t,x,\xi,\eta)\\
 -\ \int _{[0,T] \times \T \times \R^2} \phi(t,x)\, f(\xi,\eta)\, \ud \tilde{\mu}(t,x,\xi,\eta)\bigg|\, =\ 0,
\end{multline}
and 
\begin{multline} 
\lim_{\varepsilon \to 0} \tilde{\E} \bigg| \int _{[t_0,T] \times \T \times \R^2} c'(\tilde{u}^\varepsilon(t,x))\, \varphi(t,x)\,  g(\xi,\eta)\, \ud \tilde{\mu}^\varepsilon(t,x,\xi,\eta)\\
 -\ \int _{[t_0,T] \times \T \times \R^2} c'(\tilde{u}(t,x))\,  \varphi(t,x)\, g(\xi,\eta)\, \ud \tilde{\mu}(t,x,\xi,\eta)\bigg|^{2/p}\, =\ 0, 
\end{multline}
for all $\phi \in L^{\frac{2}{2-r}}([0,T] \times \T)$, $\varphi \in \cup_{q' \in (q,3)} L^{\frac{q'}{q'-q}}([t_0,T] \times \T)$. Finally, the convergence \eqref{CVut} holds true when $f_1, f_2 \in C(\R)$ satisfy the growth condition $|f_1(\xi)| + |f_2(\xi)| \leqslant C(1+|\xi|^r)$ and $\varphi \in L^{\frac{2}{2-r}}(\T)$.
\end{pro}

\subsubsection{The stochastic integral}\label{subsubsec:limdefectmeas-sto}

We would like to show that the equations \eqref{ReqepthetaItoConservativeEPS}-\eqref{SeqepthetaItoConservativeEPS} are satisfied $\tilde{\Pro}$-a.s. by the quantities with tildas and then pass to the limit as $\eps\to 0$. So we would like to establish, for all $\varphi\in C^1(\T)$ and $h\in\mathcal{S}$, the identity
\begin{equation}\label{ReqepthetaItoConservativeEPStilde-weak}
\tilde{M}_{\varphi,h}(t)\ =\ \int_\T \int_{t_0}^t\ h'(\tilde{R}^\varepsilon)\, \varphi\, \Phi^\eps\, \ud \tilde{W}^\eps(s)\, \ud x,
\end{equation}
where
\begin{align} \nonumber
\tilde{M}_{\varphi,h}(t)\ &\eqdef\ \int_\T h(\tilde{R}^\varepsilon)(t,x)\, \varphi(x)\, \ud x \ -\ \int_\T h(\tilde{R}^\varepsilon)(t_0,x)\, \varphi(x)\, \ud x
\ -\ \int_{t_0}^t\int_\T (c(\tilde{u}^\varepsilon)\, h(\tilde{R}^\varepsilon))\, \varphi_x\, \ud x\, \ud s \\ \nonumber
&\quad -\ \int_{t_0}^t\int_\T \left[\tilde{c}'(\tilde{u}^\varepsilon) \left[(\tilde{S}^\varepsilon\, -\, \tilde{R}^\varepsilon)\, B_h(\tilde{R}^\varepsilon,\tilde{S}^\varepsilon)\, -\, h'(\tilde{R}^\varepsilon)\, \chi_\varepsilon(\tilde{R}^\varepsilon)\right] -\, \half\, q^\varepsilon\, h''(\tilde{R}^\varepsilon)\right] \varphi\, \ud x\, \ud s \\ \label{ReqepthetaItoConservativeEPStilde-weak-M}
&\quad -\, 2\, \int_{t_0}^t\int_\T\tilde{c}'(\tilde{u}^\varepsilon)\, \tilde{\Theta}^\varepsilon\, (\tilde{R}^\varepsilon\, h'(\tilde{R}^\varepsilon)
\, -\, 2\, h(\tilde{R}^\varepsilon))\, \ud x\, \ud s,
\end{align}
(the term $B_h$ being defined by \eqref{defBh}). The identification \eqref{ReqepthetaItoConservativeEPStilde-weak} is not trivial. We can either use a characterization in law of the stochastic integral in terms of martingales as in \cite{Ondrejat10,BrzezniakOndrejat11,HofmanovaSeidler12,Hofmanova13b,DebusscheHofmanovaVovelle16} for instance, or resort to the trick of Bensoussan, \cite[p.~282]{Bensoussan1995}, which uses a regularization by convolution of the integrand of the stochastic integral and the stochastic Fubini theorem to ``fully integrate'' the $\ud W(s)$. This last approach will indeed give 
\eqref{ReqepthetaItoConservativeEPStilde-weak}-\eqref{ReqepthetaItoConservativeEPStilde-weak-M}, the details are left to the reader. Let us however specify what is the filtration involved here and what are the properties of the Wiener process $\tilde{W}^\eps$. Let $\left(\tilde{\mathcal{F}}^{0,\varepsilon}_t\right)$ denote the filtration generated by the process $\tilde{X}^\varepsilon$ and let $\left(\tilde{\mathcal{F}}^\varepsilon_t\right)$ be the augmented filtration, obtained by the completion of the right-continuous filtration $\left(\tilde{\mathcal{F}}^{0,\varepsilon}_{t+}\right)$.

\begin{lem}[Equivalent cylindrical Wiener process]\label{lem:EqWiener} The process $\tilde{W}^\varepsilon$ is a $(\tilde{\mathcal{F}}^\varepsilon_t)$-adapted cylindrical Wiener process, the increment $\tilde{W}^\varepsilon(t)-\tilde{W}^\varepsilon(s)$ is independent on $\tilde{\mathcal{F}}^\varepsilon_s$ for all $t\geqslant s\geqslant 0$ and $\tilde{\Pro}$-a.s 
\begin{equation}\label{tildeWep}
\tilde{W}^\varepsilon\ =\ \sum_k \tilde{\beta}^\varepsilon_k\, g_k
\end{equation}
in $C([0,T], \mathfrak{U}_{-1})$, where $(\tilde{\beta}^\varepsilon_1(t),\tilde{\beta}^\varepsilon_2(t),\dotsc)$ are independent one-dimensional Wiener processes.
\end{lem}

\begin{proof}
We obtain directly from the definition of the filtration $(\tilde{\mathcal{F}}_t^\varepsilon)$ that $\tilde{W}^\varepsilon$ is a $(\tilde{\mathcal{F}}^\varepsilon_t)$-adapted Wiener process. If $t\geqslant s\geqslant 0$, that $\tilde{W}^\varepsilon(t)-\tilde{W}^\varepsilon(s)$ is independent on $\tilde{\mathcal{F}}^{0,\varepsilon}_s$ means 
\begin{equation}\label{IncrementTilde}
\tilde{\E}\left[G\left(\tilde{W}^\varepsilon(t)-\tilde{W}^\varepsilon(s)\right)H(\tilde{X}^\varepsilon(s_1),\dotsc,\tilde{X}^\varepsilon(s_m))\right]\, =\ 0,
\end{equation}
for all $m\geqslant 1$, all times $0\leqslant s_1\leqslant\dotsb\leqslant s_m\leqslant s$, and all continuous functions $G\colon\mathfrak{U}_{-1}\to\R$, $H\colon \mathcal{X}^m\to\R$. The condition \eqref{IncrementTilde} is clearly invariant by change of the probability space and so true, since satisfied by the original variables when the tildas are removed. We have then 
\begin{equation}\label{IncrementTilde+}
\tilde{\E}\left[G\left(\tilde{W}^\varepsilon(t)-\tilde{W}^\varepsilon(s')\right)\tilde{Y}_s\right]\, =\ 0,
\end{equation}
for all $s<s'\leqslant t$ and all bounded $\tilde{\mathcal{F}}^{0,\varepsilon}_{s+}$-measurable random variable $\tilde{Y}_s$. By continuity of $r\mapsto \tilde{W}^\varepsilon(r)$, we can let $s'\downarrow s$ in \eqref{IncrementTilde+} and obtain the independence of $\tilde{W}^\varepsilon(t)-\tilde{W}^\varepsilon(s)$ and $\tilde{\mathcal{F}}^\varepsilon_{s+}$, and thus also independence with the completed $\sigma$-algebra. Finally, we have the expansion \eqref{tildeWep} with 
\begin{equation}
\tilde{\beta}^\varepsilon_k(t)\ \eqdef\ \dual{\tilde{W}^\varepsilon(t)}{g_k}_{\mathfrak{U}_{-1}}.
\end{equation}
The map $w\mapsto \dual{w}{g_k}_{\mathfrak{U}_{-1}}$ is continuous $C([0,T];\mathfrak{U}_{-1})\to C([0,T])$, so, for each $J$ finite, $(\tilde{\beta}^\varepsilon_k(t))_{k\in J}$ and $(\beta_k(t))_{k\in J}$ have the same laws in $C([0,T])^J$. It follows that $(\tilde{\beta}^\varepsilon_1(t),\tilde{\beta}^\varepsilon_2(t),\dotsc)$ are independent one-dimensional Wiener processes.
\end{proof}

Once \eqref{ReqepthetaItoConservativeEPStilde-weak}-\eqref{ReqepthetaItoConservativeEPStilde-weak-M} is established we can use \cite[Lemma~2.1]{DebusscheGlattHoltzTemam2011} to pass to the limit in the stochastic integral. This is explained in the next section~\ref{subsubsec:limdefectmeas-lim}.

\subsubsection{Limit equation with defect measure}\label{subsubsec:limdefectmeas-lim}

We introduce the following notations 
\begin{equation}\label{fbardef}
\langle f(R,S)\rangle(t,x)\ \eqdef\ \int_{\R^2} f(\xi,\eta)\, \ud \tilde{\mu}_{t,x}(\xi,\eta),
\end{equation}
for any $f$ such that \eqref{fbardef} is well defined.
As a consequence of Proposition~\ref{prop:structureLimit2} and  Proposition~\ref{pro:Youngconv}, if $t_0\in(0,T)$ and if $f$ and $g$ satisfy \eqref{fgmu} while $f_1$, $f_2$ satisfy \eqref{growthf1f2},  then, up to a subsequence, we have, $\tilde{\Pro}$-a.s., 
\begin{align*}
f(\tilde{R}^\varepsilon,\tilde{S}^\varepsilon)\  &\rightharpoonup\ \langle f(R,S)\rangle  &\text{in } L^{2/r}([0,T] \times \T), \\
c'(\tilde{u}^\varepsilon)\, g(\tilde{R}^\varepsilon,\tilde{S}^\varepsilon)\     &\rightharpoonup\ c'(\tilde{u})\, \langle g(R,S)\rangle  &\text{in } L^{q'/q}([t_0,T] \times \T) \quad \forall q' \in (q,3),\\
\int_\T f_1(\tilde{R}^\varepsilon)+f_2(\tilde{S}^\varepsilon)\, \varphi\, \ud x\ &\to\ \int_\T \langle f_1(R)+f_2(S)\rangle\, \varphi\, \ud x
&\text{in }C([0,T]).
\end{align*}
Using the dominated convergence theorem, we also have 
\begin{equation}\label{Phiconvergence}
\lim_\varepsilon \|\Phi - \Phi^\varepsilon\|_{L_2(\mathfrak{U},L^2(\T))}^2\ =\ \sum_{k \geqslant 1} \lim_\varepsilon \|\sigma_k - \sigma_k^\varepsilon\|_{L^2(\T)}^2\ =\ 0.
\end{equation}

Let $(\tilde{\mathcal{F}}_t)$ be defined as the augmented filtration generated by $\tilde{X}$. Lemma~\ref{lem:EqWiener} has a completely equivalent version for $\left((\tilde{\mathcal{F}}_t),(\tilde{W}(t))\right)$ so we can claim that
\begin{equation}\label{tildeStochasticBasis}
\left(\tilde{\Omega},\tilde{\Pro},\tilde{\mathcal{F}},(\tilde{\mathcal{F}}_t),(\tilde{W}(t))\right)\ \mbox{is a stochastic basis.}
\end{equation}
We can now pass to the limit $[\eps\to 0]$ in \eqref{ReqepthetaItoConservativeEPStilde-weak}-\eqref{ReqepthetaItoConservativeEPStilde-weak-M}, using \cite[Lemma~2.1]{DebusscheGlattHoltzTemam2011} and a similar argument as \eqref{Phiconvergence} for the stochastic integral, \eqref{cvchiepstilde}, and Proposition~\ref{pro:Youngconv}, to obtain the limit equation
\begin{gather}\nonumber
\int_\T \langle h(R)\rangle(t,x)\, \varphi(x)\, \ud x \ -\ \int_\T \langle h(R)\rangle(t_0,x)\, \varphi(x)\, \ud x
\ -\ \int_{t_0}^t\int_\T c(\tilde{u})\, \langle h(R)\rangle\, \varphi_x\, \ud x\, \ud s\\ \nonumber
=\ \int_{t_0}^t\int_\T \left[ \tilde{c}'(\tilde{u}) \left[\left\langle 2\, (S-R)\, h(R)\, +\, h'(R)\, (R^2-S^2)\right\rangle\right] +\, \half\, q\, \langle h''(R)\rangle\right] \varphi\, \ud x\, \ud s \\ \label{RtildeLIM}
+\, \int_\T \int_{t_0}^t \langle h'(R)\rangle\, \varphi\, \Phi\, \ud \tilde{W}(s)\, \ud x,
\end{gather}
and, similarly,
\begin{gather}\nonumber
\int_\T \langle h(S)\rangle(t,x)\, \varphi(x)\, \ud x \ -\ \int_\T \langle h(S)\rangle(t_0,x)\, \varphi(x)\, \ud x
\ +\ \int_{t_0}^t\int_\T c(\tilde{u})\, \langle h(S)\rangle\, \varphi_x\, \ud x\, \ud s\\ \nonumber
=\ \int_{t_0}^t\int_\T\left[\tilde{c}'(\tilde{u}) \left[\left\langle 2\, (R-S)\, h(S)\, +\, h'(S)\, (S^2-R^2)\right\rangle\right] +\, \half\, q\, \langle h''(S)\rangle\right] \varphi\,  \ud x\, \ud s\ \\ \label{StildeLIM}
+\, \int_\T  \int_{t_0}^t \langle h'(S)\rangle\, \varphi\, \Phi\, \ud \tilde{W}(s)\, \ud x,
\end{gather}
those equations \eqref{RtildeLIM} and \eqref{StildeLIM} being satisfied when $h\in W^{2,\infty}_\mathrm{loc}(\R)$ is such that 
\begin{equation}\label{growth-h0}
|h(\xi)|\leqslant C(1+|\xi|^r),\qquad |h'(\xi)|\leqslant C(1+|\xi|^{r-1}), \qquad |h''(\xi)|\leqslant C,
\end{equation}
for some given constant $C\geqslant 0$ and some exponent $r\in[1,2)$. In particular, for $h(\xi)=\xi$, we obtain
\begin{multline}\label{RtildeLIMID} 
\int_\T \langle R\rangle(t,x)\, \varphi(x)\, \ud x \ -\ \int_\T \langle R\rangle(t_0,x)\, \varphi(x)\, \ud x
\ -\ \int_{t_0}^t\int_\T c(\tilde{u})\, \langle R\rangle\, \varphi_x\, \ud x\, \ud s\\
=\ -\int_{t_0}^t\int_\T\tilde{c}'(\tilde{u}) \left[\langle R^2\rangle\, -\, 2\, \langle RS\rangle\, +\, \langle S^2\rangle\right] \varphi\,  \ud x\, \ud s\, 
+\, \int_\T \int_{t_0}^t \varphi\, \Phi\, \ud \tilde{W}(s)\, \ud x,
\end{multline}
and, similarly,
\begin{multline}\label{StildeLIMID}
\int_\T \langle S\rangle(t,x)\, \varphi(x)\, \ud x \ -\ \int_\T \langle S\rangle(t_0,x)\, \varphi(x)\, \ud x
\ +\ \int_{t_0}^t\int_\T c(\tilde{u})\, \langle S\rangle\, \varphi_x\, \ud x\, \ud s\\
=\ -\int_{t_0}^t\int_\T\tilde{c}'(\tilde{u}) \left[\langle R^2\rangle\, -\, 2\, \langle RS\rangle\, +\, \langle S^2\rangle\right] \varphi\, \ud x\, \ud s\, 
+\, \int_\T \int_{t_0}^t  \varphi\, \Phi\, \ud \tilde{W}(s)\, \ud x,
\end{multline}
Assume $\langle RS\rangle=\langle R\rangle\langle S\rangle$ (actually, this is established as a first step in the next section~\ref{sec:Young}). Then \eqref{RtildeLIMID}-\eqref{StildeLIMID} is the expected limit equation \eqref{SVWE2} in conservative form, up to the defect (non-negative) terms $\langle R^2\rangle-\langle R\rangle^2$ and $\langle S^2\rangle-\langle S\rangle^2$. We will show in Section~\ref{sec:Young} that these defect measures are trivial. To that effect, we will need an evolution equation for 
some non-linear functions of $\langle R\rangle$, as stated in the following proposition.

\begin{proposition}[Renormalization in \eqref{RtildeLIMID}]\label{prop:Renormalize} Let $h\in C^\infty(\R)$ satisfy the growth condition \eqref{growth-h0}. We have then 
\begin{gather}\nonumber
\int_\T  h(\langle R \rangle)(t,x)\, \varphi(x)\, \ud x \ -\ \int_\T \langle h(\langle R \rangle)(t_0,x)\, \varphi(x)\, \ud x
\ -\ \int_{t_0}^t\int_\T c(\tilde{u})\, h(\langle R\rangle)\, \varphi_x\, \ud x\, \ud s\\ \nonumber
= \int_{t_0}^t\int_\T\tilde{c}'(\tilde{u}) \left[2\left(\langle R \rangle-\langle S \rangle\right)
\left[\langle R \rangle h'(\langle R \rangle)\, -\, h(\langle R \rangle)\right]
\, -\, h'(\langle R \rangle) \left\langle\left(R-S\right)^2\right\rangle\right] \varphi\, \ud x\, \ud s \\ \label{RtildeLIMRenormalize}
+\, \half \int_{t_0}^t\int_\T  q\, h''(\langle R \rangle)\varphi\, \ud x\, \ud s\,
+\, \int_\T \int_{t_0}^t h'(\langle R \rangle)\, \varphi\, \Phi\, \ud \tilde{W}(s)\, \ud x,
\end{gather}
for all $0\leqslant t_0\leqslant t$.
\end{proposition}

\begin{proof}[Proof of Proposition~\ref{prop:Renormalize}] The proof is the same as in \cite{DiPernaLions89}: by regularization and study of the remainder terms. Since we have to treat additional stochastic terms here, and since renormalization will again be invoked later (in the derivation of \eqref{evolutionHkappa} precisely), we give few details. Let $(\rho_\delta)$ be an approximation of the unit constituted of even functions and let $J_\delta$ be the Friedrichs regularization operator $F\mapsto F\ast \rho_\delta$. We use $\varphi:=J_\delta\varphi$ in \eqref{RtildeLIMID} to obtain the equation
\begin{equation}\label{RegularizedRLIMID}
\ud R^\delta(t,x)\ +\, \left[ c(\tilde{u}(t,x))\,  R^\delta(t,x)\right]_x \ud t\ =\
F^\delta(t,x)\, \ud t\ +\ \gamma^\delta(t,x)\, \ud t\ +\ \Phi^\delta(x)\, \ud \tilde{W}(t),
\end{equation}
where 
$R^\delta=J_\delta\langle R \rangle$, $\Phi^\delta=J_\delta\Phi$ and
\begin{equation}
F^\delta(t)\ =\ J_\delta F(t),\quad F\ \eqdef\ -\tilde{c}'(\tilde{u}) \left[\langle R^2\rangle\, -\, 2\, \langle RS\rangle\, +\, \langle S^2\rangle\right],
\end{equation}
and $\gamma^\delta$ is the commutator
\begin{equation}
\gamma^\delta\ =\ \left[ c(\tilde{u})\,  R^\delta\right]_x -\, J_\delta \left[ c(\tilde{u})\langle R\rangle \right]_x.
\end{equation}
The equation \eqref{RegularizedRLIMID} is an equation satisfied for all $x\in\T$ by the real-valued stochastic process $(R^\delta(t,x))$. Let $h\in C^\infty_c(\R)$. The It\^o formula and \eqref{RegularizedRLIMID} give
\begin{gather}\nonumber
\ud h(R^\delta(t,x))\, +\, h'(R^\delta(t,x)) \left[ c(\tilde{u}(t,x))\, R^\delta(t,x) \right]_x \ud t\ =\
h'(R^\delta(t,x))\, F^\delta(t,x)\, \ud t\\ \label{RegularizedRLIMID-Ito}
 +\ h'(R^\delta(t,x))\, \gamma^\delta(t,x)\, \ud t\ +\ \half\, h''(R^\delta(t,x))\, q^\delta(x)\, \ud t\, 
+\, h'(R^\delta(t,x))\, \Phi^\delta(x)\, \ud \tilde{W}(t),
\end{gather}
where $q^\delta$ is defined in \eqref{qeps-q}.
By the chain-rule for $H^1$-functions and \eqref{tildeRSuep} we have
\begin{equation}\label{chainruleH1Tr}
h'(R^\delta) \left[ c(\tilde{u})\, R^\delta \right]_x =  \left[ c(\tilde{u})\, h(R^\delta) \right]_x +\, 2\, \tilde{c}'(\tilde{u}) \left( \tilde{S} - \tilde{R} \right) \left[h'(R^\delta)\, R^\delta\, -\, h(R^\delta) \right].
\end{equation}
We combine \eqref{chainruleH1Tr} with \eqref{tildeRSuep}, integrate \eqref{RegularizedRLIMID-Ito} against a function $\varphi\in C^1(\T)$ and pass to the limit $\delta\to 0$ (on the formulation integrated in time) to obtain \eqref{RtildeLIMRenormalize}. The term involving the commutator $\gamma^\delta$ is
\begin{equation}
\int_{t_0}^t\int_\T h'(R^\delta)\, \gamma^\delta\, \varphi\, \ud x\, \ud s,
\end{equation}
which converges to $0$ since $h'(R^\delta)$ and $\varphi$ are bounded, while $\gamma^\delta\to 0$ in $L^1((t_0,t)\times\T)$ by \cite[Lemma~II.1.ii)]{DiPernaLions89} applied with $\alpha=2$, $p=2$. The limit of the other terms is obtained by standard considerations.
\end{proof}

\section{Reduction of the Young measures}\label{sec:Young} 

Our aim is to show that $\tilde{\mu}_{t,x} = \delta_{\tilde{R}(t,x)} \otimes \delta_{\tilde{S}(t,x)}$. 
We prove in a first step that $\tilde{\mu}_{t,x} = \tilde{\nu}^1_{t,x}  \otimes \tilde{\nu}^2_{t,x}$ where $\tilde{\nu}^1_{t,x} $ and $\tilde{\nu}^2_{t,x} $ are two probability measures on $\R$ that will be identified later.

\subsection{Compensated compactness}

For almost all $(t,x) \in [0,T] \times \T$ we can define the marginals $\tilde{\nu}^1_{t,x} $ and $\tilde{\nu}^2_{t,x} $ of $\tilde{\mu}_{t,x}$:
\begin{equation*}
\tilde{\nu}^1_{t,x} (A)\ \eqdef\ \tilde{\mu}_{t,x}(A \times \R), \qquad \tilde{\nu}^2_{t,x} (A)\ \eqdef\ \tilde{\mu}_{t,x}(\R \times A),\quad \forall A \in \mathcal{B}(\R).
\end{equation*}
The aim of this section is to prove the following lemma 
\begin{lem}\label{lem:productmu} For almost all $(\omega,t,x) \in \tilde{\Omega} \times [0,T] \times \T$ we have
\begin{equation}\label{Product}
\tilde{\mu}_{t,x}\ =\ \tilde{\nu}^1_{t,x}\,  \otimes\, \tilde{\nu}^2_{t,x}.
\end{equation}
\end{lem}

\begin{proof}[Proof of Lemma~\ref{lem:productmu}] We will use the div-curl lemma. Our aim is to prove that, given 
$f\in C^\infty_c(\R)$, the sequence 
\begin{align}\label{divcurlR}
\left\{ f(\tilde{R}^\varepsilon)_t\ +\, \left[ c(\tilde{u})\, f(\tilde{R}^\varepsilon)\right]_x \right\}_\varepsilon
\end{align}
is tight in  $H^{-1}((0,T) \times \T)$. By \eqref{ReqepthetaItoConservativeEPS}, we can write \eqref{divcurlR} as the sum
\begin{equation}\label{divcurlRdec}
\partial_t M_{\tilde{R}}^\varepsilon\ + \ T^\varepsilon\ +\ \left[ \left(c(\tilde{u}) - c(\tilde{u}^\varepsilon) \right) f(\tilde{R}^\varepsilon)\right]_x,
\end{equation}
where
\begin{align*}
T^\varepsilon\ 
&\eqdef\ \half\, q^\varepsilon\, f''(\tilde{R}^\varepsilon)\ +\ \tilde{c}'(\tilde{u}^\varepsilon) \left[ \left( (\tilde{R}^\varepsilon)^2\, -\, (\tilde{S}^\varepsilon)^2  \right) f'(\tilde{R}^\varepsilon)\, +\, 2\, (\tilde{S}^\varepsilon\, -\, \tilde{R}^\varepsilon)\, f(\tilde{R}^\varepsilon) \right] \\ 
&\quad +\ \tilde{c}'(\tilde{u}^\varepsilon) \left[ \left( 2\, \tilde{\Theta}^\varepsilon\, \tilde{R}^\varepsilon\, -\, \chi_\varepsilon(\tilde{R}^\varepsilon)\right) f'(\tilde{R}^\varepsilon)\, -\, 4\, \tilde{\Theta}^\varepsilon\, f(\tilde{R}^\varepsilon) \right],
\end{align*}
and 
\begin{equation*}
M_{\tilde{R}}^\varepsilon(t,x)\ \eqdef\ \sum_k \int_0^t f'(\tilde{R}^\varepsilon(s,x))\, {\sigma}_k^\varepsilon\, (x)\, \ud \tilde{\beta}_k^\varepsilon(s).
\end{equation*}
Clearly, the quantity 
\begin{equation}
\left[ c(\tilde{u}) - c(\tilde{u}^\varepsilon) \right] f(\tilde{R}^\varepsilon)
\end{equation} 
tends to $0$ $\tilde{\Pro}$-almost surely in $L^2((0,T) \times \T)$ and so the sequence 
\begin{equation}
\left\{ \left[ \left( c(\tilde{u}) - c(\tilde{u}^\varepsilon) \right) f(\tilde{R}^\varepsilon)\right]_x \right\}_\varepsilon
\end{equation} 
is tight in $H^{-1}((0,T) \times \T)$. There remains to examine the terms $M_{\tilde{R}}^\varepsilon$ and $T^\eps$, which is done in the following two steps, before we conclude the argument in Step~3.\medskip
 
\textbf{Step 1.} Using the Itô isometry with the embedding $L^2 \hookrightarrow H^{-1}$ gives us
\begin{align*}
\tilde{\E}\, \|M_{\tilde{R}}^\varepsilon(t+h,\cdot) - M_{\tilde{R}}^\varepsilon (t,\cdot) \|_{H^{-1}(\T)}^2\ 
\leqslant\ \tilde{\E}\, \|M_{\tilde{R}}^\varepsilon(t+h,\cdot) - M_{\tilde{R}}^\varepsilon (t,\cdot) \|_{L^2(\T)}^2\
\leqslant\ C\, T\, |h|.
\end{align*}
Integrating with respect to $t$, we obtain 
\begin{align*}
\tilde{\E}\, \|M_{\tilde{R}}^\varepsilon(\cdot+h,\cdot) - M_{\tilde{R}}^\varepsilon \|_{L^2((0,T-h), H^{-1}(\T))}^2\ 
\leqslant\ C\,  |h|.
\end{align*}
Using the Itô isometry again we also have
\begin{align*}
\tilde{\E}\, \| M_{\tilde{R}}^\varepsilon (t,\cdot) \|_{L^2(\T)}^2\ 
\leqslant\ C\, T\, t,
\end{align*}
and integrating with respect to $t$ yields $\tilde{\E} \| M_{\tilde{R}}^\varepsilon (t,\cdot) \|_{L^2((0,T) \times \T)}^2 \leqslant C$. By \cite[Theorem 3]{Simon87}, the set 
\begin{align*}
\left\{ M\in L^2((0,T)\times\T); \|M\|_{L^2((0,T) \times \T)} + \sup_{h \in [0,T/2]} \left\| \frac{M(\cdot+h)-M}{\sqrt{h}} \right\|_{L^2((0,T-h), H^{-1}(\T))}\leqslant R \right\}
\end{align*}
is compact in $L^2((0,T), H^{-1}(\T))$. It follows then that $(M_{\tilde{R}}^\varepsilon)_\varepsilon$ is tight in $L^2([0,T], H^{-1}(\T))$. Finally, we can define $\partial_t M_{\tilde{R}}^\varepsilon$ as a $H^{-1}((0,T) \times \T)$-valued random variable and the sequence $(\partial_t M_{\tilde{R}}^\varepsilon)_\varepsilon$ is tight in $H^{-1}((0,T) \times \T)$.\medskip

\textbf{Step 2.} Since $f\in C^\infty_c(\R)$ , $\chi_\varepsilon(\xi) \leqslant \xi^2$ and by the inequality \eqref{alpha+2} we obtain that $(T^\varepsilon)_\varepsilon$ is bounded in $L^p(\tilde{\Omega} \times (0,T) \times \T)$ for all $p \in (1,3/2)$.
By the compact embedding $L^p((0,T) \times \T) \Subset H^{-1}([0,T] \times \T)$, it follows that $(T^\varepsilon)_\varepsilon$ is tight in  $H^{-1}((0,T) \times \T)$. We can therefore conclude that the sequence \eqref{divcurlR} is tight in $H^{-1}((0,T) \times \T)$.\medskip


\textbf{Step 3.} The counterpart of the analysis above for the variable $S^\eps$ is that, for all $g\in C^\infty_c(\R)$, 
the sequence 
\begin{align*}
\left\{ g(\tilde{S}^\varepsilon)_t\ -\, \left[ c(\tilde{u})\, g(\tilde{S}^\varepsilon)\right]_x \right\}_\varepsilon
\end{align*}
is tight in $H^{-1}((0,T) \times \T)$. Set
\begin{align*}
\tilde{Y}^\varepsilon\ \eqdef\  \begin{pmatrix}
f(\tilde{R}^\varepsilon) \\
c(\tilde{u})\, f( \tilde{R}^\varepsilon)
\end{pmatrix}, \qquad 
\tilde{Z}^\varepsilon\ \eqdef\  \begin{pmatrix}
c(\tilde{u})\, g(\tilde{S}^\varepsilon) \\
g( \tilde{S}^\varepsilon)
\end{pmatrix}, \qquad
\end{align*}
and
\begin{align*}
\tilde{Y}\ \eqdef\  \begin{pmatrix}
\langle f({R})\rangle\\
c(\tilde{u})\, \langle f( {R})\rangle
\end{pmatrix}, \qquad 
\tilde{Z}\ \eqdef\  \begin{pmatrix}
c(\tilde{u})\, \langle g({S})\rangle \\
 \langle g( {S})\rangle
\end{pmatrix}.
\end{align*}
We know that $(\tilde{Y}^\varepsilon,\tilde{Z}^\varepsilon)$ converges weakly to $(\tilde{Y},\tilde{Z})$ as $\varepsilon \to 0$ and from the previous steps, we have that $\left\{ \mathrm{div}_{t,x} \tilde{Y}^\varepsilon \right\}_\varepsilon $ and $\left\{ \mathrm{curl}_{t,x} \tilde{Z}^\varepsilon \right\}_\varepsilon $ are tight in $H^{-1}((0,T) \times \T)$. Then, for any $\alpha>0$ there exists a compact set $K_\alpha \subset H^{-1}((0,T) \times \T)$ such that 
\begin{equation}\label{divcurlalpha}
 \mathrm{div}_{t,x} \tilde{Y}^\varepsilon \in K_\alpha\ \mathrm{and}\ \mathrm{curl}_{t,x} \tilde{Z}^\varepsilon \in K_\alpha 
\end{equation}
with a probability larger than $1-\alpha$.
If \eqref{divcurlalpha} is realized, the div-curl lemma ensures that $\tilde{Y}^\varepsilon \cdot \tilde{Z}^\varepsilon$ converges weakly to $\tilde{Y} \cdot \tilde{Z}$ as $\varepsilon \to 0$. We have therefore 
\begin{equation}\label{eq:productmu}
\langle f(R) g(S)\rangle\\ = \langle f(R)\rangle\ \langle g(S)\rangle
\end{equation} 
for almost all $(\omega,t,x) \in A_\alpha$ where $\tilde{\Pro} \times \mathcal{L}^2 (A_\alpha) \geqslant (1-\alpha) T$.
We choose a decreasing sequence $(\alpha_n)_n$ that converges to $0$ and assume without loss of generality that $(K_{\alpha_n})_n$ and $(A_{\alpha_n})_n$ are increasing sequences of sets. Then $A:= \cup_n A_{\alpha_n}$ is of full measure, $\tilde{\Pro} \times \mathcal{L}^2 (A)=T$,  and \eqref{eq:productmu} is satisfied on $A$.
%
\end{proof}

\begin{remark}\label{rk:growthh} As a consequence of \eqref{eq:productmu} and \eqref{Rbar_L2plus}, and provided $t_0>0$, we can partially relax the growth condition \eqref{growth-h0} on the function $h$ in \eqref{RtildeLIM}, \eqref{StildeLIM} and \eqref{RtildeLIMRenormalize}, to admit the following growth
\begin{equation}\label{growth-h1}
|h(\xi)|\leqslant C(1+|\xi^-|^{p}+|\xi^+|^r),\quad p\in [1,\infty),\, r\in [1,2).
\end{equation}
\end{remark}

\subsection{An evolution equation for the defect measure}

Our goal here is the show that the measures $\tilde{\nu}^1_{t,x} $ and $\tilde{\nu}^2_{t,x} $ are Dirac measures, which we will characterize by the identities $\langle R^2\rangle=\langle R\rangle^2$ and $\langle S^2\rangle=\langle S\rangle^2$. We introduce the non-negative quantity $\Delta \eqdef \frac12\left(\langle R^2\rangle-\langle R\rangle^2\right)$ and derive an evolution transport equation (inequality, see \eqref{DeltaR}) satisfied by $\Delta$. Since $\Delta(t,\cdot)\to 0$ when $t\downarrow 0$ (see Proposition~\ref{prop:CVDelta0}), this will imply $\Delta(t,\cdot)=0$ at all positive times $t\in [0,T]$.\medskip

Since the map $\xi\mapsto \xi^2$ does not satisfy \eqref{growth-h1}, we have to work with suitable truncate functions. For $\kappa\geqslant 0$, we define the truncate function
\begin{equation}\label{hatQkappadef}
Q_\kappa(\xi)\ \eqdef\ \half\, \xi^2\, -\, \half \left((\xi-\kappa)^+\right)^2.
\end{equation}

\begin{lem}\label{lem:controlTrunc}
For any $T > 0$ we have, $\tilde{\Pro}$-a.s., and in $L^2(\Omega)$,
\begin{gather}\label{L2conv}
\lim_{\kappa \to \infty} \left\| \langle Q_\kappa '(R)\rangle\ -\ Q_\kappa '\left(\langle R\rangle\right) \right\|_{L^2([0,T] \times \T)}\ =\ \lim_{\kappa \to \infty} \left\| \langle Q_\kappa '(R)\rangle\ -\ \langle R\rangle\right\|_{L^2([0,T] \times \T)}\
 =\ 0.
\end{gather}
Also, for all $\kappa \geqslant 0$, we have for almost all $(\omega,t,x) \in \tilde{\Omega} \times [0,T] \times \T$,
\begin{align}\label{TS}
\half \left[ \langle Q_\kappa '(R)\rangle\ -\ Q_\kappa '\left(\langle R\rangle\right) \right]^2\ 
&\leqslant\ \langle Q_\kappa(R)\rangle\ -\ Q_\kappa\left(\langle R\rangle\right),
\end{align}
and
\begin{equation}\label{CompareDelta}
\Delta_\kappa\ \eqdef\  \langle Q_\kappa(R)\rangle - Q_\kappa(\langle R\rangle)\
\leqslant\ \Delta\ \eqdef\ \half \left[\langle R^2\rangle - \langle R\rangle^2\right],
\end{equation}
and, $\tilde{\Pro}$-a.s., 
\begin{equation}\label{CVDelta}
\lim_{\kappa \to \infty}  \left\| \Delta_\kappa-\Delta \right\|_{L^1([0,T] \times \T)}\ =\ 0.
\end{equation}
\end{lem}

\begin{proof}[Proof of Lemma~\ref{lem:controlTrunc}] By differentiation in \eqref{hatQkappadef}, we obtain
\begin{equation}
\xi-Q_\kappa'(\xi)=(\xi-\kappa)^+.
\end{equation}
By the triangular inequality, we have then  
\begin{align*}
\left|  \langle Q_\kappa '(R)\rangle\, -\, Q_\kappa '(\langle R\rangle) \right|\ 
&\leqslant\ \left| \langle Q_\kappa '(R)\rangle\, -\, \langle R\rangle \right|\, +\, \left| Q_\kappa '(\langle R\rangle)\, -\, \langle R\rangle\right|\\
&=\ \langle R-Q_\kappa '(R)\rangle\, +\, \langle R\rangle-Q_\kappa '(\langle R\rangle).
\end{align*}
Since $\xi\mapsto (\xi-\kappa)^+$ is convex, we obtain
\begin{equation*}
\left|  \langle Q_\kappa '(R)\rangle\, -\, Q_\kappa '(\langle R\rangle) \right|\ 
\leqslant 2 \langle (R-\kappa)^+\rangle,
\end{equation*}
and by Jensen's inequality again,
\begin{equation*}
\left|  \langle Q_\kappa '(R)\rangle\, -\, Q_\kappa '(\langle R\rangle) \right|^2\ 
\leqslant 4 \langle \left[(R-\kappa)^+\right]^2\rangle.
\end{equation*}
Since $\langle R^2\rangle\in L^1(\Omega\times (0,T)\times\T)$, \eqref{L2conv} follows. Let us now establish \eqref{TS}.
We want to prove that the function
\begin{equation*}
\varphi(\kappa)\ \eqdef\ \langle  Q_\kappa(R)\rangle\ -\  Q_\kappa\left(\langle R\rangle\right)\ -\ \half \left[ \langle  Q_\kappa '(R)\rangle\ -\  Q_\kappa '\left(\langle R\rangle\right) \right]^2
\end{equation*}
is non-negative. We have $\partial_\kappa Q_\kappa(\xi)=(\xi-\kappa)^+=\xi\vee\kappa-\kappa$ and $ Q_\kappa'(\xi)=\xi - \xi\vee\kappa$, so
\begin{equation}
\varphi'(\kappa)\ =\, \left(\langle R\vee\kappa\rangle\ -\ \langle R\rangle\vee\kappa\right)\left[1 -\left( \langle \mathds{1}_{\{R\geqslant \kappa\}}\rangle\ -\ \mathds{1}_{\{\langle R\rangle\geqslant \kappa\}}\right) \right].
\end{equation}
Since $\xi\mapsto\xi\vee\kappa$ is convex, $\langle R\vee\kappa\rangle - \langle R\rangle\vee\kappa$ is non-negative by Jensen's inequality, while
\begin{equation}
1 -\left( \langle \mathds{1}_{\{R\geqslant \kappa\}}\rangle\ -\ \mathds{1}_{\{\langle R\rangle\geqslant \kappa\}}\right)
\geqslant 0.
\end{equation}
We have then $\varphi'(\kappa)\geqslant 0$, so it will be sufficient to study the case $\kappa=0$. Set $\xi^-=-\min(\xi,0)$. We have $ Q_0(\xi)=(\xi^-)^2/2$, $ Q_0'(\xi)=-\xi^-$, so
\begin{equation}
2\, \varphi(0)\ =\ \langle (R^-)^2\rangle-(\langle R\rangle^-)^2-\left(\langle R^-\rangle-\langle R\rangle^-\right)^2.
\end{equation}
If $\langle R\rangle\geqslant 0$, then
\begin{equation}
2\, \varphi(0)\ =\ \langle (R^-)^2\rangle-\left(\langle R^-\rangle\right)^2
\end{equation}
is non-negative, by Jensen's inequality. If $\langle R\rangle<0$, then, using Jensen's inequality again, we get
\begin{equation*}
2\, \varphi(0)\, =\, \langle (R^-)^2\rangle\, -\, (\langle R\rangle)^2\, -\left(\langle R^-\rangle\, +\, \langle R\rangle\right)^2
\geqslant\,  -\, 2\, \langle R\rangle\, (\langle R^-\rangle\, +\, \langle R\rangle)\,
=\, -2\, \langle R\rangle\, \langle R^+\rangle\, \geqslant\, 0,
\end{equation*}
and this achieves the proof of \eqref{TS}. The inequality \eqref{CompareDelta} follows from Jensen's inequality since
\begin{equation}\label{QVSQ}
\half\, \xi^2\, -Q_\kappa(\xi)\ =\ \half \left((\xi-\kappa)^+\right)^2
\end{equation}
is a convex function of $\xi$. By \eqref{QVSQ} and \eqref{CompareDelta}, we also have
\begin{equation*}
\tilde{\E} \left\| \Delta_\kappa-\Delta \right\|_{L^1([0,T] \times \T)}\, =\, 
\half\,  \tilde{\E} \int_0^T\int_\T\int_\R \left((\xi-\kappa)^+\right)^2 \ud\tilde{\mu}_{t,x}\, \ud x\, \ud t
-\half\, \tilde{\E} \int_0^T\int_\T \left((\tilde{R}-\kappa)^+\right)^2  \ud x\, \ud t,
\end{equation*}
and \eqref{CVDelta} follows by dominated convergence.
\end{proof}

Here is the central proposition of this section.

\begin{proposition}[Evolution equation for the defect measure]\label{prop:EvolDefect} Let $\Delta$ be defined by \eqref{CompareDelta}. Let $t_0 \in (0,T)$. Then $\tilde{\Pro}$-almost surely, we have
\begin{gather}\label{DeltaR}
\left[ \sqrt{\Delta /c(\tilde{u})} \right]_t\, +\, \left[ \sqrt{c(\tilde{u})\, \Delta} \right]_x\, \leqslant\ 0 \qquad \text{in}\ \mathcal{D}'((t_0,T) \times \T).
\end{gather}
\end{proposition}

\begin{proof}[Proof of Proposition~\ref{prop:EvolDefect}] The proof breaks into several steps. \medskip

\textbf{Step 1. (Truncation)} We derive a first equation for the version $\Delta_\kappa$ of $\Delta$ with truncation defined in \eqref{CompareDelta}. We use the equations \eqref{RtildeLIM} and \eqref{RtildeLIMRenormalize} with $h=Q_\kappa$ (\textit{cf.} Remark~\ref{rk:growthh}). Subtracting the second equation from the first one gives  us
\begin{gather}\label{eq:deltakappa1}
\left[ \Delta_\kappa \right]_t\, +\, \left[ c(\tilde{u})\, \Delta_\kappa \right]_x \, =\  \half\, {q} \left[ \langle Q_{\kappa }''(R)\rangle\, -\, Q_{\kappa }''(\langle R\rangle) \right] \, +\, \partial_t M_{\kappa}\\
+\, \tilde{c}'(\tilde{u}) \left[ \langle\Psi_\kappa(R)\rangle\, -\, \Psi_\kappa(\langle R\rangle)\right]
 +\, \tilde{c}'(\tilde{u})\, 
 Q_{\kappa }'(\langle R\rangle) \left(\langle R^2\rangle\,  -\,  \langle R\rangle^2\,  \right)\label{eq:deltakappa1-2}\\
 +\ \tilde{c}'(\tilde{u}) \left[ 2\, \langle S\rangle\, \Delta_\kappa\, +\,
\langle S^2\rangle \left[ Q_{\kappa}'(\langle R\rangle)\, -\,  \langle Q'_\kappa({R})\rangle\right] \right]\label{eq:deltakappa1-3}
\end{gather}
in $\mathcal{D}'((t_0,T) \times \T)$, where
\begin{equation}
\Psi_\kappa(\xi)\ \eqdef\ \xi^2\, Q'_\kappa(\xi)\, -\, 2\, \xi\, Q_\kappa(\xi),
\end{equation}
and
\begin{equation}\label{defMkappa}
M_\kappa(t)\ \eqdef\ \int_{t_0}^t \left[\langle Q_\kappa'(R)\rangle-Q_\kappa'(\langle R\rangle)\right]\Phi\, \ud\tilde{W}(s).
\end{equation}

\textbf{Step 2. (Preparation for renormalization)} Different obstructions prevent us from pas\-sing to the limit $[\kappa\to\infty]$ on the parameter of truncation $\kappa$ in \eqref{eq:deltakappa1}-\eqref{eq:deltakappa1-2}-\eqref{eq:deltakappa1-3}. Indeed, we can make no sense of a term $\tilde{c}'(\tilde{u}) L(\langle R\rangle,\langle S\rangle)$ (or $\tilde{c}'(\tilde{u}) \langle L(R,S)\rangle$) unless $L$ satisfies the growth condition $|L(\xi,\eta)|\leqslant C(L)(1+|\xi|^{3-\gamma}+|\eta|^{3-\gamma})$ with $\gamma>0$ for positive $\xi,\eta$. It is therefore not possible to define the limit as $\kappa\to \infty$ of the three last terms of \eqref{eq:deltakappa1}-\eqref{eq:deltakappa1-2}-\eqref{eq:deltakappa1-3}. Renormalization will help to solve this problem. Indeed, the first term in \eqref{eq:deltakappa1-3} has a factor $\Delta_\kappa$, the second term in \eqref{eq:deltakappa1-3} is controlled by $\Delta_\kappa^{1/2}$ as a consequence of \eqref{TS}. However, the second term in \eqref{eq:deltakappa1-2} has not the desired form, 
since this is $2\tilde{c}'(\tilde{u}) Q_{\kappa }'(\langle R\rangle)\Delta$ and not $2\tilde{c}'(\tilde{u}) Q_{\kappa }'(\langle R\rangle)\Delta_\kappa$. We exploit the fact that $Q_{\kappa }'(\xi)\leqslant \kappa$ and the inequality \eqref{CompareDelta} to write 
\begin{equation}\label{eq:deltakappa1-22}
\eqref{eq:deltakappa1-2}\ \leqslant\ 
\tilde{c}'(\tilde{u}) \left[ \langle\Phi_\kappa(R)\rangle\, -\, \Phi_\kappa(\langle R\rangle)\right]
 +\, 2\, \tilde{c}'(\tilde{u})\, Q_{\kappa }'(\langle R\rangle)\, \Delta_\kappa,
\end{equation}
where
\begin{equation}
\Phi_\kappa(\xi)\ =\ \Psi_\kappa(\xi)\, +\, \kappa\, (\xi^2 - 2\, Q_\kappa(\xi))\ =\ -\kappa^2(\xi-\kappa)^+.
\end{equation}
Since $\Phi_\kappa$ is concave, we obtain
\begin{equation}
\eqref{eq:deltakappa1-2}\ \leqslant\ 2\, \tilde{c}'(\tilde{u})\, Q_{\kappa }'(\langle R\rangle)\, \Delta_\kappa,
\end{equation}
and thus
\begin{gather}\label{eq:deltakappa2}
\left[ \Delta_\kappa \right]_t\, +\, \left[ c(\tilde{u})\, \Delta_\kappa \right]_x \, \leqslant\  \half\, {q} \left[ \langle Q_{\kappa }''(R)\rangle\, -\, Q_{\kappa }''(\langle R\rangle) \right] \, +\, \partial_t M_{\kappa}\\
+\, \tilde{c}'(\tilde{u})\left[2 \left(Q_{\kappa }'(\langle R\rangle)\, +\, \langle S\rangle\right)\, \Delta_\kappa+  \langle S^2\rangle \left[ Q_{\kappa}'(\langle R\rangle)\, -\,  \langle Q'_\kappa({R})\rangle\right]
\right].\label{eq:deltakappa2-2}
\end{gather}

\textbf{Step 3. (Renormalization)} Let $H(c,\Delta)$ be a smooth function of its argument. We can do the formal computations
\begin{gather*}
\ud H(c(\tilde{u}),\Delta_\kappa)\, +\, \partial_x \left[c(\tilde{u})\, H(c(\tilde{u}),\Delta_\kappa)\right] \ud t\,
=\,
{\textstyle \frac{\partial H}{\partial c}}(c(\tilde{u}),\Delta_\kappa)
\left[\partial_t c(\tilde{u})\, +\, c(\tilde{u})\, \partial_x c(\tilde{u})\right] \ud t\\
+\, {\textstyle \frac{\partial H}{\partial \Delta}}(c(\tilde{u}),\Delta_\kappa)
\left[\ud \Delta_\kappa\, +\, \partial_x\left(c(\tilde{u})\, \Delta_\kappa\right) \right]\ud t\, 
+\, \half {\textstyle \frac{\partial^2 H}{\partial \Delta^2}}(c(\tilde{u}),\Delta_\kappa)\, \ud \langle M_\kappa,M_\kappa\rangle(t)\\
+\, \partial_x c(\tilde{u})\left(H-\Delta {\textstyle\frac{\partial H}{\partial \Delta}}\right)(c(\tilde{u}),\Delta_\kappa)\, \ud t,
\end{gather*}
where $\langle M_\kappa,M_\kappa\rangle_t$ is the quadratic variation of the martingale $M_\kappa$ defined in \eqref{defMkappa}, \textit{i.e.}
\begin{equation}\label{defMkappaQuad}
\langle M_\kappa,M_\kappa\rangle(t)\ \eqdef\
q\, \int_{t_0}^t \left|\langle Q_\kappa'(R)\rangle-Q_\kappa'(\langle R\rangle)\right|^2 \ud s.
\end{equation}
Using the formulas 
\begin{equation*}
\partial_t c(\tilde{u})\, =\, c'(\tilde{u})\, \tilde{u}_t\, =\, 2\, c(\tilde{u})\, \tilde{c}'(\tilde{u})\, (\langle R\rangle+\langle S\rangle),\quad
\partial_x c(\tilde{u})\, =\, c'(\tilde{u})\, \tilde{u}_x\, =\, 2\, \tilde{c}'(\tilde{u})\, (\langle S\rangle-\langle R\rangle),
\end{equation*}
we obtain
\begin{align*}
\ud H_\kappa+\partial_x \left(c(\tilde{u}) H_\kappa\right) \ud t\,
&=\,
4\, c(\tilde{u})\, \tilde{c}'(\tilde{u})\, {\textstyle\frac{\partial H_\kappa}{\partial c}}
\langle S\rangle\, \ud t\,
+\, {\textstyle \frac{\partial H_\kappa}{\partial \Delta}}
\left[\ud \Delta_\kappa\, +\, \partial_x\left(c(\tilde{u})\, \Delta_\kappa\right) \ud t\right]\\
& \quad +\, \half {\textstyle \frac{\partial^2 H_\kappa}{\partial \Delta^2}}\, \ud \langle M_\kappa,M_\kappa\rangle(t)\,
+\, 2\, \tilde{c}'(\tilde{u})\left(H_\kappa-\Delta_\kappa {\textstyle \frac{\partial H_\kappa}{\partial \Delta}}\right) (\langle S\rangle-\langle R\rangle)\, \ud t,
\end{align*}
where, for a function $F=F(c,\Delta)$, $F_\kappa$ stands for $F(c(\tilde{u}),\Delta_\kappa)$. By \eqref{eq:deltakappa2}-\eqref{eq:deltakappa2-2}, we then deduce the following inequality in $\mathcal{D}'((t_0,T)\times\T)$
\begin{gather}\nonumber
\partial_t H_\kappa\, +\, \partial_x \left(c(\tilde{u}) H_\kappa\right)
\leqslant\,
4\, c(\tilde{u})\, \tilde{c}'(\tilde{u}) {\textstyle \frac{\partial H_\kappa}{\partial c}}
\langle S\rangle\, 
+ \, \half\, {q}\, {\textstyle  \frac{\partial H_\kappa}{\partial \Delta}} \left[ \langle Q_{\kappa }''(R)\rangle\, -\, Q_{\kappa }''(\langle R\rangle) \right]
+\, \partial_t N_\kappa
\\ \nonumber
+\, {\textstyle \frac{\partial H_\kappa}{\partial \Delta}}
\tilde{c}'(\tilde{u})\left[2\, \left(Q_{\kappa }'(\langle R\rangle)\,  +\, \langle S\rangle\right)\, \Delta_\kappa+  \langle S^2\rangle \, \left[ Q_{\kappa}'(\langle R\rangle)\, -\,  \langle Q'_\kappa({R})\rangle\right]
\right]\\ \label{evolutionHkappa}
+\, \half {\textstyle  \frac{\partial^2 H_\kappa}{\partial \Delta^2}} \partial_t\langle M_\kappa,M_\kappa\rangle(t)\,
+\, 2\, \tilde{c}'(\tilde{u})\left(H_\kappa-\Delta_\kappa {\textstyle \frac{\partial H_\kappa}{\partial \Delta}}\right) (\langle S\rangle-\langle R\rangle),
\end{gather}
where
\begin{equation}\label{defNkappa}
N_\kappa(t)\ \eqdef\ \int_{t_0}^t {\textstyle \frac{\partial H}{\partial \Delta}}\left[\langle Q_\kappa'(R)\rangle-Q_\kappa'(\langle R\rangle)\right]\Phi\, \ud\tilde{W}(s).
\end{equation}
The rigorous derivation of \eqref{evolutionHkappa} uses a preliminary step of convolution and DiPerna-Lions' commutator lemma and the chain-rule for the $H^1$-function $c(\tilde{u})$. This is very similar to the proof of Proposition~\ref{prop:Renormalize}, so we will not give the details here.\medskip

\textbf{Step 4. (Limit $[\kappa\to\infty]$)} Assume that the function $H$ in \eqref{evolutionHkappa} satisfies the following bounds: there is a constant $C_H\geqslant 0$ such that, for all $c\in [c_1,c_2]$, for all $\Delta,\Delta'\geqslant 0$,
\begin{equation}\label{BoundH1}
\left|\Delta  {\textstyle \frac{\partial H}{\partial\Delta}(c,\Delta)}\right|\, \leqslant\ C_H\left(1+\Delta^{1/2}\right),\quad
\left(1+\Delta^{1/2}\right)\left|{\textstyle\frac{\partial H}{\partial\Delta}}(c,\Delta)\right|\, \leqslant\  C_H,
\end{equation}
and
\begin{equation}\label{LipH}
\left|F(c,\Delta)\, -\, F(c,\Delta')\right|^2\, \leqslant\ C_H\left|\Delta\, -\, \Delta'\right|,\quad F\in\left\{H, {\textstyle \frac{\partial H}{\partial c}},\Delta {\textstyle\frac{\partial H}{\partial\Delta}}\right\}.
\end{equation} 
Suppose additionally that, for all $c \in [c_1,c_2]$, the map $\Delta \mapsto H(c,\Delta)$ is concave. We consider then the limit of \eqref{evolutionHkappa} when $\kappa\to\infty$ along a given sequence. By taking a subsequence if necessary, we can assume that the convergences 
\begin{equation}\label{CVprimeQ}
\langle Q_\kappa '(R)\rangle\, -\, Q_\kappa '\left(\langle R\rangle\right)\ \to\ 0,\quad \langle Q_\kappa '(R)\rangle\, -\, \langle R\rangle\ \to\ 0
\end{equation}
are satisfied $\tilde{\Pro}$-a.s. in $L^2((0,T)\times\R)$ as in \eqref{L2conv}, but also a.s. in $(\omega,t,x)$. 
From \eqref{LipH}, we can then deduce the following convergence results in $\mathcal{D}'((t_0,T)\times\T)$:
\begin{align}\label{CVH1}
\partial_t H_\kappa\, +\, \partial_x \left(c(\tilde{u}) H_\kappa\right)\, \
&\to  \ \
\partial_t H\, +\, \partial_x \left(c(\tilde{u}) H\right),\\
\label{CVH2}
4\, c(\tilde{u})\, \tilde{c}'(\tilde{u}) {\textstyle \frac{\partial H_\kappa}{\partial c}}
\langle S\rangle\ \
&\to \ \
 4\, c(\tilde{u})\, \tilde{c}'(\tilde{u}) {\textstyle\frac{\partial H}{\partial c}}
\langle S\rangle ,\\
\label{CVH3}
2\, \tilde{c}'(\tilde{u})\left(H_\kappa-\Delta_\kappa {\textstyle\frac{\partial H_\kappa}{\partial \Delta}}\right) (\langle S\rangle-\langle R\rangle)\ \
&\to\ \
2\, \tilde{c}'(\tilde{u})\left(H-\Delta {\textstyle\frac{\partial H}{\partial \Delta}}\right) (\langle S\rangle-\langle R\rangle),\\
\label{CVH4}
{\textstyle\frac{\partial H_\kappa}{\partial \Delta}} \tilde{c}'(\tilde{u})\, \langle S\rangle\, \Delta_\kappa \ \
&\to \ \
{\textstyle\frac{\partial H}{\partial \Delta}} \tilde{c}'(\tilde{u})\, \langle S\rangle\, \Delta,
\end{align}
%
where, for a function $F=F(c,\Delta)$, $F$ stands for $F(c(\tilde{u})(t,x),\Delta(t,x))$. Next, the bounds \eqref{TS}, \eqref{BoundH1} and the convergence \eqref{CVprimeQ} show that
\begin{equation}\label{CVH5}
\partial_t N_\kappa,\quad {\textstyle \frac{\partial H_\kappa}{\partial \Delta}}\left[ \langle Q_{\kappa }''(R)\rangle\, -\, Q_{\kappa }''(\langle R\rangle) \right], \quad
\tilde{c}'(\tilde{u}) {\textstyle  \frac{\partial H}{\partial \Delta}}\, \langle S^2\rangle  \left[ Q_{\kappa}'(\langle R\rangle)\, -\,  \langle Q'_\kappa({R})\rangle\right]
\end{equation}
all converge to $0$ in $\mathcal{D}'((t_0,T)\times\T)$. Let us give few details about the convergence of the last term in \eqref{CVH5}: by \eqref{TS} and \eqref{BoundH1}, we have the bound
\begin{equation}\label{CVH5b}
\left|\tilde{c}'(\tilde{u})\,  {\textstyle \frac{\partial H}{\partial \Delta}}\, \langle S^2\rangle  \left[ Q_{\kappa}'(\langle R\rangle)\, -\,  \langle Q'_\kappa({R})\rangle\right]\right|\,
\leqslant\ C_H\, \tilde{c}'(\tilde{u})\, \langle S^2\rangle, 
\end{equation}
while the term converges to $0$ $\tilde{\Pro}$-a.s., a.e. in $(t,x)$. The $\tilde{\Pro}$-a.s. convergence to $0$ follows by dominated convergence. We also have
\begin{equation}\label{ConcaveH}
\half {\textstyle \frac{\partial^2 H_\kappa}{\partial \Delta^2}}\partial_t\langle M_\kappa,M_\kappa\rangle(t)\ \leqslant\ 0.
\end{equation}
There remains to examine the term
\begin{equation}
{\textstyle \frac{\partial H_\kappa}{\partial \Delta}}
\tilde{c}'(\tilde{u})\, Q_{\kappa }'(\langle R\rangle)\, \Delta_\kappa,
\end{equation}
which can be split as
\begin{equation}\label{CVH8}
F\, \tilde{c}'(\tilde{u})\, \langle R\rangle\ +\, \left[F_\kappa\, -\, F\right]\tilde{c}'(\tilde{u})\, \langle R\rangle
\, -\, F_\kappa\, \tilde{c}'(\tilde{u})\, \left[\langle R\rangle- Q_{\kappa }'(\langle R\rangle)\right],
\end{equation}
with $F= \Delta\frac{\partial H}{\partial\Delta}(c,\Delta)$ and $F_\kappa = \Delta_\kappa \frac{\partial H}{\partial\Delta}(c,\Delta_\kappa)$. The same arguments as above then show that the two last terms in \eqref{CVH8} converge to $0$ and we finally deduce from \eqref{evolutionHkappa} the inequality
\begin{multline*}
\partial_t H\, +\, \partial_x \left(c(\tilde{u}) H\right)\,
\leqslant\
2\, \tilde{c}'(\tilde{u})\left\{
2\, c(\tilde{u}) {\textstyle \frac{\partial H}{\partial c}}
\langle S\rangle\,
+\, \Delta {\textstyle \frac{\partial H}{\partial \Delta}} \left(\langle R\rangle\,  +\, \langle S\rangle\right)
+\, \left(H-\Delta {\textstyle \frac{\partial H}{\partial \Delta}}\right) (\langle S\rangle-\langle R\rangle)
\right\},
\end{multline*}
that is to say
\begin{equation}\label{evolutionH}
\partial_t H\, +\, \partial_x \left(c(\tilde{u}) H\right)\,
\leqslant\
2\, \tilde{c}'(\tilde{u})\left\{
\left[2\, c(\tilde{u})\, {\textstyle \frac{\partial H}{\partial c}}\, +\, H\right]
\langle S\rangle\,
+\left[2\, \Delta\, {\textstyle \frac{\partial H}{\partial \Delta}} -H\right]\langle R\rangle
\right\}.
\end{equation}

\textbf{Step 5. (Conclusion)} 
We take now
\begin{equation}
H(c,\Delta)\ =\ H_\delta(c,\Delta)\ \eqdef\ \sqrt{\Delta/c+\delta},
\end{equation}
where $\delta$ is a positive parameter. It is easy to check that the map  $\Delta \mapsto H(c,\Delta)$ is concave and that \eqref{BoundH1} and \eqref{LipH} are satisfied. In \eqref{evolutionH}, we obtain
\begin{equation}\label{evolutionHdelta}
\partial_t H_\delta+\partial_x \left(c(\tilde{u}) H_\delta\right)\,
\leqslant\
2\, \delta\, \tilde{c}'(\tilde{u})\,\frac{\langle S\rangle-\langle R\rangle}{\sqrt{\Delta/c+\delta}}\
\leqslant\ 2\, \delta^\frac12\, \tilde{c}'(\tilde{u})\left|\langle S\rangle-\langle R\rangle\right|.
\end{equation}
Taking the limit $[\delta\to 0]$ in \eqref{evolutionHdelta} yields \eqref{DeltaR}.
\end{proof}

To fully exploit \eqref{DeltaR}, we need to show that $\Delta(t_0)\to 0$ when $t_0\to 0$. This is the content of the following proposition.

\begin{proposition}[No initial boundary layer]\label{prop:CVDelta0} Recall that $\Delta$ is defined by \eqref{CompareDelta}. Set also
\begin{equation}\label{Deltacheck}
\check{\Delta}\ \eqdef\ \half \left(\langle S^2\rangle\ -\ \langle S\rangle^2\right).
\end{equation} 
Then, as $t \to 0$, we have, $\tilde{\Pro}$-almost surely, the convergences
\begin{equation}\label{tauto0a}
\langle R\rangle(t)\ \to\ R_0,\qquad \langle S\rangle(t)\ \to\ S_0,\qquad \mbox{in }L^2(\T),
\end{equation}
and
\begin{equation}\label{tauto0aa}
\Delta(t),\, \check{\Delta}(t)\ \to\ 0,\qquad \mbox{in }L^1(\T).
\end{equation}
\end{proposition} 

\begin{proof}[Proof of Proposition~\ref{prop:CVDelta0}]  Let $\bar{Q}_\kappa$ denote the truncation of $\xi\mapsto\xi^2/2$ from above and from below defined by
\begin{equation}
\bar{Q}_\kappa(\xi)\ =\ \half \left[\xi^2-\left((\xi-\kappa)^+\right)^2- \left((\xi+\kappa)^-\right)^2\right].
\end{equation}
 We apply \eqref{CVut} with $f_1=f_2=\bar{Q}_\kappa$ and $\varphi\equiv 1$ to obtain $\tilde{\Pro}$-almost surely, for $t > 0$, 
\begin{equation}\label{weakbarQeps}
\int_\T  \left[ \langle\bar{Q}_\kappa({R})\rangle\, +\, \langle\bar{Q}_\kappa({S})\rangle \right] (t,x)\, \ud x\ =\ \lim_{\varepsilon\to 0}
\int_\T \left[\bar{Q}_\kappa(\tilde{R}^\varepsilon)\, +\, \bar{Q}_\kappa(\tilde{S}^\varepsilon) \right](t,x)\, \ud x.
\end{equation}
Since $\bar{Q}_\kappa(\xi)\leqslant \xi^2/2$, the estimate 
\begin{equation}\label{weakbarQt}
\int_\T  \left[ \langle\bar{Q}_\kappa({R})\rangle\ +\ \langle\bar{Q}_\kappa({S})\rangle \right] (t,x)\, \ud x\  \leqslant\ \half\,
 \liminf_{\varepsilon \to 0}\, \|(\tilde{R}^\varepsilon, \tilde{S}^\varepsilon)(t)\|_{L^2(\mathds{T})}^2
\end{equation}
follows. We then use the energy inequality \eqref{eq:TotalEnergyep} to get a bound on the right-hand side of \eqref{weakbarQt}. Indeed, we can pass to the limit in the energy inequality \eqref{eq:TotalEnergyep} since the explicit expression of the martingale $\mathcal{M}^\eps$ is given by \eqref{Meps-energy} and 
\begin{equation*}
2\, \int_\T \int_{0}^\cdot  \left( \tilde{R}^\varepsilon\, +\, \tilde{S}^\varepsilon \right) \Phi^\varepsilon\, \ud \tilde{W}^\varepsilon\, \ud x \quad \to \quad M_\mathrm{e}\ \eqdef\ 2\, \int_\T \int_{0}^\cdot   \left( \langle R\rangle\, +\,\langle S\rangle \right) \Phi\,\ud \tilde{W}\, \ud x \quad \text{in}\ C([0,T]).
\end{equation*}
Using also \eqref{Reduction0} to treat the initial terms, we obtain
\begin{equation}\label{barQtOK}
\int_\T  \left[ \langle\bar{Q}_\kappa({R})\rangle\ +\ \langle\bar{Q}_\kappa({S})\rangle \right] (t,x)\, \ud x\ \  \leqslant\ \half\, \|(R_0,S_0)\|_{L^2(\mathds{T})}^2\ +\  \|q\|_{L^1(\T)}\, t\ +\ \half\, M_\mathrm{e}(t).
\end{equation}
Taking $\kappa \to \infty$ in \eqref{barQtOK} then gives us
\begin{equation}\label{barQtOKOK}
\int_\T  \left[ \langle R^2\rangle\ +\ \langle S^2\rangle \right] (t,x)\, \ud x\  \leqslant\  \|(R_0,S_0)\|_{L^2(\mathds{T})}^2\ +\ 2\, \|q\|_{L^1(\T)}\, t\ +\  M_\mathrm{e}(t),
\end{equation}
and so
\begin{equation}\label{barQtOKOKOK}
 \|(\langle R\rangle,\langle S\rangle)\|_{L^2(\mathds{T})}^2(t)\ +\ 2 \int_\T \left(\Delta(t)\, +\, \check{\Delta}(t)\right) \ud x\
   \leqslant\ \|(R_0,S_0)\|_{L^2(\mathds{T})}^2\ +\ 2\, q_0\, t\ +\ M_\mathrm{e}(t).
\end{equation}
The convergence of $(\tilde{R}^\varepsilon, \tilde{S}^\varepsilon)$ to $(\langle R\rangle, \langle S\rangle)$ in $C([0,T],L^2_w(\T))$ has the consequence that 
\begin{equation}\label{InitalOK}
\|(R_0,S_0)\|_{L^2(\T)}^2\ \leqslant\ \liminf_{t \to 0}  \|(\langle R\rangle, \langle S\rangle)\|_{L^2(\T)}^2.
\end{equation}
From \eqref{barQtOKOKOK} and \eqref{InitalOK}, we deduce \eqref{tauto0aa} and the convergence of the norms
\begin{equation}
\lim_{t \to 0}  \|(\langle R\rangle, \langle S\rangle)\|_{L^2(\T)}^2\ =\ \|(R_0,S_0)\|_{L^2(\T)}^2,
\end{equation}
which, combined with the weak convergence, gives \eqref{tauto0a}.
\end{proof}

We can now conclude the proof of reduction (to Dirac masses) of the Young measures.

\begin{proposition}[Reduction of the Young measures]\label{prop:RedYoungOK}
For almost all $(\omega, t,x) \in \tilde{\Omega} \times [0,T] \times \T$, we have 
\begin{equation}\label{reductionmuOK}
\tilde{\mu}_{t,x}\ =\ \delta_{\left(\langle R\rangle(t,x),\langle S\rangle(t,x)\right)}.
\end{equation}
\end{proposition} 

\begin{proof}[Proof of Proposition~\ref{prop:RedYoungOK}] Using \eqref{DeltaR}, we have: almost-surely
\begin{equation*}
\partial_t G\ \leqslant\ 0\ \mbox{ in }\ \mathcal{D}'(t_0,T),\qquad G(t)\ \eqdef\ \int_\T \sqrt{\Delta /c(\tilde{u})}(t)\, \ud x.
\end{equation*} 
It follows then from \eqref{tauto0aa} that, almost-surely, $G(t)=0$ at all Lebesgue points $t$ of $G$, and thus $\Delta(x,t)=0$ for a.e. $(x,t)$, which is equivalent to $\nu^1_{t,x}=\delta_{\langle R\rangle(t,x)}$ $\tilde{\Pro}$-a.s., for a.e. $(t,x)$. We can also establish in a completely similar manner the identity  $\nu^2_{t,x}=\delta_{\langle S\rangle(t,x)}$ $\tilde{\Pro}$-a.s., for a.e. $(t,x)$. Then \eqref{reductionmuOK} follows from the decomposition \eqref{Product}.
\end{proof}

\subsection{Existence of weak martingale solutions}\label{subsec:proofMain}

\begin{proof}[Proof of Theorem~\ref{thm:global-existR2SS}] Once the reduction \eqref{reductionmuOK} of the Young measure is established, the equations \eqref{RtildeLIMID}-\eqref{StildeLIMID} take the desired form
\begin{align}\nonumber
\int_\T \langle R\rangle(t,x)\, \varphi(x)\, \ud x \ &=\ \int_\T  R_0 (x)\, \varphi(x)\, \ud x
\ +\ \int_{0}^t\int_\T c(\tilde{u})\, \langle R\rangle\, \varphi_x\, \ud x\, \ud s\\ \label{weakR}  
& \quad -\ \int_{0}^t\int_\T\tilde{c}'(\tilde{u}) \left[\langle R\rangle\, -\, \langle S \rangle \right]^2 \varphi\,  \ud x\, \ud s\, 
+\, \int_\T \int_{0}^t \varphi\, \Phi\, \ud \tilde{W}(s)\, \ud x,\\
\nonumber
\int_\T \langle S\rangle(t,x)\, \varphi(x)\, \ud x \ &=\ \int_\T  S_0(x)\, \varphi(x)\, \ud x
\ -\ \int_{0}^t\int_\T c(\tilde{u})\, \langle S\rangle\, \varphi_x\, \ud x\, \ud s\\ \label{weakS}
& \quad -\ \int_{0}^t\int_\T\tilde{c}'(\tilde{u}) \left[\langle R\rangle\, -\, \langle S \rangle \right]^2 \varphi\, \ud x\, \ud s\, 
+\, \int_\T \int_{0}^t  \varphi\, \Phi\, \ud \tilde{W}(s)\, \ud x.
\end{align}
where $R_0$ and $S_0$ are defined by \eqref{IC}. Summing up \eqref{weakR} and \eqref{weakS} and using the identities \eqref{tildeRSuep} with the chain-rule for $H^1$-functions we obtain \eqref{weaku}. We have already establish that $\tilde{R},\tilde{S}\in L^2(\Omega;L^\infty([0,T]; L^2(\T)))$ and $\tilde{R},\tilde{S}\in C([0,T]; L^2(\T)-\mathrm{weak}))$ $\tilde{\Pro}$-a.s. 
Now, we can use \eqref{reductionmuOK} and follow the same proof as in Proposition \ref{prop:CVDelta0} for almost all \(t_0\) to obtain the right continuity \eqref{rightcontinuity} and the energy dissipation \eqref{energydissipation}. Note that the energy dissipation \eqref{energydissipation} follows from \eqref{barQtOKOK} when applying the proof for \(t_0\) instead of 0.
Taking into account \eqref{tildeStochasticBasis}, we can conclude to the existence of a weak martingale solution to \eqref{SVWE1}. Finally, the estimates \eqref{L3-estimates}-\eqref{Entropy} follow from \eqref{L3-}-\eqref{Oleinik}.
\end{proof}

%
%
%
%
%
%
%
%

\paragraph{Acknowledgement.} This work was supported by the LABEX MILYON (ANR-10-LABX-0070) of Universit\'e de Lyon, within the program ”Investissements d’Avenir” (ANR-11-IDEX-0007) operated by the French National Research Agency (ANR). It was also supported by the Unit\'e de Math\'ematiques Pure et Appliqu\'ees, UMPA (CNRS and ENS de Lyon).

\def\cprime{$'$}

\end{document}